\documentclass[11pt,reqno]{amsart}
\usepackage{amssymb}
\usepackage{graphicx}
\usepackage{cite}
\usepackage[left=2.5cm,right=2.5cm,top=3cm,bottom=3cm]{geometry}
\usepackage{color}
\usepackage{tikz-cd}
\usepackage{mathtools}
\usepackage[hidelinks]{hyperref}
\usepackage{amsmath}
\usepackage{amsthm}
\usepackage{tikz}
\usepackage{amscd}
\usepackage{latexsym}
\usepackage{epsfig}
\usepackage{graphicx}
\usepackage{psfrag}
\usepackage{caption}
\usepackage{rotating}
\usepackage{mathtools}
\usepackage{xypic}
\usepackage{enumitem}
\usepackage{bigints} 
\usepackage{longtable}

\allowdisplaybreaks[3]

\newcounter{tmp}

\makeatletter
\@namedef{subjclassname@1991}{2020 Mathematics Subject Classification}
\makeatother

\newtheorem{theorem}{Theorem}[subsection]

\theoremstyle{definition}

\newtheorem*{question*}{Motivating Question}

\renewcommand{\k}{\mathfrak{k}}
\renewcommand{\t}{\mathfrak{t}}
\newcommand{\G}{\mathcal{G}}
\renewcommand{\H}{\mathcal{H}}
\newcommand{\too}{\longrightarrow}
\newcommand{\mtoo}{\longmapsto}
\newcommand{\mto}{\mapsto}
\newcommand{\tto}{\rightrightarrows}
\newcommand{\sss}{\mathsf{s}}
\newcommand{\ttt}{\mathsf{t}}

\newcommand{\dfn}[1]{\textit{\textbf{#1}}}
\newcommand{\sll}[1]{\mkern-4mu\mathbin{/\mkern-5mu/}_{\mkern-4mu{#1}}}

\newcommand{\s}{\subseteq}
\renewcommand{\c}{\mathfrak{c}}
\renewcommand{\sl}{\mathfrak{sl}}
\newcommand{\rk}{\operatorname{rk}}
\newcommand{\im}{\operatorname{im}}
\renewcommand{\O}{\mathcal{O}}

\newcommand{\I}{\mathcal{I}}
\newcommand{\Lie}{\operatorname{Lie}}

\newcommand{\R}{\mathbb{R}}
\newcommand{\C}{\mathbb{C}}

\newcommand{\g}{\mathfrak{g}}
\newcommand{\h}{\mathfrak{h}}

\newcommand{\SLn}{\mathrm{SL}}

\newcommand{\ad}{\mathrm{ad}}
\newcommand{\Ad}{\mathrm{Ad}}

\numberwithin{equation}{section}

\theoremstyle{theorem}
\newtheorem{thm}{Theorem}[section]

\newtheorem{lem}[thm]{Lemma}
\newtheorem{cor}[thm]{Corollary}
\newtheorem{prop}[thm]{Proposition}
\theoremstyle{definition}
\newtheorem{rem}[thm]{Remark}

\newtheorem{ex}[thm]{Example}
\newtheorem{defn}[thm]{Definition}
\newtheorem{conj}[thm]{Conjecture}

\title[Symplectic reduction along a submanifold]{Symplectic reduction along a submanifold}

\author[Peter Crooks]{Peter Crooks}
\author[Maxence Mayrand]{Maxence Mayrand}
\address[Peter Crooks]{Department of Mathematics \\ Northeastern University \\ 360 Huntington Avenue \\ Boston, MA 02115, USA}
\email{p.crooks@northeastern.edu}
\address[Maxence Mayrand]{Department of Mathematics \\ University of Toronto \\ 40 St.\ George St.\ \\ Toronto, Ontario \\ M5S 2E4 \\ Canada}
\email{mayrand@math.toronto.edu}
\subjclass{53D20 (primary); 14J42, 53D17 (secondary)}
\keywords{Moore--Tachikawa variety, symplectic reduction, symplecic groupoid}

\begin{document}

\begin{abstract}
We introduce the process of symplectic reduction along a submanifold as a uniform approach to taking quotients in symplectic geometry. This construction holds in the categories of smooth manifolds, complex analytic spaces, and complex algebraic varieties, and has an interpretation in terms of derived stacks in shifted symplectic geometry. It also encompasses Marsden--Weinstein--Meyer reduction, Mikami--Weinstein reduction, the pre-images of Poisson transversals under moment maps, symplectic cutting, symplectic implosion, and the Ginzburg--Kazhdan construction of Moore--Tachikawa varieties in TQFT. A key feature of our construction is a concrete and systematic association of a Hamiltonian $G$-space $\mathfrak{M}_{G, S}$ to each pair $(G,S)$, where $G$ is any Lie group and $S\subseteq\mathrm{Lie}(G)^*$ is any submanifold satisfying certain non-degeneracy conditions. The spaces $\mathfrak{M}_{G, S}$ satisfy a universal property for symplectic reduction which generalizes that of the universal imploded cross-section. While these Hamiltonian $G$-spaces are explicit and natural from a Lie-theoretic perspective, some of them appear to be new.               

\end{abstract}

\maketitle
\begin{scriptsize}
\tableofcontents
\end{scriptsize}

\section{Introduction}
\subsection{Informal context}

Noether's approach to symmetries and conserved quantities in classical mechanics naturally gives rise to quotient constructions in symplectic geometry. The most basic of these constructions is Marsden--Weinstein--Meyer reduction \cite{marsden-weinstein,meyer} for Hamiltonian Lie group actions, a cornerstone of symplectic geometry \cite{guillemin-sternberg-invent,sjamaar-lerman,sjamaar,meinrenken-sjamaar} with wide-ranging implications for algebraic geometry \cite{losev,hitchin,atiyah-bott,kirwan-1984} and geometric representation theory \cite{balibanu,bezrukavnikov-finkelberg-mirkovic,gan-ginzburg-2,kazhdan-kostant-sternberg,nakajima,etingof-ginzburg,braverman-finkelberg-nakajima}. This construction applies to a Hamiltonian $G$-space, i.e.\ a symplectic manifold $M$ equipped with a Hamiltonian action of a Lie group $G$ and a moment map $\mu:M\longrightarrow\g^*$, where $\g$ is the Lie algebra of $G$. Each element $\xi\in\g^*$ has a $G$-stabilizer $G_{\xi}\subseteq G$ and determines a topological quotient 
\[
M \sll{\xi} G \coloneqq \mu^{-1}(\xi) / G_\xi,
\] called the \textit{Marsden--Weinstein--Meyer reduction of $M$ by $G$ at level $\xi$}.
The space $M \sll{\xi} G$ is a symplectic manifold under certain hypotheses.

Motivated by recent advances in topological quantum field theory \cite{moore-tachikawa, ginzburg-kazhdan}, we introduce and study a generalization of Marsden--Weinstein--Meyer reduction. Our starting point is the well-known equivalence between Hamiltonian $G$-spaces and Hamiltonian spaces for the cotangent groupoid $T^*G\tto\g^*$ \cite{mikami-weinstein}. One then imagines upgrading the level $\xi\in\g^*$ and stabilizer subgroup $G_{\xi}\subseteq G$ to a submanifold $S\subseteq\g^*$ and a ``stabilizer subgroupoid" $G_S\subseteq T^*G$, respectively, and considering
\[
M \sll{S} G \coloneqq \mu^{-1}(S) / G_S.
\]
A potential for further generalization then comes from the notion of Hamiltonian actions by arbitrary symplectic groupoids \cite{mikami-weinstein}; one might hope to replace $T^*G$, $S\subseteq\g^*$, and $G_S\subseteq T^*G$ with a symplectic groupoid $G\tto X$, a submanifold $S\subseteq X$, and a ``stabilizer subgroupoid" $\mathcal{H}\tto S$, respectively.    

To be somewhat more precise, let us address the class of submanifolds $S\subseteq X$ for which our construction can be performed; these are the so-called \emph{pre-Poisson submanifolds} of $X$, as defined and studied by Cattaneo and Zambon \cite{cattaneo-zambon-2007,cattaneo-zambon-2009}. Poisson transversals and Poisson submanifolds are automatically pre-Poisson, and there is a reasonable sense in which generic submanifolds of $X$ are pre-Poisson. At the same time, each pre-Poisson submanifold $S\subseteq X$ canonically determines a Lie subalgebroid $L_S \longrightarrow S$ of the Lie algebroid of $\mathcal{G}$. The notion of a \emph{stabilizer subgroupoid} for $S$ is then transparent; it refers to any isotropic immersed Lie subgroupoid $\mathcal{H}\tto S$ of $\mathcal{G}$ integrating $L_S$. One has a non-empty, discrete family of stabilizer subgroupoids for $S$, exactly one of which is source-connected and source-simply-connected.

Our generalization of Marsden--Weinstein--Meyer reduction amounts to the topological space \[
M \sll{S,\mathcal{H}}\mathcal{G}\coloneqq \mu^{-1}(S) / \mathcal{H}
\] being a symplectic manifold under reasonable hypotheses, where $\mu:M\longrightarrow X$ is the moment map. Any two source-connected choices of $\mathcal{H}\tto S$ yield the same quotient $M \sll{S,\mathcal{H}}\mathcal{G}$, allowing us to define \begin{equation}\label{Equation: Definition}M\sll{S}\mathcal{G}\coloneqq M \sll{S,\mathcal{H}}\mathcal{G}\quad\text{for any source-connected }\mathcal{H}\tto S.\end{equation} We use the nomenclature \textit{symplectic reduction along a submanifold} for the construction described in these last two sentences. In very general terms, the following are some features that it enjoys. The reader is referred to \S\ref{Subsection: Statement of main results} for more precise descriptions of these features. 

\subsubsection*{Categories}
Our construction has analogues in the categories of complex analytic spaces, complex algebraic varieties, and derived Artin stacks, all of which are developed and proved in this paper.

\subsubsection*{Relation to other constructions}
Several well-studied quotient constructions in symplectic geometry can be realized as special cases of symplectic reduction along a submanifold. Such constructions include Marsden--Weinstein--Meyer reduction, Mikami--Weinstein reduction for Hamiltonian symplectic groupoid actions \cite{mikami-weinstein}, the pre-images of Poisson transversals under moment maps \cite{frejlich-marcut,bielawski97,crooks-roeser,crooks-vanpruijssen}, and the Ginzburg--Kazhdan presentation \cite{ginzburg-kazhdan} of Moore--Tachikawa varieties \cite{moore-tachikawa}.
We also find that symplectic cutting \cite{lerman,lerman-meinrenken-tolman-woodward,weitsman,martens-thaddeus,fisher-rayan} and symplectic implosion \cite{guillemin-jeffrey-sjamaar,dancer-kirwan-swann,dancer-kirwan-roeser,safronov} may be described more simply as reduction along a polyhedral set and reduction along a closed Weyl chamber, respectively. Our work also has an interpretation in shifted symplectic geometry \cite{ptvv} as a derived intersection of two Lagrangians.

\subsubsection*{Universal reduced spaces}\label{Subsubsection: Universal reduced spaces}
Our construction yields a very simple, systematic, Lie-theoretic technique for producing Hamiltonian Lie group spaces. One begins with a Lie group $G$ and pre-Poisson submanifold $S\subseteq\g^*$, e.g. $S$ could be any $G$-invariant submanifold or any Poisson transversal. If the pair $(G,S)$ satisfies some non-degeneracy conditions, we show that it determines a \textit{universal} Hamiltonian $G$-space $\mathfrak{M}_{G, S}$; the precise meaning of ``universal" is given in \S\ref{Subsection: Statement of main results}. 

We realize several Hamiltonian $G$-spaces in this way, including products of coadjoint orbits, open Moore--Tachikawa varieties \cite{bielawskipreprint,ginzburg-kazhdan}, universal imploded cross-sections \cite{guillemin-jeffrey-sjamaar}, and some spaces that appear to be new. We also show $\mathfrak{M}_{G,S}$ to be a symplectic groupoid when $S\subseteq\g^*$ is $G$-invariant and appropriate non-degeneracy conditions are imposed. 

\subsection{Statement of main results}\label{Subsection: Statement of main results}
We now give precise formulations of our main results.

\subsubsection{The smooth category}\label{Subsubsection: The smooth category}
Let $(X,\sigma)$ be a Poisson manifold with Poisson bivector field $\sigma:T^*X\longrightarrow TX$. A submanifold $S\subseteq X$ is called \textit{pre-Poisson} if $\sigma^{-1}(TS)\cap TS^{\circ}$ has constant rank over $S$, where $TS^{\circ}\subseteq T^*X$ is the annihilator of $TS\subseteq TX$. Any symplectic groupoid $\mathcal{G}\tto X$ has the property that $X$ is Poisson, and we may therefore consider a pre-Poisson submanifold $S\subseteq X$. One then has a Lie subalgebroid $$L_S\coloneqq\sigma^{-1}(TS)\cap TS^{\circ}$$ of the Lie algebroid of $\mathcal{G}\tto X$. We use the term \textit{stabilizer subgroupoid of} $S$ for any isotropic Lie subgroupoid\footnote{Our convention is that a Lie subgroupoid is a Lie groupoid with a possibly non-injective immersion; see \S\ref{sratijkn}.} $\mathcal{H}\tto S$ of $\mathcal{G}\tto X$ having $L_S$ as its Lie algebroid. 

On the other hand, Mikami and Weinstein \cite{mikami-weinstein} define what it means for $\mathcal{G}\tto X$ to act in a Hamiltonian fashion on a symplectic manifold $(M,\omega)$. Any stabilizer subgroupoid $\mathcal{H}\tto S$ then acts on $N\coloneqq\mu^{-1}(S)$, where $\mu:M\longrightarrow X$ is the moment map realizing the Hamiltonian action of $\mathcal{G}\tto X$ on $M$. We call the quotient topological space $$M\sll{S,\mathcal{H}}\mathcal{G}\coloneqq N/\mathcal{H}$$ the \textit{symplectic reduction of $M$ by $\mathcal{G}$ along $S$ with respect to $\mathcal{H}$}. This space features in the following result.

\begingroup
\setcounter{tmp}{\value{theorem}}% store current value of theorem counter
\setcounter{theorem}{0} %assign desired value to theorem counter
\renewcommand\thetheorem{\Alph{theorem}}% locally redefine the representation of the theorem counter
\begin{theorem}[Smooth category]\label{Theorem: Smooth category}
	Suppose that a symplectic groupoid $\mathcal{G}\tto X$ acts on a symplectic manifold $(M,\omega)$ in a Hamiltonian fashion with moment map $\mu:M\longrightarrow X$. Let $\mathcal{H}\tto S$ be a stabilizer subgroupoid of a pre-Poisson submanifold $S\subseteq X$ and set $N\coloneqq\mu^{-1}(S)$.
	\begin{itemize}
		\item[\textup{(i)}] Assume that $S$ intersects $\mu$ cleanly and that $M\sll{S,\mathcal{H}}\mathcal{G} \coloneqq N / \H$ has a smooth manifold structure for which the quotient map $\pi:N\longrightarrow M\sll{S,\mathcal{H}}\mathcal{G}$ is a submersion. The manifold $M\sll{S,\mathcal{H}}\mathcal{G}$ then carries a unique symplectic form $\bar{\omega}$ for which
		$$\pi^*\bar{\omega}=i^*\omega,$$
		where $i:N\longrightarrow M$ is the inclusion map.
		\item[\textup{(ii)}] If $\mathcal{H}$ acts freely on $N$, then $S$ is transverse to $\mu$. The hypotheses of \textup{Part (i)} are therefore satisfied if $\mathcal{H}$ acts freely and properly on $N$.
		\item[\textup{(iii)}] Let $\sigma:T^*X\longrightarrow TX$ denote the Poisson bivector field on $X$. Assume that the hypotheses of \textup{Part (i)} are satisfied and that $S/\mathcal{H}$ has a smooth manifold structure for which the quotient map $\rho:S\longrightarrow S/\mathcal{H}$ is a submersion. The manifold $S/\mathcal{H}$ then has a unique Poisson structure satisfying the following condition: if $x\in S$, $f,g$ are smooth functions on $S/\mathcal{H}$ defined near $\rho(x)$, and $F,G$ are smooth functions on $X$ defined near $x$ with $dF(\sigma(TS^{\circ}))=0=dG(\sigma(TS^{\circ}))$, then $$\rho^*\{f,g\}_{\mathcal{S}/\mathcal{H}}=\{F,G\}_X\big\vert_S.$$ Furthermore, the moment map $\mu:M\longrightarrow X$ descends to a Poisson map $M\sll{S,\mathcal{H}}\mathcal{G}\longrightarrow S/\mathcal{H}$.
	\end{itemize}
\end{theorem}
\endgroup   
This result appears in the main text as Theorems \ref{main-theorem-smooth}, \ref{1gjd75z2}, and \ref{omd5dx3w}.

Let us suppose that $S=\{x\}$ is a singleton and that $\mathcal{H}$ is the isotropy group at $x$. In this case, $M\sll{S,\mathcal{H}}\mathcal{G}$ is precisely the Mikami--Weinstein reduction of $M$ at level $x$ \cite{mikami-weinstein}. In other words, Mikami--Weinstein reduction is a special case of Theorem \ref{Theorem: Smooth category}. The same is therefore true of Marsden--Weinstein--Meyer reduction, as it is a special case of Mikami--Weinstein reduction.

It is also illuminating to apply Theorem \ref{Theorem: Smooth category} in the case of a Poisson transversal $S\subseteq X$. The trivial groupoid $\mathcal{H}\tto S$ is then a stabilizer subgroupoid of $S$ in $\mathcal{G}$, and one obtains the symplectic submanifold $M\sll{S,\mathcal{H}}\mathcal{G}=\mu^{-1}(S)\subseteq M$.

\subsubsection{The complex analytic and algebraic categories}
Theorem \ref{Theorem: Smooth category} and its consequences have natural counterparts in the category of complex analytic spaces. To this end, call a complex analytic space $(X,\mathcal{O}_X)$ \textit{Poisson} if $\mathcal{O}_X$ is a sheaf of complex Poisson algebras. Now assume that $X$ is a holomorphic Poisson manifold and denote its Poisson bivector field by $\sigma:T^*X\longrightarrow TX$. We call a complex submanifold $S\subseteq X$ a \textit{pre-Poisson complex submanifold} if $\sigma^{-1}(TS)\cap TS^{\circ}$ has constant rank over $S$. The base space of a holomorphic symplectic groupoid $\mathcal{G}\tto X$ is a holomorphic Poisson manifold, in which context we may consider a pre-Poisson complex submanifold $S\subseteq X$. In analogy with \S\ref{Subsubsection: The smooth category}, we may define what it means for $\mathcal{H}\tto S$ to be a \textit{holomorphic stabilizer subgroupoid of} $S$. Hamiltonian actions are also defined analogously, i.e.\ one has an analogous notion of $\mathcal{G}\tto X$ acting on a holomorphic symplectic manifold $M$ in a Hamiltonian fashion with moment map $\mu:M\longrightarrow X$. Let us assume that $N\coloneqq\mu^{-1}(S)$ is reduced. Let us also consider a \textit{complex analytic quotient} $\pi:N\longrightarrow Q$ of $N$ by $\mathcal{H}$; by this, we mean that $Q$ is a complex analytic space, that $\pi$ is holomorphic, and that the canonical map $\mathcal{O}_{Q}\longrightarrow(\pi_*\mathcal{O}_N)^{\mathcal{H}}$ is an isomorphism. We refer to $Q$ as the \textit{symplectic reduction of $M$ by $\mathcal{G}$ along $S$ with respect to $\mathcal{H}$ and $\pi$} and write
$$M\sll{S,\mathcal{H},\pi}\mathcal{G}\coloneqq Q.$$
These considerations yield the following counterpart of Theorem \ref{Theorem: Smooth category}.

\begingroup
\setcounter{tmp}{\value{theorem}}% store current value of theorem counter
\setcounter{theorem}{1} %assign desired value to theorem counter
\renewcommand\thetheorem{\Alph{theorem}}% locally redefine the representation of the theorem counter
\begin{theorem}[Complex analytic category]\label{Theorem: Complex analytic category}
	Let a holomorphic symplectic groupoid $\mathcal{G}\tto X$ act on a holomorphic symplectic manifold $(M,\omega)$ in a Hamiltonian fashion with moment map $\mu:M\longrightarrow X$. Suppose that $\mathcal{H}\tto S$ is a holomorphic stabilizer subgroupoid of a pre-Poisson complex submanifold $S\subseteq X$  and that $N\coloneqq\mu^{-1}(S)$ is reduced. Let us also suppose that $\pi:N\longrightarrow Q$ is a complex analytic quotient of $N$ by $\mathcal{H}$.
	\begin{itemize}
\item[\textup{(i)}] If $p\in N$ and $f\in\mathcal{O}_{Q,\pi(p)}$, then there exists $F\in\mathcal{O}_{M,p}$ satisfying $\pi^*f=F\big\vert_N$ and $dF(TN^{\omega})=0$, where $TN^\omega$ is the annihilator of $TN$ with respect to $\omega$.
\item[\textup{(ii)}] The complex analytic space $Q$ has a unique Poisson structure with the following property: if $p\in N$, $f,g\in\mathcal{O}_{Q,\pi(p)}$, and $F,G\in\mathcal{O}_{M,p}$ satisfy $\pi^*f=F\big\vert_N$, $\pi^*g=G\big\vert_N$, and $dF(TN^{\omega})=0=dG(TN^{\omega})$, then 
$$\pi^*\{f,g\}_Q=\{F,G\}_M\big\vert_N.$$
\item[\textup{(iii)}] Assume that there exists $p\in N$ such that $d\pi_p$ is surjective, $\pi^{-1}(\pi(p))$ is an $\mathcal{H}$-orbit, and $\pi(p)$ is a smooth point of $Q$. The Poisson structure in \textup{Part (ii)} is then non-degenerate on an open dense subset of the smooth locus of $Q$ containing $\pi(p)$. 
\item[\textup{(iv)}] Suppose that $S$ intersects $\mu$ cleanly, and that the topological quotient $Q = N/\H$ has a complex manifold structure such that $\pi : N \longrightarrow Q$ is a holomorphic submersion. The Poisson structure on $Q$ is then induced by a holomorphic symplectic form $\bar{\omega}$ satisfying $$\pi^*\bar{\omega} = i^*\omega,$$where $i : N \longrightarrow M$ is the inclusion map.
\item[\textup{(v)}] If $\H$ acts freely on $N$, then $S$ is transverse to $\mu$.
The hypotheses of \textup{Part (iv)} are therefore satisfied if $\mathcal{H}$ acts freely and properly on $N$.
\item[\textup{(vi)}] Let $\sigma:T^*X\longrightarrow TX$ denote the Poisson bivector field on $X$. If $\rho : S \longrightarrow R$ is a complex analytic quotient of $S$ by $\H$, then there is a unique holomorphic Poisson structure on $R$ such that 
\[
\rho^*\{f, g\}_R = \{F, G\}_X\big\vert_S
\]
for all $x \in S$, $f, g \in \O_{R, \rho(x)}$, and $F, G \in \O_{X, x}$ satisfying $\rho^*f = F\big\vert_S$, $\rho^*g = G\big\vert_S$, and $dF(\sigma(TS^{\circ}))=0=dG(\sigma(TS^{\circ}))$.
\item[\textup{(vii)}] Assume that the hypothesis of \textup{Part (vi)} holds. If $\mu : M \longrightarrow X$ descends to a holomorphic map $\bar{\mu} : Q \too R$, then $\bar{\mu}$ is Poisson with respect to the Poisson structures in \textup{Parts} \textup{(ii)} and \textup{(vi)}.
\end{itemize} 
\end{theorem}
\endgroup

This result appears in the main text as Theorem \ref{ytazg95b}, Proposition \ref{ugr3y1lq}, Proposition \ref{45y1ghc8}, and Theorem \ref{6hr3szsv}.

Theorem \ref{Theorem: Complex analytic category} has incarnations in complex algebraic geometry. The setup is entirely analogous to that outlined in the paragraph preceding Theorem \ref{Theorem: Complex analytic category}; one simply replaces each complex analytic notion with its algebro-geometric counterpart, e.g. holomorphic symplectic groupoids with algebraic symplectic groupoids, holomorphic Poisson manifolds with smooth Poisson varieties, and complex analytic quotients with algebraic quotients\footnote{We refer the reader to Definition \ref{Definition: algebraic quotient} for the precise meaning of this term.}. The algebro-geometric definition of $M\sll{S,\mathcal{H},\pi}\mathcal{G}$ is then analogous to that given in the complex analytic setting.          

\begingroup
\setcounter{tmp}{\value{theorem}}% store current value of theorem counter
\setcounter{theorem}{2} %assign desired value to theorem counter
\renewcommand\thetheorem{\Alph{theorem}}% locally redefine the representation of the theorem counter
\begin{theorem}[Complex algebraic category]\label{Theorem: Algebraic category}
	Let an algebraic symplectic groupoid $\mathcal{G}\tto X$ act on a symplectic variety $(M,\omega)$ in a Hamiltonian fashion with moment map $\mu:M\longrightarrow X$. Suppose that $\mathcal{H}\tto S$ is a stabilizer subgroupoid of a pre-Poisson subvariety $S\subseteq X$ and that $N\coloneqq\mu^{-1}(S)$ is reduced. Suppose also that $\pi:N\longrightarrow Q$ is an algebraic quotient of $N$ by $\mathcal{H}$. Let $\mathcal{O}_M^{\emph{an}}$ denote the structure sheaf of $M$ as a complex analytic space.
	\begin{itemize}
		\item[\textup{(i)}] If $p\in N$ and $f\in\mathcal{O}_{Q,\pi(p)}$, then there exists $F$ in $\mathcal{O}_{M,p}^{\emph{an}}$ satisfying $\pi^*f=F\big\vert_N$ and $dF(TN^{\omega})=0$.
		\item[\textup{(ii)}] The variety $Q$ has a unique algebraic Poisson structure with the following property: if $p\in N$, $f,g\in\mathcal{O}_{Q,\pi(p)}$, and $F,G\in\mathcal{O}_{M,p}^{\emph{an}}$ satisfy $\pi^*f=F\big\vert_N$, $\pi^*g=G\big\vert_N$, and $dF(TN^{\omega})=0=dG(TN^{\omega})$, then 
		$$\pi^*\{f,g\}_Q=\{F,G\}_M\big\vert_N.$$ 
		\item[\textup{(iii)}] Assume that there exists $p \in N$ such that $d\pi_p$ is surjective, $\pi^{-1}(\pi(p))$ is an $\mathcal{H}$-orbit, and $\pi(p)$ is a smooth point of $Q$.
		The Poisson structure in \textup{Part (ii)} is then non-degenerate on a Zariski-open subset of the smooth locus of $Q$ containing $\pi(p)$.   
	\end{itemize} 
\end{theorem}
\endgroup

Theorem \ref{Theorem: Algebraic category} appears in the main text as Theorem \ref{Theorem: Algebraic theorem}.

\subsubsection{Symplectic reduction by a Lie group along a submanifold}\label{Subsubsection: Symplectic reduction by a Lie group along a submanifold}
Our results have implications for classical Hamiltonian $G$-spaces in both the smooth and complex analytic categories, where $G$ is any real or complex Lie group with Lie algebra $\g$. We develop the implications for smooth Hamiltonian $G$-spaces in parallel with those for holomorphic Hamiltonian $G$-spaces, allowing context to resolve any ambiguities that this may create. 

Recall the identification of Hamiltonian $G$-spaces with Hamiltonian spaces for the cotangent groupoid $T^*G\tto\g^*$ described in \cite{mikami-weinstein} and mentioned in \S\ref{Subsubsection: Universal reduced spaces}. This gives rise to the notation $$M\sll{S,\mathcal{H}}G\coloneqq M\sll{S,\mathcal{H}}T^*G$$ for the reduction of a Hamiltonian $G$-space $M$ along a pre-Poisson submanifold $S\subseteq\g^*$ with respect to a stabilizer subgroupoid $\mathcal{H}\tto S$ in $T^*G$. As per \eqref{Equation: Definition} and the discussion preceding it, we may write $$M\sll{S} G\coloneqq M\sll{S,\mathcal{H}} G\quad\text{for any source-connected }\mathcal{H}\tto S.$$ On the other hand, we call the pair $(S,\mathcal{H})$ a \textit{clean reduction datum} if it satisfies the hypotheses of Theorem \ref{Theorem: Smooth category}(i) or Theorem \ref{Theorem: Complex analytic category}(iv). Note that $M\sll{S,\mathcal{H}}G$ is a symplectic manifold in this case. 

One may specialize this discussion to the Hamiltonian $G$-space $M=T^*G$, where $G$ acts on $T^*G$ by right translations. Let us begin by considering
$$\mathfrak{M}_{G,S,\mathcal{H}}\coloneqq T^*G\sll{S,\mathcal{H}}G\quad\text{and}\quad\mathfrak{M}_{G,S}\coloneqq \mathfrak{M}_{G,S,\mathcal{H}}\quad\text{for any source-connected }\mathcal{H}\tto S.$$
Note that $\mathfrak{M}_{G,S,\mathcal{H}}$ is a symplectic manifold if $(S,\mathcal{H})$ is a clean reduction datum for the right translation action of $G$ on $T^*G$. In this case, the left translation action descends to a Hamiltonian $G$-action on $\mathfrak{M}_{G,S,\mathcal{H}}$, i.e.\ $\mathfrak{M}_{G,S,\mathcal{H}}$ is a Hamiltonian $G$-space. Part (i) of the following result justifies our calling $\mathfrak{M}_{G,S,\mathcal{H}}$ the \textit{universal reduced space} associated to $(G,S,\mathcal{H})$.  

\begingroup
\setcounter{tmp}{\value{theorem}}% store current value of theorem counter
\setcounter{theorem}{3} %assign desired value to theorem counter
\renewcommand\thetheorem{\Alph{theorem}}% locally redefine the representation of the theorem counter

\begin{theorem}[Reduction by a Lie group]\label{Theorem: Reduction by a Lie group along a submanifold}
Let $S\subseteq\g^*$ be a pre-Poisson submanifold.
\begin{itemize}
\item[\textup{(i)}] Suppose that $(S,\mathcal{H})$ is a clean reduction datum for both a Hamiltonian $G$-space $M$ and the right translation action of $G$ on $T^*G$. We then have a canonical symplectomorphism
\[
M \sll{S, \H} G \cong (M \times \mathfrak{M}_{G, S, \H}^-) \sll{0} G,
\]
where $\mathfrak{M}_{G, S, \H}^-$ is $\mathfrak{M}_{G, S, \H}$ with the negated symplectic form, $G$ acts diagonally on $M \times \mathfrak{M}_{G, S, \H}^-$, and $\sll{0}$ denotes Marsden--Weinstein--Meyer reduction at level $0$.
\item[\textup{(ii)}] Suppose that the Lie subalgebroid $L_S \coloneqq \sigma^{-1}(TS) \cap TS^\circ \subseteq T^*\g^*$ is contained in the kernel of the Kirillov--Kostant--Souriau Poisson structure $\sigma:T^*\g^*\longrightarrow T\g^*$. The subspace $$\h_{\xi}\coloneqq (T_{\xi}S)^{\circ}\cap\g_{\xi}$$ is then a Lie subalgebra of $\g$ for all $\xi\in S$, where $(T_{\xi}S)^{\circ}\subseteq\g$ is the annihilator of $T_{\xi}S\subseteq\g^*$ and $\g_{\xi}\subseteq\g$ is the centralizer of $\xi$. Furthermore, we have 
$$L_S=\{(x,\xi)\in \g\times S:x\in\h_{\xi}\}.$$
\item[\textup{(iii)}] Retain the hypothesis of \textup{Part (ii)}. Let $H_{\xi}\subseteq G$ be the connected Lie subgroup with Lie algebra $\h_{\xi}$ for all $\xi\in S$, and suppose that $$\mathcal{H}\coloneqq\{(g,\xi)\in G\times S:g\in H_{\xi}\}$$ is closed in $G\times S$. The subset $\mathcal{H}$ is then a stabilizer subgroupoid of $S$ in $T^*G\tto\g^*$ if one uses the left trivialization to identify $T^*G$ with $G\times\g^*$.
\item[\textup{(iv)}] Suppose that $S$ is a $G$-invariant submanifold of $\g^*$. The hypothesis of \textup{Part (ii)} is then satisfied. Let us also assume that the hypotheses of \textup{Part (iii)} are satisfied.
Then $(S,\mathcal{H})$ is a clean reduction datum for the right translation action of $G$ on $T^*G$, and the symplectic manifold $\mathfrak{M}_{G,S}$ inherits the structure of a symplectic groupoid integrating $S$.   
\end{itemize}
\end{theorem}

Theorem \ref{Theorem: Reduction by a Lie group along a submanifold} appears in the main text as Proposition \ref{4ywv9zyt} and Theorems \ref{universality}, \ref{amc0wisd}, and \ref{Theorem: invariant}.

We proceed to obtain the following specific examples of universal reduced spaces.

\begingroup
\setcounter{tmp}{\value{theorem}}% store current value of theorem counter
\setcounter{theorem}{4} %assign desired value to theorem counter
\renewcommand\thetheorem{\Alph{theorem}}%

\begin{theorem}[Specific examples of universal reduced spaces]\label{Theorem: Specific examples}
	Let $G$ be a connected complex semisimple Lie group with Lie algebra $\g$.
	\begin{itemize}
		\item[\textup{(i)}] If $G$ is of adjoint type, $\mathcal{S}\subseteq\g^*$ is a principal Slodowy slice, and $\Delta_n\mathcal{S}\subseteq(\g^*)^n$ is the diagonally embedded copy of $\mathcal{S}$, then $\mathfrak{M}_{G^n,\Delta_n\mathcal{S}}$ is the $n^{\text{th}}$ open Moore--Tachikawa variety. Its affinization is the $n^{\text{th}}$ Moore--Tachikawa variety constructed by Ginzburg--Kazhdan, and the Poisson structure on this variety can be recovered from Theorem \ref{Theorem: Algebraic category}.
		\item[\textup{(ii)}] If $(\g^*_{\mathrm{irr}})^{\circ}$ is the set of subregular semisimple elements in $\g^*$, then $$\mathfrak{M}_{G,(\g^*_{\mathrm{irr}})^{\circ}}=\bigsqcup_{\xi\in (\g^*_{\mathrm{irr}})^{\circ}}G/[G_{\xi},G_{\xi}],$$ where $G_{\xi}$ is the $G$-stabilizer of $\xi\in\g^*$. Furthermore, this space has the structure of a holomorphic symplectic groupoid integrating $(\g^*_{\mathrm{irr}})^{\circ}$. Its symplectic form, Hamiltonian $G$-action, and groupoid structure are described in Example \ref{Example: g_irr}.    
	\end{itemize}
\end{theorem}

Theorem \ref{Theorem: Specific examples} appears in the main text as Example \ref{Example: g_irr} and Theorem \ref{Theorem: Moore-Tachikawa}.

A few brief comments are warranted. We first note that the stabilizer subgroupoid used to construct $\mathfrak{M}_{G^n,\Delta_n\mathcal{S}}$ is a genuine group scheme, i.e.\ a group scheme whose fibres are not all isomorphic to one another. This contrasts with the more familiar construction of reduced spaces as quotients by Lie group actions. The Ginzburg--Kazhdan construction of Moore--Tachikawa varieties thereby inspired much of the setup in this paper. In particular, this construction was the main motivation for our work. 

Our second comment is that Theorem \ref{Theorem: Specific examples}(ii) holds in greater generality; one can reduce along any decomposition class $\mathcal{D}\subseteq\g^*$ \cite{borho-kraft} and obtain an explicit description of the Hamiltonian $G$-space $\mathfrak{M}_{G,\mathcal{D}}$. The Hamiltonian $G$-spaces $\mathfrak{M}_{G,\mathcal{D}}$  do not seem to appear in the literature.

\subsubsection{Examples of symplectic reduction along a submanifold}
In addition to Marsden--Weinstein--Meyer, Mikami--Weinstein reduction, and the pre-images of Poisson transversals under moment maps, we realize each of the following as special cases of symplectic reduction along a submanifold. Retain the notation used in \S\ref{Subsubsection: Symplectic reduction by a Lie group along a submanifold} and identify $T^*G$ with $G\times\g^*$ via the left trivialization as appropriate. Write $G_{\xi}$ for the $G$-stabilizer of $\xi\in\g^*$.    

\begingroup
\setcounter{tmp}{\value{theorem}}% store current value of theorem counter
\setcounter{theorem}{5} %assign desired value to theorem counter
\renewcommand\thetheorem{\Alph{theorem}}% locally redefine the representation of the theorem counter

\begin{theorem}[General examples]\label{Theorem: General examples}
Let $G$ be a Lie group with Lie algebra $\g$. Suppose that $G$ acts on a symplectic manifold $M$ in a Hamiltonian fashion and with moment map $\mu:M\longrightarrow\g^*$.
\begin{itemize}
\item[\textup{(i)}] If $\mathcal{O}\subseteq\g^*$ is a coadjoint orbit, $\mathcal{H}=\{(g,\xi)\in G\times\mathcal{O}:g\in G_{\xi}\}$, $(\mathcal{O},\mathcal{H})$ is a clean reduction datum, and $\mathcal{O}^{-}$ is $\mathcal{O}$ with the negated symplectic form, then $M\sll{\mathcal{O},\mathcal{H}}G=\mu^{-1}(\mathcal{O})/G\times\mathcal{O}^{-}$.
\item[\textup{(ii)}] If $G$ is compact and connected and $\mathfrak{t}_+^*\subseteq\g^*$ is a closed fundamental Weyl chamber, then $M\sll{\mathfrak{t}_+^*}G$ is the imploded cross-section of $M$. 
\item[\textup{(iii)}] If $G=T$ is a compact torus and $P\subseteq\mathfrak{g}^*=\mathfrak{t}^*$ is a polyhedral set, then $M\sll{P}T$ is the symplectic cut of $M$ with respect to $P$.
\end{itemize} 
\end{theorem}

While $\mathfrak{t}_+^*$ and $P$ are not submanifolds, each is stratified into pre-Poisson submanifolds. This allows one to form each of $M\sll{\mathfrak{t}_+^*}G$ and $M\sll{P}T$ on a stratum-by-stratum basis. We refer the reader to Remark \ref{zlyzkwmo} for further details. 

Theorem \ref{Theorem: General examples} appears in the main text as Propositions \ref{l73ftj1i}, \ref{0dxzp0g8}, and \ref{Proposition: Cut}.

\subsubsection{Shifted symplectic interpretation}
The notion of symplectic reduction along a submanifold can be interpreted as a construction in shifted symplectic geometry \cite{ptvv} similar to the one for Marsden--Weinstein--Meyer reduction \cite[\S2.1.2]{calaque} and quasi-Hamiltonian reduction \cite{safronov-2016}.
More precisely, we obtain it as a derived intersection of two Lagrangian morphisms \cite[Theorem 2.9]{ptvv} in the $1$-shifted symplectic stack $[X/\G]$ associated to a symplectic groupoid $\G \tto X$.
One consequence is that most of the assumptions of Theorem \ref{Theorem: Algebraic category} can be dropped at the expense of obtaining a derived stack $[M \sll{S, \H} \G]$ endowed with a 0-shifted symplectic structure.
This picture also explains the definition of a stabilizer subgroupoid, as shown by Part (i) of the following theorem.

\begingroup
\setcounter{tmp}{\value{theorem}}% store current value of theorem counter
\setcounter{theorem}{6} %assign desired value to theorem counter
\renewcommand\thetheorem{\Alph{theorem}}% locally redefine the representation of the theorem counter

\begin{theorem}[Shifted symplectic interpretation]\label{Theorem: Shifted}
Let $\G \tto X$ be an algebraic symplectic groupoid.
\begin{enumerate}
\item[\textup{(i)}]
An algebraic subgroupoid $\H \tto S$ of $\G \tto X$ is a stabilizer subgroupoid of a pre-Poisson subvariety if and only if the zero $2$-form on $S$ is a Lagrangian structure on the morphism of quotient stacks $[S/\H] \too [X/\G]$.

\item[\textup{(ii)}] Suppose that $\G \tto X$ acts on a symplectic variety $M$ in a Hamiltonian fashion with moment map $\mu : M \too X$, and let $\H \tto S$ be a stabilizer subgroupoid of a pre-Poisson subvariety $S$ in $\G \tto X$.
We then have two Lagrangian morphisms $[S/\H] \too [X/\G]$ and $[\mu] : [M/\G] \too [X/\G]$ on the $1$-shifted symplectic stack $[X/\G]$, inducing a $0$-shifted symplectic structure on the derived fibre product
\[
[M\sll{S, \H} \G] \coloneqq [S/\H] \times_{[X/\G]}^h [M/\G].
\]
\end{enumerate}
\end{theorem}

This result appears in the main text as Proposition \ref{yrb59dpc} and Theorem \ref{pw00lbch}.

\subsection{Organization}  Our paper is organized as follows. Sections \ref{smooth-case}, \ref{phj103sm}, \ref{lie-group-case}, \ref{algebraic-case}, and \ref{95hfl28k} are based around the discussions, proofs, implications of Theorems \ref{Theorem: Smooth category}, \ref{Theorem: Complex analytic category}, \ref{Theorem: Reduction by a Lie group along a submanifold}, \ref{Theorem: Algebraic category}, and \ref{Theorem: Shifted}, respectively. Theorem \ref{Theorem: Specific examples} is divided between Sections \ref{lie-group-case} and \ref{algebraic-case}, while Theorem \ref{Theorem: General examples} appears in Section \ref{lie-group-case}. Each of the sections in this paper begins with an outline of its content and structure. Section \ref{95hfl28k} is followed by a list of recurring notation in our paper. 

\subsection*{Acknowledgements} We gratefully acknowledge Victor Ginzburg for sending us a draft of joint work with David Kazhdan \cite{ginzburg-kazhdan}. We also thank Justin Hilburn, Steven Rayan, and Marco Zambon for fruitful conversations.

\section{Main construction: smooth version}\label{smooth-case}
The present section is devoted to the discussion, proof, and implications of Theorem \ref{Theorem: Smooth category}. We begin with an overview of the Lie groupoid-theoretic preliminaries in \S\ref{8s76u88n}. This leads to treatment of pre-Poisson submanifolds and stabilizer subgroupoids in \S\ref{sratijkn}. The proofs Parts (i), (ii), and (iii) in Theorem \ref{Theorem: Smooth category} then appear in \S\ref{Subsection: Symplectic reduction along a submanifold}, \S\ref{jp33sbh7}, and \S\ref{jp33sbh7}, respectively. We conclude with \S\ref{Subsection: Smooth examples}, which briefly outlines some special cases of Theorem \ref{Theorem: Smooth category}.  

\subsection{Preliminaries}\label{8s76u88n}
We begin by assembling the concepts and conventions needed to define and study symplectic reduction along a submanifold.

\subsubsection{Actions of Lie groupoids}

Let $\G \tto X$ be a Lie groupoid.
We denote the source and target maps by $\sss:\mathcal{G}\longrightarrow X$ and $\ttt:\mathcal{G}\longrightarrow X$, respectively.
Our convention is that, given $g, h \in \G$, the product $gh$ is defined if and only if $\ttt(g) = \sss(h)$.
The base $X$ is identified as a submanifold of $\G$ via the identity bisection $1 : X \hookrightarrow \G$.
We say that $\G \tto X$ is an \dfn{ssc Lie groupoid} if every fibre of $\sss$ is connected and simply-connected.
The Lie algebroid of $\G \tto X$ shall be denoted $\Lie(\G)$ and realized as the vector bundle $(\ker d\ttt)\big\vert_X \longrightarrow X$ with anchor map $d\sss : (\ker d\ttt)\big\vert_X \longrightarrow TX$.

Now consider a smooth manifold $M$, a smooth map $\mu:M\longrightarrow X$, and the fibred product \[
	\G \prescript{}{\ttt}{\times}_\mu M \coloneqq \{(g, p) \in \G \times M : \ttt(g) = \mu(p)\}.
	\] Recall that an \dfn{action} of $\G \tto X$ on $M$ with \dfn{moment map} $\mu$ is a smooth map
	\[
	\G \prescript{}{\ttt}{\times}_\mu M \too M, \quad (g, p) \mtoo g \cdot p,
	\] such that $\mu(g \cdot p) = \sss(g)$, $(g \cdot h) \cdot p = g \cdot (h \cdot p)$, and $1_{\mu(p)} \cdot p = p$ for all $p \in M$ and $g, h \in \G$ for which these expressions are defined.
The $\G$-\dfn{orbit} of a point $p \in M$ is the set
\[
\G \cdot p \coloneqq \{ g \cdot p : (g, p) \in \G \prescript{}{\ttt}{\times}_\mu M \},
\]
an immersed submanifold of $M$.
The \dfn{quotient} of $M$ by $\G$, denoted $M / \G$, is the space of orbits with the quotient topology.
We say that the action is \dfn{free} if $g \cdot p = p$ for $(g, p) \in \G \prescript{}{\ttt}{\times}_\mu M$ implies $g = 1_{\mu(p)}$, and \dfn{proper} if the map
\begin{equation}\label{ff300lau}
\G \prescript{}{\ttt}{\times}_\mu M \too M \times M,\quad (g, p) \mtoo (p, g\cdot p)
\end{equation}
is proper.
As for Lie group actions, if the action is free and proper, or more generally the image of \eqref{ff300lau} is a closed embedded submanifold, then the quotient $M / \G$ is a topological manifold with a unique smooth structure such that the quotient map $M \longrightarrow M / \G$ is a smooth submersion \cite[Theorem 1.6.20]{mackenzie}.

\subsubsection{Hamiltonian systems}
Let us recall the notion of a Hamiltonian action of a symplectic groupoid on a symplectic manifold, following Mikami--Weinstein \cite{mikami-weinstein}.

A \dfn{Poisson bivector field} on a smooth manifold $P$ is a bundle map $\sigma : T^*P \longrightarrow TP$ such that the bilinear map $\{f, g\} \coloneqq \sigma(df)g$ on smooth functions $f, g \in C^\infty(P)$ is a Lie bracket.
In this case, we call the pair $(P, \sigma)$ a \dfn{Poisson manifold}.
For a smooth function $f$ on $P$, we denote by $X_f$ the Hamiltonian vector field $X_f \coloneqq \sigma(df)$.
We denote by $P^-$ the Poisson manifold with the opposite Poisson structure $-\sigma$.
We say that a submanifold $Q \s P$ is \dfn{coisotropic} if $\sigma(TQ^\circ) \s TQ$, where $TQ^\circ$ is the annihilator of $TQ$ in $T^*P$.

A \dfn{symplectic groupoid} \cite{weinstein-1987, coste-dazord-weinstein} is a Lie groupoid $\G\tto X$ together with a symplectic form $\Omega$ on $\G$, such that the graph
\[
\Gamma_\G \coloneqq \{(g, h, gh) : (g, h) \in \G \prescript{}{\ttt}{\times}_\sss \G\}
\]
of multiplication is Lagrangian in $\G \times \G \times \G^-$.
In this case, there is a canonical vector bundle isomorphism $\Lie(\G) \cong T^*X$ given by
\begin{equation}\label{r6nwj6xb}
\ker d\ttt_x \too T^*_xX, \quad v \mtoo \Omega(v)\big\vert_{T_xX}
\end{equation}
for all $x \in X$, where $\Omega$ is viewed as a bundle map $T\G \longrightarrow T^*\G$.
The anchor map $\sigma : \Lie(\G) \cong T^*X \longrightarrow TX$ is then a Poisson bivector field on $X$ such that the source and target maps are Poisson and anti-Poisson, respectively.
%Conversely, every Poisson manifold $(X, \sigma)$ induces a Lie algebroid structure on $T^*X$ that can be integrated to a symplectic local groupoid. \pcnote{Can we avoid parenthetical sentences? This could perhaps be made into a footnote or remark.}

\begin{defn}[{Mikami--Weinstein \cite{mikami-weinstein}}]\label{Definition: Hamiltonian system}
An action of a symplectic groupoid $\G\tto X$ on a Poisson manifold $M$ with moment map $\mu : M \longrightarrow X$ is \dfn{Hamiltonian} if its graph 
\[
\Gamma \coloneqq \{(g, p, q) \in \G \times M \times M :
\ttt(g) = \mu(p) \text{ and } q = g \cdot p \}
\]
is coisotropic in $\G \times M \times M^-$.
If the Poisson structure on $M$ is induced by a symplectic form $\omega$, we call $((M, \omega), \G \tto X, \mu)$ a \dfn{Hamiltonian system}.
\end{defn}

\begin{rem}
The moment map $\mu : M \longrightarrow X$ of a Hamiltonian system is automatically Poisson \cite[Theorem 3.8]{mikami-weinstein}.
Conversely, any Poisson map $\mu : M \longrightarrow X$ from a symplectic manifold $M$ to a Poisson manifold $X$ uniquely determines an action of the symplectic local groupoid integrating $X$ \cite[Chapitre III, Th\'eor\`eme 1.1]{coste-dazord-weinstein}.
\end{rem}

\begin{rem}
Each Lie group $G$ with Lie algebra $\mathfrak{g}$ has an associated symplectic groupoid $T^*G\tto\mathfrak{g}^*$. 
It has the property that Hamiltonian systems $((M, \omega), T^*G \tto \mathfrak{g}^*, \mu)$ are equivalent to the more familiar Hamiltonian $G$-spaces. We refer the reader to \S\ref{fz9catpk} for further details.    
\end{rem}

\subsection{Pre-Poisson submanifolds and stabilizer subgroupoids}\label{sratijkn}
For a submanifold $S$ of a Poisson manifold $(X, \sigma)$, consider the subset of $T^*X$ given by
\begin{equation}\label{pb4s1a5o}
L_S \coloneqq \sigma^{-1}(TS) \cap TS^\circ,
\end{equation}
where $TS^\circ$ is the annihilator of $TS$ in $T^*X$.
Our reduction procedure is based on the following class of submanifolds.

\begin{defn}[{Cattaneo--Zambon \cite{cattaneo-zambon-2007} and \cite[Definition 2.2]{cattaneo-zambon-2009}}]
A submanifold $S$ of a Poisson manifold $(X, \sigma)$ is called \dfn{pre-Poisson} if \eqref{pb4s1a5o} has constant rank over $S$.
\end{defn}

\begin{rem}\label{157l9vhc}
In \cite{cattaneo-zambon-2007, cattaneo-zambon-2009}, pre-Poisson submanifolds are defined by the condition that $\sigma(TS^\circ) + TS$ has constant rank.
The two definitions are equivalent since $(\sigma(TS^\circ) + TS)^\circ = \sigma^{-1}(TS) \cap TS^\circ$.
\end{rem}

Note that there is no constraint on the dimension of $S$.
In fact, a generic submanifold is pre-Poisson in the following sense.

\begin{prop}
For a generic point $x \in X$, a generic subspace $V \s T_xX$, and an arbitrary submanifold $S \s X$ with $x \in S$ and $T_xS = V$, $S$ is pre-Poisson in a neighbourhood of $x$.
\end{prop}

\begin{proof}
The term ``generic'' used in this statement refers to an element of some fixed open dense set.
The set of $x \in X$ such that $\sigma$ has constant rank in a neighbourhood of $x$ is open and dense, so this reduces to the case of a Poisson vector space $(W, \sigma)$.
The proposition then follows from the fact that the set of $k$-dimensional subspaces $V \s W$ such that $V + \sigma(V^\circ)$ has maximal dimension is open and dense in the Grassmannian of $k$-planes in $W$.
\end{proof}

We have the following important result on pre-Poisson submanifolds.

\begin{thm}[{Cattaneo--Zambon \cite[Proposition 3.6 and Proposition 7.2]{cattaneo-zambon-2009}}]\label{gzjt46cf}
Let $\G \tto X$ be a symplectic groupoid and $S \s X$ a pre-Poisson submanifold.
\begin{enumerate}
\item[\textup{(i)}]
The subset $L_S$ of $T^*X$ given by \eqref{pb4s1a5o} is a Lie subalgebroid of $\Lie(\G) = T^*X$.
\item[\textup{(ii)}] The ssc Lie groupoid $\G_S \tto S$ with Lie algebroid $L_S$ is isotropic in $\G$.
\end{enumerate}
\end{thm}

Part (ii) requires further explanation.
First, recall that a Lie subalgebroid of an integrable Lie algebroid is integrable \cite[Proposition 3.4]{moerdijk-mrcun-2002}.
It follows from Part (i) that $L_S$ integrates to an ssc Lie groupoid $\G_S \tto S$ together with an immersion $\G_S \longrightarrow \G$ \cite[Proposition 3.5]{moerdijk-mrcun-2002}.
The condition that $\G_S$ is isotropic in $\G$ then means that the pullback of the symplectic form on $\G$ to $\G_S$ vanishes.

In this paper, we follow the convention that a \dfn{Lie subgroupoid} of a Lie groupoid $\G$ is a Lie groupoid $\H$ together with a possibly non-injective immersion $\H \longrightarrow \G$ that is also a Lie groupoid homomorphism.
The groupoid $\G_S \tto S$ in Theorem \ref{gzjt46cf} is therefore an isotropic Lie subgroupoid of $\G \tto S$.
These considerations motivate the following definition.

\begin{defn}\label{Definition: Subgroupoid}
Let $\G \tto X$ be a symplectic groupoid and $S \s X$ a pre-Poisson submanifold.
The Lie algebroid $L_S$ defined in \eqref{pb4s1a5o} is called the \dfn{stabilizer subalgebroid} of $S$.
A \dfn{stabilizer subgroupoid of $S$ in $\G$} is any (possibly non-source-connected) isotropic Lie subgroupoid $\H \tto S$ of $\G \tto X$ with Lie algebroid $L_S$.
The unique ssc stabilizer subgroupoid of $S$ in $\G$ is denoted $\G_S \tto S$.
\end{defn}

\begin{rem}\
\label{2g4d9b8v}
\begin{enumerate}
\item[(i)]
This terminology is justified by the fact that stabilizer subgroupoids generalize stabilizer subgroups in Marsden--Weinstein--Meyer reduction, which we explain in  \S\ref{marsden-weinstein-example}.

\item[(ii)]
By \cite[Proposition 7.2]{cattaneo-zambon-2009}, any source-connected Lie subgroupoid $\H \tto S$ of $\G \tto X$ with Lie algebroid $L_S$ is isotropic, and hence a stabilizer subgroupoid.
More generally, if $\H$ is source-connected over an open dense subset of $S$, then it is isotropic.

\item[(iii)]
On the other hand, a non-source-connected Lie subgroupoid $\H \tto S$ with Lie algebroid $L_S$ might not be isotropic.
For example\footnote{We thank Marco Zambon for pointing this out.}, let $X$ be a smooth manifold with the zero Poisson structure $\sigma = 0$.
Then $\G = T^*X$ with its canonical symplectic form is a symplectic groupoid integrating $X$ \cite[Ch.\ II, Example 3.3]{coste-dazord-weinstein}.
Consider $S = X$, so that $L_S$ is trivial.
For any non-vanishing $1$-form $\alpha$ on $X$, 
\[
\H \coloneqq \{k\alpha_x : x \in X, k \in \mathbb{Z}\}
\]
is a Lie subgroupoid of $\G = T^*X$ with Lie algebroid $L_S = 0$.
However, the pullback of the canonical symplectic form on $T^*X$ by $\alpha : X \longrightarrow T^*X$ is $d\alpha$. It follows that $\H$ is isotropic if and only if $\alpha$ is closed.

\item[(iv)]
One explanation for the isotropy condition is that stabilizer subgroupoids are precisely the Lie subgroupoids of $\G$ for which the map of quotient stacks $[S/\H] \longrightarrow [X/\G]$ has a canonical Lagrangian structure.
This is explained in \S\ref{95hfl28k}.
\end{enumerate}
\end{rem}

Another important result of \cite{cattaneo-zambon-2009} is that pre-Poisson submanifolds can be embedded coisotropically in Poisson transversals, as we now recall.

We begin by recalling that a submanifold $Y$ of a Poisson manifold $(X, \sigma)$ is called a \dfn{Poisson transversal} (or sometimes a \dfn{cosymplectic submanifold}) if $Y$ intersects every symplectic leaf transversally and in a symplectic submanifold of the leaf \cite{frejlich-marcut}.
Equivalently, $Y$ is Poisson transversal if
\[
TX\big\vert_Y = TY \oplus \sigma(TY^\circ).
\]
In this case, $Y$ inherits a canonical Poisson structure from $X$ \cite[Lemma 3]{frejlich-marcut}.
To describe this Poisson structure as a bivector field $\sigma_Y:T^*Y \longrightarrow TY$, note that the restriction of linear functionals defines an isomorphism $\sigma^{-1}(TY) \longrightarrow T^*Y$.
One then defines $\sigma_Y$ by $\sigma_Y(\xi) = \sigma(\hat{\xi})$, where $\hat{\xi}$ is the unique element of $\sigma^{-1}(TY)$ mapping to $\xi \in T^*Y$.

Cattaneo--Zambon \cite{cattaneo-zambon-2009} prove that every pre-Poisson submanifold $S$ of a Poisson submanifold $(X, \sigma)$ can be embedded in a Poisson transversal $Y \s X$ such that $S$ is coisotropic in $Y$ \cite[Theorem 3.3]{cattaneo-zambon-2009}.

\subsection{Symplectic reduction along a submanifold}\label{Subsection: Symplectic reduction along a submanifold}
Let $\mu : M \longrightarrow X$ be a smooth map between smooth manifolds. We say that a submanifold $S \s X$ intersects $\mu$ \dfn{cleanly} if $\mu^{-1}(S)$ is a submanifold of $M$ satisfying $$T_p \mu^{-1}(S) = d\mu_p^{-1}(T_{\mu(p)}S)$$ for all $p \in \mu^{-1}(S)$.
Recall that this condition is satisfied if $S$ is \dfn{transverse} to $\mu$, i.e.\ if \ $$T_{\mu(p)}S + \im d\mu_p = T_{\mu(p)}X$$ for all $p \in \mu^{-1}(S)$.

We are now equipped to begin proving Theorem \ref{Theorem: Smooth category} from the introduction. The main idea is that for a Hamiltonian system $((M, \omega), \G \tto X, \mu)$, pre-Poisson submanifold $S \s X$, and stabilizer subgroupoid $\H \tto S$, the distribution defined by the kernel of the restriction of $\omega$ to $\mu^{-1}(S)$ coincides with the orbits of the $\H$-action.
This result can be stated more generally for Poisson manifolds:

\begin{thm}\label{distribution-proposition}
Let $\G \tto X$ be a symplectic groupoid acting on a Poisson manifold $(M, \tau)$ in a Hamiltonian way with moment map $\mu : M \longrightarrow X$.
Let $S \s X$ be a pre-Poisson submanifold intersecting $\mu$ cleanly, let $\H \tto S$ be a stabilizer subgroupoid of $S$ in $\G$, and let $N \coloneqq \mu^{-1}(S)$.
We then have
\[
T_p(\H \cdot p) = T_pN \cap \tau(T_pN^\circ)
\]
for all $p \in N$.
\end{thm}

To prove this theorem, we begin with a few lemmas.
Let $((M, \tau), \G \tto X, \mu)$, $\H \tto S$, and $N$ be as in Theorem \ref{distribution-proposition}.

\begin{lem}\label{1hpyffiz}
We have
\[
T_x\H = d\sss_x^{-1}(T_xS) \cap d\ttt_x^{-1}(T_xS) \cap (T_xS)^\Omega
\]
for all $x \in S$, where $\Omega$ is the symplectic form on $\G$.
\end{lem}

\begin{proof}
Let $\varphi : (\ker d\ttt)\big\vert_X \longrightarrow T^*X$ be the isomorphism given by \eqref{r6nwj6xb}, so that $$\sigma \circ \varphi = d\sss : (\ker d\ttt)\big\vert_X \longrightarrow TX.$$
Let $x \in S$.
The definition of $\H$ then implies that
\[
\ker d\ttt_x \cap T_x\H = \varphi^{-1}(\sigma^{-1}(T_xS) \cap T_xS^\circ) = \ker d\ttt_x \cap d\sss^{-1}_x(T_xS) \cap T_xS^\Omega.
\]
It follows that
\[
T_x\H = T_xS \oplus (\ker d\ttt_x \cap T_x\H) = T_xS \oplus (\ker d\ttt_x \cap d\sss^{-1}_x(T_xS) \cap T_xS^\Omega).
\]
But $T_xS \subseteq T_xS^\Omega$ since $X$ is Lagrangian in $\G$, so
\[
T_x\H = (T_xS \oplus (\ker d\ttt_x \cap d\sss^{-1}_x(T_xS))) \cap T_xS^\Omega = d\sss_x^{-1}(T_xS) \cap d\ttt_x^{-1}(T_xS) \cap T_xS^\Omega.\qedhere
\]
\end{proof}

Denote the action map by $\psi : \G \prescript{}{\ttt}{\times}_\mu M \longrightarrow M$ and the orbit map at a point $p \in M$ over $x \coloneqq \mu(p) \in X$ by
\[
\psi_p : \G_x \too M, \quad g \mtoo g \cdot p,
\]
where $\G_x \coloneqq \ttt^{-1}(x)$.
If $p \in N$ then $T_p(\H \cdot p)$ is the image of $T_x\H_x$ under $d\psi_p$.

\begin{lem}\label{gn98u3jv}
For all $f \in C^\infty(X)$ and $(g, p) \in \G \times_X M$, we have
\[
d\psi_p(X_{\sss^*f}(g)) = X_{\mu^*f}(g \cdot p).
\]
\end{lem}

\begin{proof}
Consider the map
\[
F : \G \times M \times M^- \too \R, \quad F(g, p, q) = f(\sss(g)) - f(\mu(q)).
\]
The graph $\Gamma$ of the action is coisotropic and  $F\big\vert_\Gamma = 0$, so $X_F$ is tangent to $\Gamma$.
But $$dF = (d(\sss^*f), 0, -d(\mu^*f)),$$ so $X_F = (X_{\sss^*f}, 0, X_{\mu^*f})$.
Hence, $$(X_{\sss^*f}(g), 0, X_{\mu^*f}(g \cdot p)) = (u, v, d\psi(u, v))$$ for some $(u, v) \in T_{(g, p)}(\G \times_X M)$, which implies that $v = 0$ and $u = X_{\sss^*f}(g)$.
We conclude that $$X_{\mu^*f}(g \cdot p) = d\psi(u, 0) = d\psi_p(X_{\sss^*f}(g)).$$
\end{proof}

\begin{lem}\label{kvugo3gu}
Let $x \in S$ and $v \in T_x\G$.
Then $v \in T_x\H_x$ if and only if $v = X_{\sss^*f}(x)$ for some $f \in C^\infty(X)$ such that $f\big\vert_S = 0$ and $X_f(x) \in T_xS$.
\end{lem}

\begin{proof}
By Lemma \ref{1hpyffiz}, we have 
\begin{equation}\label{29xt0xz1}
T_x\H_x = \ker d\ttt_x \cap d\sss_x^{-1}(T_xS) \cap (T_xS)^\Omega.
\end{equation}
Let $v = X_{\sss^*f}(x)$ be as stated.
Then $d\ttt(v) = 0$ \cite[Theorem 1.6(i)]{mikami-weinstein}, $d\sss(v) = X_f(x) \in T_xS$ since $\sss$ is a Poisson map, and $$\Omega(v, T_xS) = df(d\sss(T_xS)) = df(T_xS) = 0.$$
Hence, $v \in T_x\H_x$ by \eqref{29xt0xz1}.

Conversely, let $v \in T_x\H_x$.
By \eqref{29xt0xz1}, we have $d\ttt(v) = 0$, $d\sss(v) \in T_xS$ and $\Omega(v)\big\vert_{T_xX} \in T_xS^\circ$.
The latter implies that $\Omega(v)\big\vert_{T_xX} = df_x$ for some $f \in C^\infty(X)$ such that $f\big\vert_S = 0$.
Then $\Omega(v) = df \circ d\sss_x$ since $T_x\G = T_xX \oplus \ker d\sss_x$ and $(\ker d\sss_x)^\Omega = \ker d\ttt$.
It follows that $v = X_{\sss^*f}(x)$ and $X_f(x) = d\sss(v) \in T_xS.$
\end{proof}

We now prove Theorem \ref{distribution-proposition}. 

\begin{proof}[Proof of Theorem \ref{distribution-proposition}]
To show that $$T_p(\H \cdot p) \s T_pN \cap \tau(T_pN^\circ),$$ let $x \coloneqq \mu(p)$ and $v \in T_x\H_x$.
Note that $$u \coloneqq d\psi_p(v) \in T_p(\H \cdot p).$$
Then $u \in T_pN$ since $\H$ acts on $N$.
By Lemma \ref{kvugo3gu}, we have $v = X_{\sss^*f}(x)$ for some $f \in C^\infty(X)$ with $f\big\vert_S = 0$ and $X_f(x) \in T_xS$, and by Lemma \ref{gn98u3jv}, $$u = d\psi_p(v) = X_{\mu^*f}(p).$$
But $d\mu(X_{\mu^*f}(p)) = X_f(x) \in T_xS$ since $\mu$ is a Poisson map; this follows from \cite[Theorem 3.8]{mikami-weinstein}, as explained in \cite[\S4.2]{bursztyn-crainic}. Hence $u \in T_pN$.
We also have $$d(\mu^*f)(T_pN) = df(d\mu(T_pN)) \s df(T_xS) = 0,$$ so $u = X_{\mu^*f}(p) \in \tau(T_pN^\circ).$

Conversely, let $v \in T_pN \cap \tau(T_pN^\circ)$.
Since $$T_pN^\circ = (d\mu^{-1}(T_xS))^\circ = d\mu^*(T_xS^\circ),$$ we have $v = X_{\mu^*f}(p)$ for some $f \in C^\infty(X)$ with $f\big\vert_S = 0$.
We also have $$X_f(x) = d\mu(v) \in T_xS.$$
We thus have $X_{\sss^*f}(x) \in T_x\H_x$ by Lemma \ref{kvugo3gu} and $v = d\psi_p(X_{\sss^*f}(x))$ by Lemma \ref{gn98u3jv}, so $v \in T_p(\H \cdot p)$.
\end{proof}

The main theorem of this section is the following generalization of Mikami--Weinstein reduction.

\begin{thm}[Theorem \ref{Theorem: Smooth category}(i)]\label{main-theorem-smooth}
Let $((M, \omega), \G \tto X, \mu)$ be a Hamiltonian system, $S \subseteq X$ a pre-Poisson submanifold, $\H \tto S$ a stabilizer subgroupoid of $S$ in $\G$, and $N \coloneqq \mu^{-1}(S)$.
Suppose that $S$ intersects $\mu$ cleanly and that $N / \H$ has the structure of a smooth manifold for which the quotient map $\pi : N \longrightarrow N / \H$ is a smooth submersion. 
Then there is a unique symplectic form $\bar{\omega}$ on $N / \H$ satisfying $\pi^*\bar{\omega} = i^*\omega$, where $i : N \longrightarrow M$ is the inclusion map.
\end{thm}

\begin{proof}
We have $\ker d\pi_p = T_p(\H \cdot p)$ for all $p \in N$ by assumption.
Hence, as in standard symplectic reduction \cite{marsden-weinstein, mikami-weinstein}, it suffices to show that the distribution $\ker(i^*\omega) = TN \cap TN^\omega$ on $N$ coincides with the one induced by the $\H$-orbits.
This is precisely the content of Theorem \ref{distribution-proposition}.
\end{proof}

\begin{defn}\label{Definition: Reduction along a submanifold}
Let $((M, \omega), \G \tto X, \mu)$ be a Hamiltonian system.
A \dfn{reduction datum} is a pair $(S, \H)$, where $S$ is a pre-Poisson submanifold of $X$ and $\H\tto S$ is a stabilizer subgroupoid of $S$ in $\G$.
The quotient topological space $\mu^{-1}(S) / \H$ is called the \dfn{symplectic reduction of $M$ by $\G$ along $S$ with respect to $\H$} and denoted
\[
M \sll{S, \H} \G \coloneqq \mu^{-1}(S) / \H.
\]
We say that the reduction datum $(S, \H)$ is \dfn{clean} if the assumptions of Theorem \ref{main-theorem-smooth} are satisfied, in which case $M \sll{S, \H} \G$ is a symplectic manifold.
If $\H$ is source-connected, we use the simplified notation
\[
M \sll{S} \G \coloneqq M \sll{S, \H} \G
\]
and terminology \dfn{symplectic reduction of $M$ by $\G$ along $S$}.
\end{defn}

\begin{rem}
The quotient $M \sll{S, \H} \G$ depends on $\H$ only up to its image in $\G$.
In particular, $M \sll{S, \H} \G = M \sll{S, \G_S} \G$ for any source-connected stabilizer subgroupoid $\H$ of $S$ in $\G$.
This explains and justifies the simplified notation $M \sll{S} \G$.
\end{rem}

\begin{rem}\label{zlyzkwmo}
More generally, let $S \s X$ be a stratified space whose strata are pre-Poisson submanifolds $S_i$ of $X$.
For a collection of stabilizer subgroupoids $\H = (\H_i \tto S_i)_i$, we may define
\[
M \sll{S, \H} \G \coloneqq \bigcup_i \mu^{-1}(S_i) / \H_i,
\]
and regard it as a topological quotient of $\mu^{-1}(S)$.
This notion will be useful when discussing symplectic implosion (\S\ref{nxuzyzm0}) and symplectic cutting (\S\ref{jyt6wx64}).
\end{rem}

\begin{rem}\label{general-reduction-remark}
Suppose that we had assumed $\mathcal{H}$ to be source-connected in Theorem \ref{main-theorem-smooth}. The theorem could then be viewed as a special case reduction along a pre-symplectic submanifold, i.e.\ a submanifold $i: N \hookrightarrow M$ such that $i^*\omega$ has constant rank.
The distribution $\ker(i^*\omega) \s TN$ would then necessarily be integrable, so that the leaf space would symplectic when smooth \cite[Theorem 25.2]{guillemin-sternberg}.
The essence of Theorem \ref{main-theorem-smooth} is as follows: if $\mu : M \longrightarrow X$ is a Poisson map, then for any pre-Poisson submanifold $S \subseteq X$, the leaves in $N \coloneqq \mu^{-1}(S)$ are explicitly realized as the orbits of a groupoid action. 
\end{rem}

\subsection{A sufficient condition for smoothness}
\label{jp33sbh7}

In the context of Marsden--Weinstein--Meyer reduction, one knows that zero is a regular value of the moment map if the Hamiltonian action in question is free.
The goal of this section is to generalize this classical fact to our setting.

Let $((M, \omega), \G \tto X, \mu)$ be a Hamiltonian system, let $(S, \H)$ be a reduction datum, and let $N \coloneqq \mu^{-1}(S)$.

\begin{prop}\label{pxamcx75}
Let $p \in N$, let $x \coloneqq \mu(p)$, and let $\varphi_p : \H_x \longrightarrow N$ be the orbit map $\varphi_p(h) = h \cdot x$.
We have
\[
T_xS + \im d\mu_p = (\ker(d\varphi_p)_x)^\circ,
\]
where $\ker(d\varphi_p)_x$ is viewed as a subspace of $T_x^*X$ via the isomorphism \eqref{r6nwj6xb}.
\end{prop}

\begin{proof}
Regarding $\ker (d\varphi_p)_x$ as a subspace of $T_x\H_x \s T_x\G$, the statement can be reformulated as
\[
T_xS + \im d\mu_p = T_xX \cap (\ker (d\varphi_p)_x)^\Omega.
\]
We prove this reformulated version below.

Lemma \ref{1hpyffiz} tells us that $$\ker (d\varphi_p)_x \s T_x\H \s (T_xS)^\Omega,$$ so $$T_xS \s T_xX \cap (\ker (d\varphi_p)_x)^\Omega.$$
To show that $\im d\mu_p \s (\ker (d\varphi_p)_x)^\Omega$, let $v \in T_pM$ and $w \in \ker (d\varphi_p)_x$.
By Lemma \ref{kvugo3gu}, $w = X_{\sss^*f}(x)$ for some $f \in C^\infty(X)$ such that $f\big\vert_S = 0$ and $X_f(x) \in T_xS$.
Lemma \ref{gn98u3jv} then shows that $X_{\mu^*f}(p) = d\varphi_p(w) = 0$, so $d(\mu^*f)_p = 0$.
It follows that $$\Omega(d\mu(v), w) = -df(d\sss(d\mu(v))) = -df(d\mu(v)) = 0.$$

It remains only to show that
\[
T_xS^\circ \cap (\im d\mu_p)^\circ \s T_xX^\circ + \Omega(\ker (d \varphi_p)_x),
\]
where all annihilators are taken in $T^*_x\G$.
To this end, let $\xi \in T_xS^\circ \cap (\im d\mu_p)^\circ$ and write $\xi\big\vert_{T_xX} = \Omega(v)\big\vert_{T_xX}$ with $v \in \ker d\ttt_x$.
It then suffices to prove that $v \in T_x\H_x$ and $d\varphi_p(v) = 0$.
We begin by verifying that $v \in T_x\H_x$.
First note that $v \in \ker d\ttt_x \cap (T_xS)^\Omega$.
By Lemma \ref{1hpyffiz}, showing that $d\sss(v) \in T_xS$ would suffice to prove that $v \in T_x\H_x$.
Since $\mu$ is Poisson, we have
\[
d\sss(v) = \sigma(\xi\big\vert_{T_xX}) = d\mu(\omega^{-1}(d\mu^*(\xi\big\vert_{T_xX}))),
\]
and the latter is $0$ since $\xi \in (\im d\mu_p)^\circ$.
It follows that $v \in T_x\H_x$.
To show that $d\varphi_p(v) = 0$, first note that Lemma \ref{kvugo3gu} implies that $v = X_{\sss^*f}(x)$ for some $f \in C^\infty(X)$ such that $f\big\vert_S = 0$ and $X_f(x) \in T_xX$.
But $$df_x = \Omega(v)\big\vert_{T_xX} = \xi\big\vert_{T_xX} \in (\im d\mu_p)^\circ,$$ so $d(\mu^*f)_p = 0$ and hence $d\varphi_p(v) = X_{\mu^*f}(p) = 0$ by Lemma \ref{gn98u3jv}.
\end{proof}

We then deduce the following sufficient conditions for $M \sll{S, \H} \G$ to be smooth.

\begin{thm}[Theorem \ref{Theorem: Smooth category}(ii)]\label{1gjd75z2}
If $\H$ acts freely on $N$, then $S$ is transverse to $\mu$.
Hence, if the action of $\H$ on $N$ is also proper, then $(S, \H)$ is a clean reduction datum and $M \sll{S, \H} \G$ is a symplectic manifold.
\end{thm}

\begin{proof}
This follows from Proposition \ref{pxamcx75} using the fact that $\ker (d\varphi_p)_x = 0$ if $\H$ acts freely.
\end{proof}

Recall that the \dfn{restriction of $\G$ to $S$} is the topological groupoid
\[
\G\big\vert_S \coloneqq \sss^{-1}(S) \cap \ttt^{-1}(S).
\]
The following criterion will be useful for investigating the properness of the action of $\H$ on $N$.

\begin{prop}\label{mbslz011}
If $\G$ acts properly on $M$ and $\H$ is closed in $\G\big\vert_S$, then $\H$ acts properly on $N$.
\end{prop}

\begin{proof}
Since $\H$ is closed in $\G\big\vert_S$, it suffices to show that the action of $\G\big\vert_S$ on $N$ is proper.
By assumption, the map $\theta : \G \prescript{}{\ttt}{\times}_\mu M \longrightarrow M \times M$ defined by \eqref{ff300lau} is proper.

Consider the pullback of $\theta$ by the inclusion $\iota : N \times N \hookrightarrow M \times M$, i.e.\ 
\begin{equation}\label{nxlsrpts}
(\G \prescript{}{\ttt}{\times}_\mu M) \times_{M \times M} (N \times N) \longrightarrow N \times N.
\end{equation}
Since $(\G \prescript{}{\ttt}{\times}_\mu M) \times_{M \times M} (N \times N)$ is closed in $(\G \prescript{}{\ttt}{\times}_\mu M) \times N \times N$ and $\theta$ is proper, \eqref{nxlsrpts} is proper.
Note that
\[
(\G \prescript{}{\ttt}{\times}_\mu M) \times_{M \times M} (N \times N) = (\G\big\vert_N) \prescript{}{\ttt}{\times}_\mu N
\]
and that \eqref{nxlsrpts} is the action map $(\G\big\vert_N) \prescript{}{\ttt}{\times}_\mu N \longrightarrow N \times N$.
This shows that the action of $\G\big\vert_S$ on $N$ is proper, as desired.
\end{proof}

\begin{rem}
The stabilizer subgroupoid $\H$ may not be closed in $\G\big\vert_S$, even when $S$ is closed in $X$.
See Remark \ref{hbxo6c96} for a specific example.
This is in contrast to the case of Hamiltonian group actions, where stabilizer subgroups are always closed.
\end{rem}

\subsection{A Poisson map on the reduced space}\label{Subsection: A Poisson map on the reduced space}
Let $\G \tto X$ be a symplectic groupoid and $S \s X$ a pre-Poisson submanifold.
Note that the distribution on $S$ induced by the orbits of the ssc stabilizer subgroupoid $\G_S \tto S$ (i.e.\ the image of the anchor map) is $TS \cap \sigma(TS^\circ)$.
Hence, if $S / \G_S$ is smooth, it inherits a Poisson structure via Marsden--Ratiu reduction applied to the triple $(X, S, \sigma(TS^\circ))$ \cite[Example D]{marsden-ratiu}.
This holds more generally for any stabilizer subgroupoid, as the following proposition shows.

\begin{prop}\label{9idw3y65}
Consider a symplectic groupoid $\mathcal{G}\tto X$, pre-Poisson submanifold $S \subseteq X$, and stabilizer subgroupoid $\mathcal{H}\tto S$ of $S$ in $\G$.
Suppose that $S / \H$ has the structure of a smooth manifold such that the quotient map $\rho : S \longrightarrow S / \H$ is a smooth submersion.
Let $i : S \longrightarrow X$ be the inclusion map.
Then there is a unique Poisson structure on $S / \H$ such that for any locally defined smooth functions $f, g$ on $S / \H$, and locally defined smooth functions $F, G$ on $X$ such that $\rho^*f = i^*F$, $\rho^*g = i^*G$, and $dF(\sigma(TS^\circ)) = dG(\sigma(TS^\circ)) = 0$, we have
\[
\rho^*\{f, g\}_{\mathcal{S}/\mathcal{H}} = \{F,G\}_X\big\vert_S.
\]
\end{prop}

\begin{proof}
Recall that if $\H$ is source-connected, then $S / \H$ is the Marsden--Ratiu reduction of the triple $(X, S, \sigma(TS^\circ))$ \cite[Example D]{marsden-ratiu}.
We are thereby reduced to proving the following claim: if $f, g, F, G$ are as in the statement, then $\{F, G\}_X$ is $\H$-invariant.
In other words, we want to show that $j^*\sss^*\{F, G\}_X = j^*\ttt^*\{F, G\}_X$, where $j : \H \longrightarrow \G$ is the immersion of $\H$ in $\G$.

Recall from \S\ref{sratijkn} that there is a Poisson transversal $Y$ in $X$ containing $S$ as a coisotropic submanifold.
The restriction $\G\big\vert_Y \coloneqq \sss^{-1}(Y) \cap \ttt^{-1}(Y)$ is then a symplectic subgroupoid of $\G$ \cite[Lemma 7.1]{cattaneo-zambon-2009} and $\H$ is an isotropic Lie subgroupoid of $\G\big\vert_Y$.
By \cite[\S5]{cattaneo-2004}, $\H$ is Lagrangian in $\G\big\vert_Y$.
We therefore have
\begin{equation}\label{51px4zgr}
T_h\H = T_h\H^\Omega \cap T_{j(h)}(\G\big\vert_Y)
\end{equation}
for all $h \in \H$, where $\Omega$ is the symplectic form on $\G$ and $T_h\H$ is identified as subspace of $T_{j(h)}\G$ via the immersion $j : \H \longrightarrow \G$.

Since $TX\big\vert_Y = TY \oplus \sigma(TY^\circ)$, there is a locally defined smooth function $\tilde{F}$ on $X$ such that $\rho^*f = i^*\tilde{F}$ and $d\tilde{F}(\sigma(TY^\circ)) = 0$.
Note that $\sigma(T_xS^\circ) \cap T_xY \s T_xS$ for all $x \in S$ since $S$ is coisotropic in $Y$.
We therefore have
\[
\sigma(T_xS^\circ) = (\sigma(T_xS^\circ) \cap T_xS) \oplus \sigma(T_xY^\circ),
\]
so $d\tilde{F}(\sigma(TS^\circ)) = 0$.
On the other hand, we have $i^*\{\tilde{F}, G\}_X = i^*\{F, G\}_X$ as $i^*\tilde{F} = i^*F$.
We may then assume without loss of generality that $dF(\sigma(TY^\circ)) = 0$, i.e.\ $X_F$ is tangent to $Y$.

We claim that $X_{\sss^*F} - X_{\ttt^*F}$ is tangent to $\H$.
A first observation is that $j^*\sss^*F = j^* \ttt^*F$, since $$j^*\sss^*F = \sss^* i^*F = \sss^* \rho^* f = \ttt^* \rho^*f = \ttt^* i^*F = j^* \ttt^* F.$$
This implies that
\[
\Omega(X_{\sss^*F} - X_{\ttt^*F}, T\H) = (d(\sss^*F) - d(\ttt^*F))(T\H) = 0,
\]
so $X_{\sss^*F} - X_{\ttt^*F}$ takes values in $T\H^\Omega$.
Now recall that $d\sss(X_{\ttt^*F}) = 0 = d\ttt(X_{\sss^*F})$ (see e.g.\ \cite[Theorem 1.6(i)]{mikami-weinstein}), so $d\sss(X_{\sss^*F} - X_{\ttt^*F}) = X_F$ and $d\ttt(X_{\sss^*F} - X_{\ttt^*F}) = X_F$ are tangent to $Y$.
By combining the last two sentences with \eqref{51px4zgr}, we get
\[
(X_{\sss^*F} - X_{\ttt^*F})\big\vert_{j(h)} \in (T_h\H)^\Omega \cap T_{j(h)}(\G\big\vert_Y) = T_h\H
\]
for all $h \in \H$.

The identity $j^* \sss^* G = j^* \ttt^* G$ therefore implies that for all $h \in \H$ and $g \coloneqq j(h)$,
\begin{align*}
(j^*\sss^*\{F, G\}_X)(h) &= \{\sss^*F, \sss^*G\}_X(g) = d(\sss^*G)(X_{\sss^*F}|_g) = d(\sss^*G)(X_{\sss^*F}|_g - X_{\ttt^*F}|_g) \\
&= d(\ttt^*G)(X_{\sss^*F}|_g - X_{\ttt^*F}|_g) = -d(\ttt^*G)(X_{\ttt^*F}|_g) = -\{\ttt^*F, \ttt^*G\}(g) \\
&= (j^*\ttt^*\{F, G\}_X)(h).\qedhere
\end{align*}
\end{proof}

\begin{thm}[Theorem A(iii)]\label{omd5dx3w}
Let $((M, \omega), \G \tto X, \mu)$ be a Hamiltonian system and $(S, \H)$ a clean reduction datum.
Suppose that $S / \H$ has the structure of a smooth manifold such that the quotient map $S \longrightarrow S / \H$ is a smooth submersion.
Then $\mu : M \longrightarrow X$ descends to a Poisson map $M \sll{S} \G \longrightarrow S / \H$.
\end{thm}

\begin{proof}
Let $N \coloneqq \mu^{-1}(S)$ and consider the commutative diagram
\[
\begin{tikzcd}
N
    \arrow[hook]{r}{i}
    \arrow{d}{\pi}
    \arrow{dr}{\tilde{\mu}}
& M \arrow{dr}{\mu} \\
N/\H \arrow{dr}{\bar{\mu}}
& S \arrow[hook]{r}{j} \arrow{d}{\rho}
& X \\
& S/\H.
\end{tikzcd}
\]
Let $f, g$ be locally defined smooth functions on $S / \H$ and $F, G$ local extensions of $\rho^*f, \rho^*g$ as in Proposition \ref{9idw3y65}.
We want to show that $\bar{\mu}^*\{f, g\} = \{\bar{\mu}^*f, \bar{\mu}^*g\}$, or equivalently, $$\pi^*\bar{\mu}^*\{f, g\} = \pi^*\{\bar{\mu}^*f, \bar{\mu}^*g\}.$$
We have $$\pi^*\bar{\mu}^*\{f, g\} = \tilde{\mu}^*\rho^*\{f, g\} = \tilde{\mu}^*j^*\{F, G\} = i^* \mu^*\{F, G\} = i^*\{\mu^*F, \mu^*G\},$$ as $\mu$ is a Poisson map.
Since the Poisson structure on $N / \H$ is also obtained from Poisson reduction, it suffices to show the following:
\begin{enumerate}
\item[(1)] $\mu^*F$ and $\mu^*G$ vanish on $TN^\omega$;
\item[(2)] $\mu^*F$ and $\mu^*G$ restrict to $\pi^*\bar{\mu}^*f$ and $\pi^*\bar{\mu}^*g$ on $N$.
\end{enumerate}
To verify (1), suppose that $v \in TN^\omega$.
We have $v = \omega^{-1}(d\mu^*(\xi))$ for $\xi \in TS^\circ$.
Since $\mu$ is Poisson, we obtain $d\mu(v) = \sigma(\xi) \in \sigma(TS^\circ).$
It follows that $d(\mu^*F)(v) = dF(d\mu(v)) = 0.$
To show (2), note that $i^*\mu^*F = \tilde{\mu}^*j^*F = \tilde{\mu}^*\rho^*f = \pi^*\bar{\mu}^*f.$
One has analogous identities for $g$ and $G$.
\end{proof}

\begin{rem}
If $TS \cap \sigma(T^*X) \s \sigma(TS^\circ)$, then the Poisson structure on $S / \H$ is trivial.
The Poisson map $\overline{\mu}$ then amounts to an ``integrable system'' on $M \sll{S} \G$.
\end{rem}

\subsection{Examples}\label{Subsection: Smooth examples}

Let $((M, \omega), \G \tto X, \mu)$ be a Hamiltonian system.
We observe the following special cases of Theorem \ref{Theorem: Smooth category}.

\begin{ex}[Singletons]
If $S = \{x\}$ is a singleton in $X$, then the isotropy group at $x$ is a stabilizer subgroupoid of $S$ in $\mathcal{G}$.
We thereby recover Mikami--Weinstein reduction \cite[Theorem 3.12]{mikami-weinstein}, of which the more classical Marsden--Weinstein--Meyer reduction \cite{marsden-weinstein,meyer} is a special case.
\end{ex}

\begin{ex}[Poisson transversals]\label{6sap623o}
If $S \s X$ is Poisson transversal, then $S$ is transverse to $\mu$ \cite[Lemma 7(1)]{frejlich-marcut}.
We also have $L_S=0$ in this case, so that the trivial groupoid over $S$ is a stabilizer subgroupoid.
It follows that $M \sll{S} \G = \mu^{-1}(S)$, and we recover the fact that $\mu^{-1}(S)$ is a symplectic submanifold of $M$.
\end{ex}

%\begin{ex}\label{osz3u5dd}
%Let $((M, \omega), \G \tto X, \mu)$ be a Hamiltonian system.
%If $S$ is a Poisson submanifold of $(X, \sigma)$, then $\sigma(TS^\circ) = 0$.
%The stabilizer subalgebroid of $S$ is therefore $L_S = TS^\circ$ with the trivial anchor map $L_S \longrightarrow TS$.
%It follows that $L_S$ is a bundle of Lie algebras.
%The source-connected, source-simply-connected stabilizer subgroupoid of $S$ is then the bundle of Lie groups $\G_S \coloneqq \bigcup_{x \in S} H_x,$ where $H_x$ is the simply-connected Lie group with Lie algebra $(L_S)_x$ \pcnote{Have we defined this notation? This last sentence is also a bit too heavily partitioned.}.
%Since the anchor map of $L_S$ is trivial, we have $S / \G_S = S$, so \pcnote{New sentence: The Poisson map ... from ... therefore takes the form ...} the Poisson map from Theorem \ref{main-theorem-smooth}(iii) is $M \sll{S} \G \longrightarrow S$.
%\end{ex}

\begin{ex}[General symplectic reduction]
\label{ty6jqjxq}
Let $(M, \omega)$ be a symplectic manifold.
The pair groupoid $\G \coloneqq M \times M^-$ is then a symplectic groupoid over $(M, \omega)$, with source map $\sss(p, q) = p$, target map $\ttt(p, q) = q$, and multiplication $(p, q)(q, r) = (p, r)$.
The symplectic groupoid $\G$ then acts on $(M, \omega)$ in a Hamiltonian way with moment map the identity map and action $(p, q) \cdot q = p$.

A submanifold $i : S \hookrightarrow M$ is pre-Poisson if and only if it is pre-symplectic.
In this case, one has $M \sll{S} \G = S / {\sim}$, where $\sim$ is the equivalence relation generated by the leaves of the distribution $\ker i^*\omega$.
We have therefore recovered the procedure of reduction along a pre-symplectic submanifold (Remark \ref{general-reduction-remark}).
\end{ex}

\section{Main construction: complex analytic version}
\label{phj103sm}

This section is concerned with Theorem \ref{Theorem: Complex analytic category} and its proof.
We develop a complex analytic version of Marsden--Ratiu reduction \cite{marsden-ratiu} in \S\ref{Subsection: Complex analytic Marsden--Ratiu reduction}, and examine its relation to complex pre-Poisson submanifolds in \S\ref{Subsection: Pre-Poisson complex submanifolds}. This creates the context needed to prove Theorem \ref{Theorem: Complex analytic category}. Parts (i)--(v) of this theorem are proved in \S\ref{Subsection: Complex analytic symplectic reduction along a submanifold}, while Parts (vi)--(vii) are proved in \S\ref{Subsection: A Poisson map}.

\subsection{Complex analytic Marsden--Ratiu reduction}\label{Subsection: Complex analytic Marsden--Ratiu reduction} 
We begin by observing that the notion of Poisson reduction introduced by Marsden--Ratiu \cite{marsden-ratiu} adapts to the complex analytic setting.

Let us begin with a brief digression on terminology.
A \dfn{holomorphic Poisson structure} on a complex analytic space $(P, \O_P)$ is an enrichment of the structure sheaf $\mathcal{O}_P$ to a sheaf of Poisson algebras.
In this case, we call $(P, \O_P)$ a \dfn{complex analytic Poisson space}.
We also often identity a Poisson structure with the corresponding \dfn{homomorphism Poisson bivector field}, i.e.\ the bundle homomorphism
\[
\tau : T^*P \longrightarrow TP,\quad df_p \mtoo X_f(p)
\]
for all $p \in P$ and $f \in \O_{P, p}$, where $X_f \coloneqq \{f, \cdot\}$ is the Hamiltonian vector field associated to $f$.
We call $P$ \dfn{symplectic} if it is smooth and $\tau = \omega^{-1}$ for some holomorphic symplectic form $\omega : TP \longrightarrow T^*P$.

Let $E$ be a holomorphic vector bundle over a complex analytic space $P$, and write $E_x$ for the fibre of $E$ over $x\in X$.
Our convention is that a \dfn{subbundle} of $E$ is simply a complex analytic subspace $R$ of $E$ whose fibres $R_x \coloneqq R \cap E_x$ are vector subspaces of $E_x$ for all $x \in P$.
Note that the fibres of $R$ may have different dimensions, so that $R$ is not necessarily a vector bundle.
For a complex analytic subspace $Y \s P$ and a subbundle $R$ of $TP\big\vert_Y$, we write $\O_P^R$ for the subsheaf of $\O_P$ defined by
\[
\O_P^R(U) \coloneqq \{f \in \O_P(U) : df(R) = 0\}
\]
for all open subsets $U \subseteq P$.
The germ of a holomorphic function $f \in \O_P(U)$ at a point $p \in U$ is denoted $f_p \in \O_{P, p}$.

The following result is a complex analytic analogue of \cite[Theorem 2.2]{marsden-ratiu}.

\begin{thm}[Complex analytic Marsden--Ratiu reduction]\label{2k4asf2u}
Let $(M, \tau)$ be a complex analytic Poisson space, $N \s M$ a reduced complex analytic subspace, $E \longrightarrow N$ a subbundle of $TM\big\vert_N$, and $D \coloneqq TN \cap E$.
Suppose that
\begin{enumerate}
\item[\textup{(1)}] $\O_M^E$ is closed under Poisson bracket,
\item[\textup{(2)}] $\tau(E^\circ) \s TN + E$, and
\item[\textup{(3)}] for all $p \in N$ and $f \in \O_{N, p}^D$, there exists $F \in \O_{M, p}^E$ such that $F\big\vert_N = f$.
\end{enumerate}
One then has a unique Poisson bracket $\{\cdot, \cdot\}'$ on $\O_N^D$ satisfying
\begin{equation}\label{jup0hvqk}
\{f, g\}' = \{F, G\}\big\vert_N
\end{equation}
for all $p \in N$, $f, g \in \O_{N, p}^D$, and $F, G \in \O_{M, p}^E$ related to $f, g$ as in \textup{(3)}.
\end{thm}

\begin{proof}
To see that \eqref{jup0hvqk} does not depend on the choice of $F$ and $G$, we need to show the following: if $p \in N$ and $F, G \in \O_{M, p}^E$ are such that $F\big\vert_N = 0$, then $\{F, G\}\big\vert_N = 0$.
Note that $X_F\big\vert_N$ takes values in $\tau(TN^\circ \cap E^\circ)$, while Condition (2) implies that $$TN^\circ \cap E^\circ = (TN + E)^\circ \s \tau(E^\circ)^\circ = \tau^{-1}(E).$$ It follows that $\tau(TN^\circ \cap E^\circ) \s E$, and hence $\{F, G\}\big\vert_N = dG(X_F\big\vert_N) = 0.$

Let $f, g \in \O_N^D(U)$ for some open set $U \s N$.
We then get a well-defined element $\{f, g\}' \in \O_N(U)$ defined by the collection of germs $\{F, G\}\big\vert_N \in \O_{N, p}$ for all $p \in U$ and $F, G \in \O_{M, p}^E$ satisfying $F\big\vert_N = f_p$ and $G\big\vert_N = g_p$.
We also have $\{f, g\}' \in \O_N^D(U)$, since
\[
d\{f, g\}'(D) = d\{F, G\}(D) \s d\{F, G\}(E) = 0
\]
by Condition (1).

The Leibniz identity follows from the fact that $\{f, gh\}' = \{F, GH\}\big\vert_N$ if $f, g, h \in \O_{N, p}$ are related to $F, G, H \in \O_{M, p}$ as in Condition (3), which holds since $d(GH)$ also vanishes on $E$.
For the Jacobi identity, note that $\{f, \{g, h\}'\}' = \{F, \{G, H\}\}\big\vert_N$, since $\{g, h\}' = \{G, H\}\big\vert_N$ by definition and $d\{G, H\}(E) = 0$ by Condition (1).
\end{proof}

Condition (3) is implicit in Marsden--Ratiu reduction \cite[Theorem 2.2]{marsden-ratiu}, as it always holds in the context of smooth manifolds.
On the other hand, the same argument works for non-singular complex analytic spaces:

\begin{lem}\label{961mx9ov}
Let $M$ be a complex manifold, $N \s M$ a complex submanifold, and $E \longrightarrow N$ a holomorphic subbundle of $TM\big\vert_N$.
Assume that $E$ and $D \coloneqq TN \cap E$ are vector bundles.
For all $p \in N$ and $f \in \O_{N, p}^D$, there exists $F \in \O_{M, p}^E$ such that $F\big\vert_N = f$.
\end{lem}

\begin{proof}
Note that this is a local statement. We may therefore assume that $E = D \oplus R$ for some holomorphic subbundle $R$ of $TM\big\vert_N$, and then extend $R$ to a holomorphic subbundle $S$ containing $R$ such that $TM\big\vert_N = TN \oplus S$.
Then there exists a tubular neighbourhood of $N$ in $S$ that is biholomorphic to a neighbourhood of $p$ in $M$.
We may define $F = f \circ \pi$, where $\pi : S \longrightarrow N$ is the bundle map.
\end{proof}

The following example of complex analytic Marsden--Ratiu reduction will be important in the remainder of this section.

\begin{prop}[{cf.\ \cite[Example D]{marsden-ratiu}}]\label{lmgqb7xq}
Let $(M, \tau)$ be a complex analytic Poisson space, $N \s M$ a reduced complex analytic subspace, and $E \coloneqq \tau(TN^\circ)$.
\begin{enumerate}
\item[\textup{(i)}]
Conditions \textup{(1)} and \textup{(2)} in \textup{Theorem \ref{2k4asf2u}} are satisfied.
\item[\textup{(ii)}]
If $M$ and $N$ are smooth and $E$ and $D \coloneqq TN \cap E$ are vector bundles, then Condition \textup{(3)} is also satisfied.
\end{enumerate}
\end{prop}

\begin{proof}
Note that since $E^\circ = \tau(TN^\circ)^\circ = \tau^{-1}(TN),$ a local holomorphic function $F$ on $M$ satisfies $dF(E) = 0$ if and only if $X_F$ is tangent to $N$.
Condition (1) therefore follows from the Jacobi identity of $\tau$, which is equivalent to $X_{\{F, G\}} = [X_F, X_G]$ for all holomorphic functions $F$ and $G$.
Condition (2) also holds since $\tau(E^\circ) = \tau(\tau^{-1}(TN)) \s TN$.
Part (ii) holds by Lemma \ref{961mx9ov}.
\end{proof}

\subsection{Pre-Poisson complex submanifolds}\label{Subsection: Pre-Poisson complex submanifolds}

As in \S\ref{smooth-case}, a complex submanifold $S$ of a holomorphic Poisson manifold $(X, \sigma)$ is called \dfn{pre-Poisson} if $\sigma^{-1}(TS) \cap TS^\circ$ has constant rank over $S$.

We now discuss a special case of complex analytic Marsden--Ratiu reduction (Theorem \ref{2k4asf2u}) obtained from a Poisson map $\mu : M \longrightarrow X$ and a pre-Poisson submanifold $S \s X$.
Despite $\mu^{-1}(S)$ possibly being singular, we will show that the conditions of Theorem \ref{2k4asf2u} are satisfied with respect to $E \coloneqq \tau(TN^\circ)$.
This will be the basis for our notion of symplectic reduction along a submanifold, where $\mu$ will be the moment map of a Hamiltonian system.

To this end, we first recall some of the results of Cattaneo--Zambon \cite{cattaneo-zambon-2009} and adapt them to the complex analytic setting.
Let $(X, \sigma)$ be a holomorphic Poisson manifold.
As in the $C^\infty$ case (\S\ref{sratijkn}), any Poisson transversal complex submanifold $Y \s X$ inherits a canonical holomorphic Poisson structure \cite[Lemma 3]{frejlich-marcut}.

\begin{prop}[{cf.\ \cite[Theorem 3.3]{cattaneo-zambon-2009}}]
\label{i9vtsb2f}\label{Proposition: Transversal germ}
Let $S\subseteq X$ be a pre-Poisson complex submanifold of a holomorphic Poisson manifold $(X, \sigma)$.
For all $x \in S$, there is a Poisson transversal complex submanifold $Y \s X$ which contains a neighbourhood of $x$ in $S$ as a coisotropic submanifold.
\end{prop}

\begin{proof}
We simply redo the proof of \cite[Theorem 3.3]{cattaneo-zambon-2009} locally.
Let $S_x$ be the germ of $S$ at $x$.
Choose a subbundle $R$ of $TX\big\vert_{S_x}$ such that $$TX\big\vert_{S_x} = (TS_x + \sigma(TS_x^\circ)) \oplus R,$$ and let $Y$ be a submanifold germ of $X$ containing $S_x$ with $TY\big\vert_{S_x} = TS_x \oplus R$.
Note that the latter is possible since we can choose a tubular neighbourhood of $S$ near $x$.
\end{proof}

\begin{rem}\label{Remark: Stein}
If $X$ is Stein and $S$ is closed, then there is a Poisson transversal $Y$ containing $S$ globally.
To see this, recall that the positive-degree cohomologies of coherent sheaves on Stein manifolds are trivial.
Every short exact sequence of vector bundles therefore splits, so that we may choose a global $R$ in the proof.
Since the tubular neighbourhood theorem holds for Stein manifolds (see e.g.\ \cite[Theorem 3.3.3]{forstneric}), we can choose a global $Y$ with $TY\big\vert_S = TS \oplus R$.
\end{rem}

The previous result is the main ingredient used to obtain Condition (3) in complex analytic Marsden--Ratiu reduction.

\begin{prop}\label{8fzyp3dy}\label{dldijdjx}
Let $\mu : (M, \tau) \longrightarrow (X, \sigma)$ be a holomorphic Poisson map between holomorphic Poisson manifolds.
Let $S \s X$ be a pre-Poisson complex submanifold such that $N \coloneqq \mu^{-1}(S)$ is reduced, let $E \coloneqq \tau(TN^\circ)$, and let $D \coloneqq TN \cap E$.
Conditions \textup{(1)--(3)} in \textup{Theorem \ref{2k4asf2u}} are then satisfied.
In particular, for all $p \in N$ and $f \in \O_{N, p}^D$, there exists $F \in \O_{M, p}^E$ such that $F\big\vert_N = f$.
\end{prop}

\begin{proof}
Conditions (1) and (2) hold by Proposition \ref{lmgqb7xq}(i), so it suffices to verify Condition (3).
We first observe that this third condition is local. 
Proposition \ref{i9vtsb2f} therefore allows us to assume that $S \s Y \s X$, where $Y$ is Poisson transversal in $X$ and $S$ is coiostropic in $Y$.
By \cite[Lemma 7]{frejlich-marcut}, $Y$ is transverse to $\mu$, $P \coloneqq \mu^{-1}(Y)$ is a Poisson transversal submanifold of $M$, and the restriction map $P \longrightarrow Y$ is Poisson\footnote{This was originally stated for smooth manifolds in \cite{frejlich-marcut}, but the same proof works in the holomorphic setting.}.
Since the preimage of a coisotropic subspace under a Poisson map is coisotropic, $N=\mu^{-1}(S)$ is coisotropic in $P$.

Now let $p \in N$ and $f \in \O_{N, p}^D$.
Extend $f$ to $\tilde{f} \in \O_{P, p}$.
Since $P$ is smooth and $TP \cap \tau(TP^\circ) = 0$, Lemma \ref{961mx9ov} implies that there exists $F \in \O_{M, p}$ such that $F\big\vert_P = \tilde{f}\big\vert_P$ and $dF(\tau(TP^\circ)) = 0$.
We also have that $TP\big\vert_N \cap E \s TN$ as $N$ is coisotropic in $P$, so $E = D \oplus \tau(TP^\circ)\big\vert_N.$
It follows that
\[
dF(E) = dF(D \oplus \tau(TP^\circ)\big\vert_N) = dF(D) = df(D) = 0.\qedhere
\]
\end{proof}

\subsection{Complex analytic symplectic reduction along a submanifold}\label{Subsection: Complex analytic symplectic reduction along a submanifold} 

We now state and prove the main result of this section --- a complex analytic version of symplectic reduction along a submanifold.

We begin with a few definitions.
A \dfn{holomorphic Hamiltonian system} is a tuple $((M, \omega), \G \tto X, \mu)$, where $M$ is a complex manifold endowed with a holomorphic symplectic form $\omega$, and $\G \tto X$ is a holomorphic symplectic groupoid acting on $M$ in a Hamiltonian way with moment map $\mu : M \longrightarrow X$; the definition is the same as in \S\ref{smooth-case}, except that ``smooth'' has become ``holomorphic''.

In particular, the base $X$ of a holomorphic Hamiltonian system inherits a canonical holomorphic Poisson bivector field $\sigma : T^*X \longrightarrow TX$.
For a pre-Poisson complex submanifold $S \s X$, the \dfn{stabilizer subalgebroid}
\[
L_S \coloneqq \sigma^{-1}(TS) \cap TS^\circ
\]
is a holomorphic Lie subalgebroid of $T^*X = \mathrm{Lie}(\G)$.
By Laurent-Gengoux--Sti\'{e}non--Xu's work on the integration of holomorphic Lie algebroids \cite{LGSX-integration}, the source-connected, source-simply-connected Lie subgroupoid $\G_S \tto S$ of $\G$ with Lie algebroid $L_S$ is a holomorphic Lie groupoid and the immersion $\G_S \longrightarrow \G$ is holomorphic \cite[Theorem 3.17 and Proposition 3.20]{LGSX-integration}.
The subgroupoid $\G_S$ is also isotropic in $\G$ by \cite[Proposition 7.2]{cattaneo-zambon-2009}.
More generally, we call any isotropic holomorphic subgroupoid $\H \tto S$ of $\G$ with Lie algebroid $L_S$ a \dfn{holomorphic stabilizer subgroupoid of $S$ in $\G$}.

Let $((M, \omega), \G \tto X, \mu)$ be a holomorphic Hamiltonian system, $S \s X$ a pre-Poisson complex submanifold, and $\H \tto S$ a holomorphic stabilizer subgroupoid of $S$ in $\G$.
Let $N \coloneqq \mu^{-1}(S) \subseteq M$ be the complex analytic space defined by the ideal generated by all elements $\mu^*f$, where $f$ is in the ideal of $S\subseteq X$.
One then has a holomorphic action of $\H \tto S$ on $N$. 
Given any point $p \in N$, we have a holomorphic orbit map
\[
\varphi_p : \H_{\mu(p)} \too N, \quad h \mtoo h \cdot p.
\]
This enables us to define
\[
T_p(\H \cdot p) \coloneqq \im d\varphi_p
\]
as a vector subspace of $T_pN$.

\begin{prop}\label{kwi1nwnj}
We have
\[
T_p(\H \cdot p) = T_pN \cap T_pN^\omega
\]
for all $p \in N$.
\end{prop}

\begin{proof}
By the definition of $N$, we have $T_pN = d\mu_p^{-1}(T_{\mu(p)}S).$
This allows us to repeat the proof of Theorem \ref{distribution-proposition} verbatim in the present context.
\end{proof}

For an $\mathcal{H}$-invariant open subset $U \s N$, we let $\O_N(U)^\H$ be the set of $\H$-invariant elements of $\O_N(U)$.
The next result gives a Poisson structure on $\O_N(U)^\H$; it is the first step towards the construction of a Poisson structure on a quotient of $N$ by $\H$.

\begin{thm}\label{zogb0t7j}
Let $((M, \omega), \G \tto X, \mu)$ be a holomorphic Hamiltonian system, $S \s X$ a pre-Poisson complex submanifold such that $N \coloneqq \mu^{-1}(S)$ is reduced, and $\H \tto S$ a holomorphic stabilizer subgroupoid of $S$ in $\G$.
For every $\H$-invariant open set $U \s N$, there is a unique Poisson bracket
\[
\{\cdot, \cdot\}' : \O_N(U)^\H \times  \O_N(U)^\H \too  \O_N(U)^\H
\]
with the following property: for all $f, g \in \O_N(U)^\H$ and $p \in U$, the germ of $\{f, g\}'$ at $p$ is $\{F, G\}\big\vert_N$, where $F, G \in \O_{M, p}$ are any extensions of the germs of $f, g$ at $p$ such that $dF(TN^\omega) = dG(TN^\omega) = 0$.
\end{thm}

\begin{proof}
Proposition \ref{kwi1nwnj} implies that $\O_N(U)^\H \s \O_N^D(U)$, where $D \coloneqq TN \cap TN^\omega$.
Since the moment map $\mu : M \longrightarrow X$ is Poisson, Proposition \ref{dldijdjx} shows that the complex analytic version of the Marsden--Ratiu theorem (Theorem \ref{2k4asf2u}) holds, producing a Poisson bracket
\[
\{\cdot, \cdot\}' : \O_N^D(U) \times \O_N^D(U) \too \O_N^D(U).
\]
It suffices to show that $\O_N(U)^\H$ is closed under $\{\cdot, \cdot\}'$.

Let $f, g \in \O_N(U)^\H$, let $p \in U$, and let $h \in \H$ be such that $\ttt(h) = \mu(p)$.
We want to show that $$\{f, g\}(h \cdot p) = \{f, g\}(p).$$
In other words, we want to show that for any $F, G \in \O_{M, p}^{TN^\omega}$ satisfying $F\big\vert_N = f_p$, $G\big\vert_N = g_p$, and any $\tilde{F}, \tilde{G} \in \O_{M, h \cdot p}^{TN^\omega}$ satisfying $\tilde{F}\big\vert_N = f_{h \cdot p}$, $\tilde{G}\big\vert_N = g_{h \cdot p}$, we have
\begin{equation}\label{0r2e2otz}
\{F, G\}(p) = \{\tilde{F}, \tilde{G}\}(h \cdot p).
\end{equation}
Since $\ttt$ is a submersion, we can find a local holomorphic section
\[
\rho : V \s S \longrightarrow \H
\]
of $\ttt$ passing through $h$, i.e.\ $V$ is a neighbourhood of $\ttt(h)$ in $S$, $\ttt \circ \rho = \mathrm{Id}$, and $\rho(\ttt(h)) = h$.
We then get a biholomorphism
\[
\psi_\rho : N \cap \mu^{-1}(V) \longrightarrow N, \quad \psi_\rho(q) = \rho(\mu(q)) \cdot q,
\]
such that $\psi_\rho(p) = h \cdot p$.

To show \eqref{0r2e2otz}, we first claim that 
\begin{equation}\label{hsj792ou}
\psi_\rho^*i^*\omega = i^*\omega,
\end{equation}
where $i : N \longrightarrow M$ is the inclusion map.
Note that for all $q \in N \cap \mu^{-1}(V)$, $(\rho(\mu(q)), q, \psi_\rho(q))$ is in the graph $\Gamma$ of the action of $\G$ on $M$, which is Lagrangian in $\G \times M \times M^-$ by definition\footnote{We are implicitly identifying $\H$ as a subset of $\G$ for simplicity of notation; this is only true locally.}.
Hence, for all $u, v \in T_qN$, we have
\[
(d\rho(d\mu(u)), u, d\psi_\rho(u)) \in T\Gamma,
\]
and the analogous statement is true for $v$.
It follows that
\[
\Omega(d\rho(d\mu(u)), d\rho(d\mu(v))) + \omega(u, v) - \omega(d\psi_\rho(u), d\psi_\rho(v)) = 0.
\]
But $\H$ is isotropic in $\G$, so the first term vanishes.
This proves \eqref{hsj792ou}.

Now note that $X_F$ and $X_G$ are tangent to $N$, as one has $$\omega(X_F, TN^\omega) = dF(TN^\omega) = 0$$ and an analogous statement for $X_G$.
In particular, we may evaluate $d\psi_\rho$ at $X_F|_p$.
We claim that
\begin{equation}\label{0c3jvlj8}
d\psi_\rho(X_F|_p) - X_{\tilde{F}}|_{\psi_\rho(p)} \in T_{\psi_\rho(p)}N^\omega.
\end{equation}
To see this, let $v \in T_{\psi_\rho(p)}N$ and write $v = d\psi_\rho(u)$ for $u \in T_pN$.
Then by \eqref{hsj792ou}, we have $$\omega(d\psi_\rho(X_F|_p), v) = \psi_\rho^*\omega(X_F|_p, u) = \omega(X_F, u) = dF(u) = df(u).$$
But $f$ is $\H$-invariant, so $$df(u) = d(\psi_\rho^*f)(u) = df(v) = d\tilde{F}(v) = \omega(X_{\tilde{F}}|_{\psi_\rho(p)}, v).$$
This proves \eqref{0c3jvlj8}.

At the same time, since $d\tilde{G}(TN^\omega) = 0$, \eqref{0c3jvlj8} shows that
\begin{align*}
\{\tilde{F}, \tilde{G}\}(h \cdot p) &= d\tilde{G}(X_{\tilde{F}}|_{h \cdot p}) = d\tilde{G}(d\psi_\rho(X_F|_p)) = dg(d\psi_\rho(X_F|_p)) = dg(X_F|_p) = dG(X_F|_p) \\
&= \{F, G\}(p).
\end{align*}
This proves \eqref{0r2e2otz} and hence $\O_N(U)^\H$ is closed under $\{\cdot, \cdot\}'$.
\end{proof}

To discuss the quotient of $N$ by $\H$ in the most general context, we make the following definition.

\begin{defn}
Let $\H$ be a holomorphic Lie groupoid acting holomorphically on a complex analytic space $(N, \O_N)$.
A \dfn{complex analytic quotient} of $N$ by $\H$ is a complex analytic space $(Q, \O_Q)$ together with an $\H$-invariant holomorphic map $\pi : N \longrightarrow Q$ such that
\[
\pi^* : \O_Q \too (\pi_*\O_N)^\H
\]
is a sheaf isomorphism.
\end{defn}

Theorem \ref{zogb0t7j} then has the following consequence.

\begin{thm}[Theorem \ref{Theorem: Complex analytic category}(i)---(iv)]\label{ytazg95b}
Let $((M, \omega), \G \tto X, \mu)$ be a holomorphic Hamiltonian system, $S \s X$ a pre-Poisson complex submanifold, and $\H \tto S$ a holomorphic stabilizer subgroupoid of $S$ in $\G$.
Suppose that $N \coloneqq \mu^{-1}(S)$ is reduced and that there is a complex analytic quotient of $N$ by $\H$, denoted $\pi : N \longrightarrow Q$.

\begin{enumerate}
\item[\textup{(i)}]
For all $p \in N$ and $f \in \O_{Q, \pi(p)}$, there exists $F \in \O_{M, p}$ such that $\pi^*f = F\big\vert_N \in \O_{N, p}$ and $dF(TN^\omega) = 0$.

\item[\textup{(ii)}]
There is a unique holomorphic Poisson structure $\{\cdot, \cdot\}_Q$ on $Q$ such that 
\[
\pi^*\{f, g\}_Q = \{F, G\}_M\big\vert_N
\]
for all $f, g \in \O_{Q, \pi(p)}$ and $F, G \in \O_{M, p}$ related to $f, g$ as in \textup{Part (i)}.\qed

\item[\textup{(iii)}]
Assume that there exists $p \in N$ such that $d\pi_p$ is surjective, $\pi^{-1}(\pi(p))$ is an $\mathcal{H}$-orbit, and $\pi(p)$ is a smooth point of $Q$.
The Poisson structure in \textup{Part (ii)} is then symplectic on a neighbourhood of $\pi(p)$.

\item[\textup{(iv)}]
Suppose that $S$ intersects $\mu$ cleanly and that $Q = N/\H$ has a complex-manifold structure such that $\pi : N \longrightarrow Q$ is a holomorphic submersion.
The Poisson structure on $Q$ is then induced by a holomorphic symplectic form $\bar{\omega}$ satisfying $\pi^*\bar{\omega} = i^*\omega$, where $i : N\longrightarrow M$ is the inclusion map.

\end{enumerate}
\end{thm}

\begin{proof}
Part (i) follows from Proposition \ref{8fzyp3dy}, and
Part (ii) follows from Theorem \ref{zogb0t7j} using the isomorphisms $\pi^* : \O_Q(U) \cong \O_N(\pi^{-1}(U))^\H$.

To prove Part (iii), let $x \coloneqq \pi(p)$ and let $\tau : T_x^*Q \longrightarrow T_xQ$ be the Poisson structure at $x$.
It suffices to show that $\tau$ has an inverse $\bar{\omega}_x : T_xQ \longrightarrow T_x^*Q$.

Since $\pi^{-1}(x) = \H \cdot p$, Proposition \ref{kwi1nwnj} shows that $\ker d\pi_p = T_pN \cap T_pN^\omega$, so we have a short exact sequence
\[
0 \longrightarrow T_pN \cap T_pN^\omega \longrightarrow T_pN \longrightarrow T_xQ \longrightarrow 0.
\]
It follows that the symplectic form $\omega_x : T_xM \longrightarrow T_x^*M$ descends to a map $\bar{\omega}_x : T_xQ \longrightarrow T_x^*Q$.
To show that $\bar{\omega}_x$ is an inverse of $\tau$, let $f \in \O_{Q, x}$ and $F \in \O_{M, p}$ be such that $\pi^*f = F\big\vert_N$ and $dF(TN^\omega) = 0$.
By the definition of $\tau$, we have $\tau(df_x) = d\pi(X_F|_p)$.
Hence, for all $v \in T_pN$, we have $$\bar{\omega}_x(\tau(df_x))(d\pi_p(v)) = \omega(X_F|_p, v) = dF_p(v) = df_x(d\pi_p(v)).$$
This shows that $\bar{\omega}_x \circ \tau = \mathrm{Id}$, so that $\bar{\omega}_x$ is indeed an inverse of $\tau$. 

Part (iv) follows immediately from Part (iii).
\end{proof}

%\begin{rem}
%The characterization properties in Theorem \ref{ytazg95b} show that the Poisson structure is a complex analytic version of Marsden--Ratiu reduction of Poisson manifolds \cite{marsden-ratiu}.
%On the other hand, since $N$ is in general singular, the existence of functions $F$ and $G$ satisfying (a) and (b) is not immediate, and the proof makes essential use of the fact that $N = \mu^{-1}(S)$ for some Poisson map $\mu$ and pre-Poisson submanifold $S$.
%\end{rem}

\begin{defn}\label{Definition: Complex analytic reduction}
The complex analytic Poisson space $(Q, \O_Q)$ in Theorem \ref{ytazg95b} is called the \dfn{symplectic reduction of $M$ by $\G$ along $S$ with respect to $\H$ and $\pi$} and denoted $M \sll{S, \H, \pi} \G$.
The pair $(S, \H)$ is called a \dfn{reduction datum}; it is called a \dfn{clean reduction datum} if the conditions of Theorem \ref{ytazg95b}(iv) are satisfied, in which case
\[
M \sll{S, \H} \G \coloneqq M \sll{S, \H, \pi} \G = \mu^{-1}(S) / \H
\]
is a complex symplectic manifold.
If $(S, \H)$ is clean and $\H$ is source-connected, we use the simplified notation
\[
M \sll{S} \G \coloneqq M \sll{S, \H} \G,
\]
and call it the \dfn{symplectic reduction of $M$ by $\G$ along $S$}.
\end{defn}

\begin{rem}\label{Remark: Source-connected}
We remind the reader that any two choices of source-connected stabilizer subgroupoid $\mathcal{H}\tto S$ yield the same reduced space $M \sll{S, \H} \G$. This justifies our notation $M \sll{S} \G$.
\end{rem}

\begin{rem}
The reducedness assumption on $N$ is not always satisfied.
To see this, consider the standard representation of $\SLn(2, \C)$ on $\C^2$ and its lift to a Hamiltonian action on $T^*\C^2 = \C^4$.
In coordinates $(z_1, z_2, w_1, w_2)$, the components of the moment map are $\mu_1 = z_1 w_2$, $\mu_2 = z_2 w_1$, and $\mu_3 = z_1 w_1 - z_2 w_2$ (see e.g.\ \cite[Example 2.6]{herbig-schwarz-seaton}).
Note that $f \coloneqq z_1w_1 + z_2w_2$ is not in the ideal generated by $\mu$, but that $f^2 = \mu_3^2 + 4\mu_1\mu_2$ belongs to this ideal.
\end{rem}

The next result shows that the sufficient condition for smoothness discussed in \S\ref{jp33sbh7} adapts to the complex analytic setting.

\begin{prop}[Theorem \ref{Theorem: Complex analytic category}(v)]\label{ugr3y1lq}
Let $((M, \omega), \G \tto X, \mu)$ be a holomorphic Hamiltonian system, $S \s X$ a pre-Poisson complex submanifold, and $\H \tto S$ a holomorphic stabilizer subgroupoid of $S$ in $\G$.
If $\H$ acts freely on $N \coloneqq \mu^{-1}(S)$, then $S$ is transverse to $\mu$.
If the $\H$-action is also proper, then $(S, \H)$ is a clean reduction datum and $M \sll{S, \H} \G$ is a complex symplectic manifold.
\end{prop}

\begin{proof}
The proof of Theorem \ref{1gjd75z2} adapts to the complex analytic category, showing that $S$ is transverse to $\mu$.
Hence, $N$ is a complex submanifold of $M$ on which $\H$ acts freely and properly.
By Godement's criterion (see e.g.\ \cite[Part II, Chapter III, Theorem 2 on p.\ 92]{serre-1992}), $N / \H$ has a unique complex manifold structure such that the quotient map $\pi : N \longrightarrow N / \H$ is a holomorphic submersion.
This shows that $(S, \H)$ is a clean reduction datum. We may therefore apply Theorem \ref{ytazg95b}(iv) to complete the proof.
\end{proof}

\subsection{A Poisson map on the reduced space}\label{Subsection: A Poisson map}
We now discuss a complex analytic version of Theorem \ref{Theorem: Smooth category}(iii).
To this end, we first discuss the existence of a Poisson structure on a complex analytic quotient of $S$ by $\H$.

\begin{prop}[Theorem \ref{Theorem: Complex analytic category}(vi)]\label{45y1ghc8}
Let $\G \tto X$ be a holomorphic symplectic groupoid, $S \s X$ a pre-Poisson complex submanifold, and $\H \tto S$ a holomorphic stabilizer subgroupoid of $S$ in $\G$.
Suppose that there is a complex analytic quotient of $S$ by $\H$, denoted $\rho : S \longrightarrow R$.
Then there is a unique holomorphic Poisson structure on $R$ such that 
\[
\rho^*\{f, g\}_R = \{F, G\}_X\big\vert_S
\]
for all $x \in S$, $f, g \in \O_{R, \rho(x)}$, and $F, G \in \O_{X, x}^{\sigma(TS^\circ)}$ satisfying $\rho^*f = F\big\vert_S$, $\rho^*g = G\big\vert_S$.
\end{prop}

\begin{proof}
Proposition \ref{dldijdjx} gives us a Poisson bracket on $\O_S^D$, where $D \coloneqq TS \cap \sigma(TS^\circ)$.
To show that this gives a Poisson bracket on $\O_{R} = (\rho_*\O_S)^\H$, it suffices to show that $\H$-invariant functions in $\O_S^D$ are closed under the Poisson bracket.
This follows from the same reasoning as in the proof of Proposition \ref{9idw3y65}, using Proposition \ref{i9vtsb2f} in the present case.
\end{proof}

\begin{thm}[Theorem \ref{Theorem: Complex analytic category}(vii)]\label{6hr3szsv}
Let $((M, \omega), \G \tto X, \mu)$ be a holomorphic Hamiltonian system, $S \s X$ a complex pre-Poisson submanifold, and $\H \tto S$ a holomorphic stabilizer subgroupoid of $S$ in $\G$.
Suppose that $N \coloneqq \mu^{-1}(S)$ is reduced and that there is a complex analytic quotient $\pi : N \longrightarrow M \sll{S, \H, \pi} \H$ of $N$ by $\H$.
Let us also suppose that there is a complex analytic quotient $\rho : S \longrightarrow R$ of $S$ by $\H$, and that the moment map $\mu : M \longrightarrow X$ descends to a holomorphic map
\[
\bar{\mu} : M \sll{S, \H, \pi} \G \too R.
\]
Then $\bar{\mu}$ is a Poisson map with respect to the Poisson structures of \textup{Theorem \ref{ytazg95b}} and \textup{Proposition \ref{45y1ghc8}}.
\end{thm}

\begin{proof}
This follows from the definitions of the Poisson structures using the same reasoning as in the proof of Theorem \ref{omd5dx3w}.
\end{proof}

\section{Symplectic reduction by a Lie group along a submanifold}
\label{lie-group-case}
This section examines the implications of \S\ref{smooth-case} and \S\ref{phj103sm} for classical Hamiltonian $G$-spaces. In \S\ref{fz9catpk}, we derive some of the Lie-theoretic machinery relevant to reducing by a Lie group along a submanifold. Our first example is Marsden--Weinstein--Meyer reduction, which we address in \S\ref{marsden-weinstein-example}. In \S\ref{fvffdxir}, we introduce the universal reduced spaces $\mathfrak{M}_{G,S,\mathcal{H}}$ and $\mathfrak{M}_{G,S}$ and prove Theorem \ref{Theorem: Reduction by a Lie group along a submanifold}(i). We subsequently define and study the class of \textit{stable} submanifolds $S\subseteq\g^*$ in \S\ref{peg44up7}; this is a particularly convenient class of submanifolds along which to reduce, and it includes all Poisson transversals and Poisson submanifolds. The proofs of Parts (ii)--(iii) in Theorem \ref{Theorem: Reduction by a Lie group along a submanifold} are also given in \S\ref{peg44up7}. We then devote \S\ref{emhsjkzj} to the proof of Theorem \ref{Theorem: Reduction by a Lie group along a submanifold}(iv). Theorem \ref{Theorem: General examples} is the subject of \S\ref{Subsection: Symplectic reduction along a coadjoint orbit}--\ref{jyt6wx64}, while Theorem \ref{Theorem: Specific examples}(ii) is proved in \S\ref{Subsection: Decomposition class}.

Except for those instances in which we explicitly indicate otherwise, all material in this section can be interpreted in both the category of smooth manifolds and the category of complex manifolds.
We will simply use the terms \dfn{manifolds} and \dfn{maps} for smooth manifolds and smooth maps, or complex manifolds and holomorphic maps.
A \dfn{submanifold} will refer to a smooth submanifold or a complex submanifold, and both are assumed to be embedded but not necessarily closed.
A \dfn{Lie group} will mean a real Lie group or a complex Lie group, and we will proceed analogously for Lie algebras.
A \dfn{symplectic form} will mean a symplectic form on smooth manifold or a holomorphic symplectic form on a complex manifold, etc.

\subsection{Main construction}\label{fz9catpk}

The notion of a Hamiltonian system introduced in \S\ref{smooth-case} and \S\ref{phj103sm} generalizes the standard one for Lie group actions \cite[Examples 3.1 and 3.9]{mikami-weinstein}.
Let us briefly recall the salient details.

Let $G$ be a Lie group with Lie algebra $\g$, adjoint representation $\Ad : G \longrightarrow \mathrm{GL}(\g)$, and coadjoint representation $\Ad^* : G \longrightarrow \mathrm{GL}(\g^*)$.
We use the left trivialization 
\[
G \times \g^* \overset{\cong}{\too} T^*G, \quad (g, \xi) \mtoo \xi\circ (dL_{g^{-1}})_g
\]
to freely identify $T^*G$ and $G \times \g^*$ throughout our paper.
Let $\Omega\coloneqq d\Theta$ be the canonical symplectic form on $T^*G$, where $\Theta$ is the tautological $1$-form.
We have 
\begin{equation}\label{lcnmyorx}
\Omega_{(g, \xi)}((u_1, \zeta_1), (u_2, \zeta_2)) = -\zeta_2(u_1) + \zeta_1(u_2) - \xi([u_1, u_2])
\end{equation}
for all $(u_i, \zeta_i) \in \g \times \g^* = T_gG \times \g^*=T_{(g, \xi)}(G \times \g^*)$ and $(g,\xi)\in G\times\g^*=T^*G$, where $T_gG$ is identified with $\g$ by left translations (see e.g.\ \cite[Proposition 4.4.1]{abraham-marsden}).

Consider the $(G \times G)$-action on $G$ defined by
\[
(a, b) \cdot g = a g b^{-1}
\]
for all $(a, b) \in G \times G$ and $g \in G$.
This action lifts to a Hamiltonian action on $T^*G = G \times \g^*$ given by
\[
(a, b) \cdot (g, \xi) = (agb^{-1}, \Ad_b^*\xi)
\]
for all $(a, b) \in G \times G$ and $(g, \xi) \in T^*G$, and it admits
\begin{equation}\label{1h6um459}
\mu : T^*G \too \g^* \times \g^*, \qquad (g, \xi) \mtoo (-\Ad_g^*\xi, \xi)
\end{equation}
as a moment map (see e.g.\ \cite[Theorem 4.4.3]{abraham-marsden}).

The identification $T^*G = G \times \g^*$ also shows that $T^*G$ has the structure of a Lie groupoid, namely the action groupoid for the coadjoint representation.
Explicitly, $T^*G \tto \g^*$ is a Lie groupoid with source and target maps 
\begin{equation}\label{lj8xx6wr}
\sss(g, \xi) = \Ad_g^*\xi \quad\text{and}\quad \ttt(g, \xi) = \xi,
\end{equation}
and multiplication
\[
(g, \xi) \cdot (h, \eta) = (gh, \eta) \quad \text{if} \quad\xi = \Ad_h^*\eta.
\]
The canonical symplectic form on $T^*G$ then gives $T^*G \tto \g^*$ the structure of a symplectic groupoid, called the \dfn{cotangent groupoid of $G$}.
The Poisson bivector field $\sigma$ on the base $\g^*$ is the Kirillov--Kostant--Souriau structure given by
\begin{equation}\label{ohdna9u5}
\sigma : T^*\g^* = \g \times \g^* \longrightarrow \g^* \times \g^*=T\g^*,\quad (x, \xi) \mtoo (-\ad_x^*\xi, \xi),
\end{equation}
where $\ad_x : \g \longrightarrow \g$ is the adjoint action $\ad_x y = [x, y]$ of $x\in\g$ and $\ad_x^* : \g^* \longrightarrow \g^*$ is its dual $\ad_x^*\xi \coloneqq -\xi \circ \ad_x$.

A Hamiltonian $G$-space $(M, \omega)$ with moment $\mu : M \longrightarrow \g^*$ is then equivalent to a Hamiltonian system (Definition \ref{Definition: Hamiltonian system}) with underlying symplectic manifold $(M, \omega)$, moment map $\mu$, and symplectic groupoid $T^*G \tto \g^*$; one simply lets $T^*G \tto \g^*$ act on $M$ by
\[
(g, \xi) \cdot p = g \cdot p \quad\text{if}\quad \xi = \mu(p)
\]
for all $(g, \xi) \in T^*G$ and $p \in M$ \cite[Example 3.9]{mikami-weinstein}. In this context, we have the following counterpart of Definitions \ref{Definition: Reduction along a submanifold} and \ref{Definition: Complex analytic reduction}.

\begin{defn}
Let $M$ be a Hamiltonian $G$-space with moment map $\mu : M \longrightarrow \g^*$.
A \dfn{reduction datum for the action of $G$ on $M$} is a reduction datum for the action of $T^*G \tto \g^*$ on $M$, i.e.\ a pair $(S, \H)$ with $S \s \g^*$ a pre-Poisson submanifold and $\H \tto S$ a stabilizer subgroupoid of $S$ in $T^*G \tto \g^*$.
We say that $(S, \H)$ is \dfn{clean} if it is a clean reduction datum for the action of $T^*G \tto \g^*$ on $M$.
In this case, the \dfn{symplectic reduction of $M$ by $G$ along $S$ with respect to $\H$} is defined to be
\[
M \sll{S, \H} G \coloneqq M \sll{S, \H} T^*G = \mu^{-1}(S) / \H.
\]
If $\H$ is source-connected, we use the simplified notation
\[
M \sll{S} G \coloneqq M \sll{S, \H} G
\]
and call it the \dfn{symplectic reduction of $M$ by $G$ along $S$}.
\end{defn}

\begin{rem}
Although $G$ is a group, $\H$ can be a non-trivial groupoid; see Theorem \ref{Theorem: Moore-Tachikawa} and the discussion preceding it.
One consequence is that $M \sll{S} G$ might not be realizable as a quotient of $\mu^{-1}(S)$ by a Lie group action.
\end{rem}

Now recall that $T^*G\big\vert_S$ denotes the restriction of the cotangent groupoid to $S\subseteq\g^*$. This features in following sufficient condition for $M \sll{S, \H} G$ to be a symplectic manifold, which we state for future reference. 

\begin{lem}\label{4z1djzwm}
Let $M$ be a Hamiltonian $G$-space with moment map $\mu : M \longrightarrow \g^*$ and $(S, \H)$ a reduction datum.
If $G$ acts freely and properly on $M$ and $\H$ is closed in $T^*G\big\vert_S$, then $(S, \H)$ is clean and $M \sll{S, \H} G$ is a symplectic manifold.
\end{lem}

\begin{proof}
Since the action of $G$ on $M$ is free, so is the action of $\H$ on $N \coloneqq \mu^{-1}(S)$.
Theorem \ref{1gjd75z2} or Proposition \ref{ugr3y1lq} then show that $S$ is transverse to $\mu$, and that it suffices to show the action of $\H$ on $N$ to be proper.
Note that the action of $T^*G \tto \g^*$ on $M$ is proper; \eqref{ff300lau} reduces to the map $G \times M \longrightarrow M \times M$, $(g, p) \mto (p, g \cdot p)$, which is proper by assumption.  
The action of $\H$ on $N$ is then proper by Proposition \ref{mbslz011}.
\end{proof}

We also record the following description of the stabilizer subalgebroid, which will be useful for examples.

\begin{lem}\label{subalgebroid-lemma}
Let $S \s \g^*$ be a submanifold and let $\sigma$ be the Kirillov--Kostant--Souriau Poisson structure on $\g^*$.
The subset $L_S \coloneqq \sigma^{-1}(TS) \cap TS^\circ$ of $\mathrm{Lie}(T^*G) = T^*\g^* = \g \times \g^*$ is given by
\[
L_S = \{(x, \xi) \in \g \times S : x \in (T_\xi S)^\circ \text{ and } \ad_x^*\xi \in T_\xi S\}.
\]
\end{lem}

\begin{proof}
This follows from the expression for $\sigma$ given by \eqref{ohdna9u5}.
\end{proof}

The following observation will be useful for showing that certain subgroupoids are in fact stabilizer subgroupoids.

\begin{lem}\label{m0mdb7qp}
Let $\H \tto S$ be an embedded Lie subgroupoid of $T^*G \tto \g^*$ over any submanifold $S \s \g^*$.
\begin{enumerate}
\item[\textup{(i)}]
We have
\[
\H = \{(g, \xi) \in G \times S : g \in \H_\xi\}
\]
for some collection of submanifolds $\H_\xi \s G$, $\xi \in S$, each containing the identity $1 \in G$.

\item[\textup{(ii)}]
The Lie algebroid of $\H$ is the Lie subalgebroid of $\mathrm{Lie}(T^*G) = \g \times \g^*$ given by
\[
\mathrm{Lie}(\H) = \{(x, \xi) \in \g \times S : x \in T_1\H_\xi\}.
\]
\end{enumerate}
\end{lem}

\begin{proof}
Let $\sss, \ttt : T^*G \tto \g^*$ be the source and target maps \eqref{lj8xx6wr}.
Since $\xi = \ttt(g, \xi) \in S$ for all $(g, \xi) \in \H$, we have 
\[
\H = \{(g, \xi) \in G \times S : g \in \H_\xi\}
\]
with $\H_\xi \coloneqq \{g \in G : (g, \xi) \in \H\}$.
Note that $\H_\xi \times \{\xi\}$ is the target fibre of $\H$ over $\xi$, so that $\H_\xi$ is necessarily smooth. This completes the proof of Part (i).

To verify Part (ii), note that the restriction $(\ker d\ttt)\big\vert_{\{1\} \times \g^*}$ of $\ker d\ttt$ to the identity section $\{1\} \times \g^* \s G\times\g^* = T^*G$ is the trivial vector bundle $\g \times \g^*$ over $\g^*$.
From \eqref{lcnmyorx}, we see that the canonical isomorphism $(\ker d\ttt)\big\vert_{\{1\} \times \g^*} \longrightarrow T^*\g^* = \g \times \g^*$ given by \eqref{r6nwj6xb} is
\[
\g \times \g^* \longrightarrow \g \times \g^*, \quad (x, \xi) \mtoo (-x, \xi).
\]
It follows that the Lie subalgebroid of $T^*\g^*$ induced by $\H$ is as claimed.
\end{proof}

\subsection{Marsden--Weinstein--Meyer reduction}
\label{marsden-weinstein-example}

We now show that Marsden--Weinstein--Meyer reduction can be recovered as symplectic reduction along a singleton.

Let $G$ be a Lie group acting on a symplectic manifold $(M, \omega)$ in a Hamiltonian way with moment map $\mu : M \longrightarrow \g^*$.
Let $H \s G$ be an immersed Lie subgroup with Lie algebra $\h \s \g$ and let $\xi \in \g^*$.
Let $H_\xi$ be the stabilizer subgroup of $\xi\big\vert_\h \in \h^*$ in $H$, and consider the annihilator $\h^{\circ}\subseteq\g^*$ of $\h$.

\begin{lem}[{cf.\ \cite[Proposition 5.1 and Example 6.2]{cattaneo-zambon-2007}}]
	The affine linear subspace $\xi + \h^\circ \s \g^*$ is pre-Poisson and
	\begin{equation}\label{olgzbvds}
	\H \coloneqq H_\xi \times (\xi + \h^\circ)
	\end{equation}
	is a stabilizer subgroupoid of $\xi + \h^\circ$ in $T^*G \tto \g^*$.
\end{lem}

\begin{proof}
	Lemma \ref{subalgebroid-lemma} implies that $$L_{\xi + \h^\circ} = \{(x, \xi + \eta) \in \g \times (\xi + \h^\circ) : x \in \h_\xi\},$$ where $\h_\xi \coloneqq \Lie(H_\xi)$.
	By Lemma \ref{m0mdb7qp}(ii), the Lie algebroid of $\H$ is $L_{\xi + \h^\circ}$.
	It then only remains to show that $\H$ is isotropic in $T^*G$.
	To this end, let $(h, \xi + \eta) \in \H$, and let $(u_i, \zeta_i) \in T_{(h, \xi + \eta)}\H = \h_\xi \times \h^\circ$, $i \in \{1, 2\}$.
	Equation \eqref{lcnmyorx} implies that
	\[
	\Omega((u_1, \zeta_1), (u_2, \zeta_2)) = -\zeta_2(u_1) + \zeta_1(u_2) - (\xi + \eta)([u_1, u_2]),
	\]
	and we see that each term vanishes.
\end{proof}

In light of this result, we may consider the symplectic reduction
\[
M \sll{\xi + \h^\circ, \H} G = \mu^{-1}(\xi + \h^\circ) / (H_\xi \times (\xi + \h^\circ)).
\]
By viewing $\xi$ as an element of $\g^* / \h^\circ = \h^*$, we also have the symplectic reduction of $M$ by $H$ at level $\xi$:
\[
M \sll{\xi} H \coloneqq \mu_\h^{-1}(\xi) / H_\xi,
\]
where $\mu_\h : M \longrightarrow \h^*$ is the induced moment map for the action of $H$ on $M$. We have $\mu_\h^{-1}(\xi) = \mu^{-1}(\xi + \h^\circ)$, so these two symplectic reductions coincide:

\begin{prop}
	We have $M \sll{\xi + \h^\circ, \H} G = M \sll{\xi} H$, where $\H$ is given by \eqref{olgzbvds}.\qed
\end{prop}

Taking $\h=\g$ then yields the following corollary.

\begin{cor}
	Let $\xi \in \g^*$. 
	Then $\H \coloneqq G_\xi \times \{\xi\}$ is a stabilizer subgroupoid of $\{\xi\}$ in $T^*G \tto \g^*$ and
	\[
	M \sll{\{\xi\}, \H} G = \mu^{-1}(\xi) / G_\xi
	\]
	is the Marsden--Weinstein--Meyer reduced space at level $\xi$.\qed
\end{cor}

\begin{rem}\label{hbxo6c96}
	Let $\G \tto X$ be a symplectic groupoid and $S \s X$ a pre-Poisson submanifold.
	One consequence of this subsection is that there may not exist any stabilizer subgroupoid $\H \tto S$ of $S$ that is closed in $\G\big\vert_S$.
	To see this, let $\mathfrak{h}\subseteq\mathfrak{g}$ be a Lie subalgebra. Any stabilizer subgroupoid of $\mathfrak{h}^{\circ}\subseteq\g^*$ in the cotangent groupoid $T^*G\tto\g^*$ must take the form $H\times\mathfrak{h}^{\circ}$, where $H$ is a Lie subgroup of $G$ with Lie algebra $\mathfrak{h}$. It now just remains to observe that $\mathcal{H}$ is closed in $\mathcal{G}$ if and only if $H$ is closed in $G$.
\end{rem}

\subsection{The universal reduced spaces}\label{fvffdxir}

Let $G$ be a Lie group and recall the Hamiltonian $(G\times G)$-action on $T^*G$ and moment map $\mu:T^*G \longrightarrow \mathfrak{g}^* \times \mathfrak{g}^*$ discussed in \S\ref{fz9catpk}.
We refer to the Hamiltonian action of $G \cong G \times \{1\} \s G \times G$ on $T^*G$ as the action by \dfn{left translations}.
Similarly, the action of $G \cong \{1\} \times G$ on $T^*G$ is called the action by \dfn{right translations}.

Let us suppose that $S \subseteq \mathfrak{g}^*$ is a pre-Poisson submanifold, that $\H \tto S$ is a stabilizer subgroupoid of $S$ in $T^*G \tto \g^*$, and that $(S, \H)$ is a clean reduction datum with respect to the action of $G$ on $T^*G$ by right translations.
It follows that
\[
\mathfrak{M}_{G, S, \H} \coloneqq T^*G\sll{S, \H}G = (G \times S) / \H
\]
is a symplectic manifold, which we call the \dfn{universal reduced space associated to $(G, S, \H)$}; Theorem \ref{universality} will elucidate this choice of nomenclature. Remark \ref{Remark: Source-connected} justifies using the simplified notation $$\mathfrak{M}_{G, S} \coloneqq \mathfrak{M}_{G, S, \H},$$
where $\mathcal{H}\tto S$ is any source-connected stabilizer subgroupoid of $S$ in $T^*G$.
Since the action of $G$ on $T^*G$ is proper, we observe that $(S, \H)$ is clean if $\H$ is closed in $T^*G\big\vert_S$ (Lemma \ref{4z1djzwm}).
We also observe that the action of $G$ on $T^*G$ by left translations descends to a Hamiltonian $G$-action on $\mathfrak{M}_{G, S,\mathcal{H}} = (G \times S) / \H$ given by
\begin{equation}\label{dex0d9ot}
G \times \mathfrak{M}_{G, S, \H} \too \mathfrak{M}_{G, S, \H}, \quad (a, [(g, \xi)]) \mtoo a \cdot [(g, \xi)] = [(ag, \xi)]
\end{equation}
for all $a, g \in G$ and $\xi \in S$.
The corresponding moment map is induced by the first component of $\mu:T^*G\longrightarrow\g^*\oplus\g^*$, i.e.\ it is given by
\begin{equation}\label{u4l3kt9f}
\mathfrak{M}_{G, S,\mathcal{H}} = (G \times S) / \H \too \g^*, \quad [(g, \xi)] \mtoo -\Ad_g^*\xi.
\end{equation}

We may summarize the previous discussion as follows.

\begin{prop}\label{4ywv9zyt}
Let $S \s \g^*$ be a pre-Poisson submanifold and $\H \tto S$ a stabilizer subgroupoid of $S$ in $T^*G \tto \g^*$.
\begin{enumerate}
\item[\textup{(i)}]
Suppose that $(S, \H)$ is a clean reduction datum for the action of $G$ on $T^*G$ by right translations.
The universal reduced space $\mathfrak{M}_{G, S, \H}$ is then a Hamiltonian $G$-space with action \eqref{dex0d9ot} and moment map \eqref{u4l3kt9f}.
It has dimension
\[
\dim \mathfrak{M}_{G, S, \H} = \dim \g + \dim S - \rk L_S.
\]
\item[\textup{(ii)}]
The hypothesis of \textup{Part (i)} is satisfied if $\H$ is closed in $T^*G\big\vert_S$.
\end{enumerate}
\end{prop}

\begin{proof}
It only remains to prove the formula for the dimension; this is simply a result of $\mathcal{H}$ acting freely on $G\times S$ and the fact that $\mathfrak{M}_{G, S, \H} = (G \times S) / \H$.
\end{proof}

The following result is analogous to the universal property of symplectic implosion \cite[\S4]{guillemin-jeffrey-sjamaar}.

\begin{thm}[Universality; Theorem \ref{Theorem: Reduction by a Lie group along a submanifold}(i)]\label{universality}
Let $G$ be a Lie group acting on a symplectic manifold $M$ in a Hamiltonian way with moment map $\mu : M \longrightarrow \g^*$.
Let $(S, \H)$ be a clean reduction datum for both the action of $G$ on $M$ and the action of $G$ on $T^*G$ by right translations.
We then have a canonical symplectomorphism
\[
M \sll{S, \H} G \cong (M \times \mathfrak{M}_{G, S, \H}^-) \sll{0} G,
\]
where $G$ acts on $M \times \mathfrak{M}_{G, S, \H}^-$ diagonally and $\sll{0}$ denotes Marsden--Weinstein--Meyer reduction at level $0$.
\end{thm}

\begin{proof}
The map
\[
\psi : \mu^{-1}(S) \too M \times T^*G, \quad p \mtoo (p, 1, \mu(p))
\]
descends to a diffeomorphism $$\mu^{-1}(S) / \H \longrightarrow (M \prescript{}{\mu}\times_{-\nu} (G \times S)/\H) / G,$$
where $\nu$ is the moment map \eqref{u4l3kt9f}.
It then remains to show that $\psi^*(\omega, -\Omega) = i^*\omega$, where $i : \mu^{-1}(S) \longrightarrow M$ is the inclusion map.
But $d\psi(u) = (u, 0, d\mu(u))$ for $u \in T\mu^{-1}(S)$, so $$\psi^*(\omega, -\Omega)(u, v) = \omega(u, v) - \Omega((0, d\mu(u)), (0, d\mu(v))) = \omega(u, v)$$
for all $p \in \mu^{-1}(S)$ and $u, v \in T_p\mu^{-1}(S)$.
\end{proof}

\begin{rem}[K\"ahler quotients]
The following remark only applies in the smooth category.
Let $G$ be a compact Lie group and $(M, \omega)$ a Hamiltonian $G$-space with moment map $\mu : M \longrightarrow \g^*$.
Suppose that $M$ has a $G$-invariant K\"ahler metric compatible with $\omega$.
The symplectic reduction $M \sll{\xi} G \coloneqq \mu^{-1}(\xi) / G_\xi$ at level $\xi \in \g^*$ is then also K\"ahler if it is smooth \cite[\S3(C)]{HKLR}.
It is natural to ask if the same holds for a general symplectic reduction $M \sll{S} G$ along a pre-Poisson submanifold $S \s \g^*$.
An immediate corollary of universality (Theorem \ref{universality}) is that it suffices to ask this question for the universal reduced space; if $\mathfrak{M}_{G, S, \H}$ has a compatible $G$-invariant K\"ahler structure, then $M \sll{S, \H} G$ is also K\"ahler.

One might then expect the K\"ahler structure on $T^*G \cong G_\C$ to induce a K\"ahler structure on $\mathfrak{M}_{G, S, \H}$.
However, the standard methods for showing that $M \sll{\xi} G$ is K\"ahler (e.g.\ \cite[\S3(C)]{HKLR} or \cite{kirwan-1984}) do not generalize to the case of symplectic reduction along a submanifold.
For instance, suppose that $S$ is a Poisson transversal in $\g^*$. We then have $\mathfrak{M}_{G, S} = G \times S$ (Example \ref{6sap623o}), which is a complex submanifold of $G_\C$ if and only if $S$ is open in $\g^*$.
On the other hand, there are special cases in which it is possible to define a K\"ahler structure on $\mathfrak{M}_{G, S, \H}$ indirectly; symplectic implosion \cite[\S6]{guillemin-jeffrey-sjamaar} is a noteworthy example (see \S\ref{nxuzyzm0}).
\end{rem}

\subsection{Symplectic reduction along stable submanifolds}
\label{peg44up7}

We now study a large class of pre-Poisson submanifolds that encompasses most of the examples in this paper. The principal advantage of working with these submanifolds is that that their stabilizer subgroupoids can be readily computed as Lie group fibrations (a.k.a.\ group schemes).

\begin{defn}
Let $(X, \sigma)$ be a Poisson manifold.
A submanifold $S \s X$ is called \dfn{stable} if it is pre-Poisson and its stabilizer subalgebroid $L_S$ is contained in $\ker \sigma$.
\end{defn}

\begin{lem}\label{3uj9s7s9}
Poisson submanifolds are stable.
\end{lem}

\begin{proof}
A submanifold $S \s X$ is Poisson if and only if $\sigma(TS^\circ) = 0$.
In this case, $$L_S = \sigma^{-1}(TS) \cap TS^\circ = TS^\circ,$$ so $L_S$ has constant rank.
The condition $\sigma(TS^\circ) = 0$ also implies that $L_S \s \ker \sigma$.
\end{proof}

\begin{lem}
Poisson transversals are stable.
\end{lem}

\begin{proof}
By Example \ref{6sap623o}, a Poisson transversal $S$ satisfies $L_S = 0$.
\end{proof}

Let $G$ be a Lie group with Lie algebra $\g$.
We specialize to the case of $X = \g^*$ with its Kirillov--Kostant--Souriau Poisson structure $\sigma$ in \eqref{ohdna9u5}. A pre-Poisson submanifold $S \s \g^*$ is then stable if and only if $(L_S)_\xi \s \g_\xi$ for all $\xi \in S$, where $(L_S)_\xi$ is given by Lemma \ref{subalgebroid-lemma} and
\[
\g_\xi \coloneqq \Lie(G_\xi) = \{ x \in \g : \ad_x^*\xi = 0\}.
\]

\begin{lem}\label{5r34bsl9}
A $G$-invariant submanifold $S \s \g^*$ is Poisson and hence stable.
\end{lem}

\begin{proof}
For all $\xi \in S$ and $x \in T^*_\xi \g^* = \g$, we have $\sigma_\xi(x) = -\ad_x^*\xi \in T_\xi S$.
\end{proof}

We now show that the stabilizer subgroupoids of stable submanifolds take a particularly simple form.

\begin{thm}[Theorem \ref{Theorem: Reduction by a Lie group along a submanifold}(ii)---(iii)]\label{amc0wisd}
Let $S \s \g^*$ be a stable submanifold. For each $\xi \in S$, let $H_\xi \s G$ be a Lie subgroup with Lie algebra $\h_\xi$. Consider the set
\[
\H \coloneqq \{(g, \xi) \in G \times S : g \in H_\xi\}.
\]

\begin{enumerate}
\item[\textup{(i)}] The subspace
\begin{equation}\label{4mlgs031}
\h_\xi \coloneqq (T_\xi S)^\circ \cap \g_\xi
\end{equation}
is a Lie subalgebra of $\g$ for all $\xi \in S$, and the stabilizer subalgebroid of $S$ is
\[
L_S = \{(x, \xi) \in \g \times S : x \in \h_\xi\}.
\]	
\item[\textup{(ii)}]
Suppose that the following two conditions hold.
\begin{enumerate}[leftmargin=4em]
\item[\textup{(C1)}] $\H$ is a submanifold of $T^*G$.
\item[\textup{(C2)}] $\H$ is isotropic in $T^*G$.
\end{enumerate}
Then $\H$ is a stabilizer subgroupoid of $S$ in $T^*G \tto \g^*$.
\item[\textup{(iii)}]
If \textup{(C1)} holds and there is an open dense subset $U$ of $S$ such that $H_\xi$ is connected for all $\xi \in U$, then \textup{(C2)} also holds.
\item[\textup{(iv)}]
If $\H$ is closed in $G \times S$ and each $H_\xi$ is connected, then both \textup{(C1)} and \textup{(C2)} hold.
\end{enumerate}
\end{thm}

\begin{proof}
Since $L_S \s \ker \sigma$, the anchor map of $L_S$ is trivial.
It follows that $L_S$ is a bundle of Lie subalgebras.
We have
\[
(L_S)_\xi = (T_\xi S)^\circ \cap \sigma_\xi^{-1}(T_\xi S) = (T_\xi S)^\circ \cap \ker \sigma_\xi = \h_\xi
\]
since $(L_S)_\xi \s \ker \sigma_\xi$. This verifies Part (i).
Parts (ii) and (iii) follow from Lemma \ref{m0mdb7qp}(ii) and Remark \ref{2g4d9b8v}(ii).
Part (iv) results from the following generalization of the closed subgroup theorem.
\end{proof}

\begin{thm}\label{s72lkap4}
Let $G$ be a Lie group with Lie algebra $\g$ and let $S$ be a manifold. Suppose that $E$ is a subbundle of the trivial vector bundle $\g \times S \longrightarrow S$ such that the fibre $E_{\xi}$ over each $\xi\in S$ is a Lie subalgebra of $\g$.
Let $H_\xi$ be the connected Lie subgroup of $G$ with Lie algebra $E_\xi$ for each $\xi\in S$, and set
\[
\H \coloneqq \{(g, \xi) \in G \times S : g \in H_\xi\}.
\]
If $\H$ is a closed subset of $G \times S$, then it is a submanifold of $G \times S$.
\end{thm}

\begin{proof}  
Since this is a local statement, we may assume that there is a complement $F \s \g \times S$ to $E$, so that $\g \times S = E \oplus F$.
Consider the exponential map $\g\longrightarrow G$, $x\mtoo e^x$, and define
\[
\Phi : \g \times S \too G \times S, \quad (x + y, \xi) \mtoo (e^x e^y, \xi) \quad \text{for $\xi \in S$, $x \in E_\xi$, and $y \in F_\xi$}.
\]
Let us also fix $\xi_0 \in S$.
We claim that there exists a neighbourhood $U$ of $(0, \xi_0)$ in $\g \times S$ such that $\Phi$ restricts to a diffeomorphism (or biholomorphism) $U \longrightarrow \Phi(U)$ and
\[
\Phi(U \cap E) = \Phi(U) \cap \H.
\]
This will establish that $\H$ is a submanifold in a neighbourhood of $(1, \xi_0)$.
The case of a general point $(h_0, \xi_0)$ will be shown to follow from this special case.

To prove the above-mentioned claim, let $\{U_n\}_{n \in \mathbb{N}}$ be a neighbourhood basis for $\g \times S$ at $(0, \xi_0)$.
The differential $d\Phi_{(0, \xi_0)} : \g \times T_{\xi_0} S \longrightarrow \g \times T_{\xi_0} S$ is the identity map, so we may assume that $\Phi$ restricts to a diffeomorphism (or biholomorphism) on each $U_n$.
We claim that there exists $n$ large enough so that $\Phi(U_n \cap E) = \Phi(U_n) \cap \H$.

Suppose that this is not the case.
We have $\Phi(U_n \cap E) \s \Phi(U_n) \cap \H$, so there exists
\[
(g_n, \xi_n) \in \Phi(U_n) \cap \H \setminus \Phi(U_n \cap E)
\]
for all $n$.
In particular,
\[
(g_n, \xi_n)  = (e^{x_n}e^{y_n}, \xi_n)
\]
for some $x_n \in E_{\xi_n}$ and $y_n \in F_{\xi_n}$ such that $(x_n + y_n, \xi_n) \in U_n$.
Since $U_n$ is a neighbourhood basis at $(0, \xi_0)$, the sequence $(x_n + y_n, \xi_n)$ converges to $(0, \xi_0)$ as $n \longrightarrow \infty$.

Note that $y_n \neq 0$ for all $n$, since $(g_n, \xi_n) \notin \Phi(U_n \cap E)$.
Fix a norm $\|\cdot\|$ on $\g$ and set $c_n = \|y_n\| > 0$.
Then $c_n^{-1}y_n$ lies on the unit sphere of $\g$, so by passing to a subsequence we may assume that $c_n^{-1}y_n \longrightarrow y_0$ for some $y_0 \in \g$ with $\|y_0\| = 1$.
Moreover, $(c_n^{-1}y_n, \xi_n) \in F$ for all $n$, so $(y_0, \xi_0) \in F$.

Let $t \in \R$.
Let $k_n$ be the largest integer such that $k_n \le t/c_n$.
Then $|k_nc_n - t| \le c_n \longrightarrow 0$ and so $k_nc_n\longrightarrow t$.
We also have $e^{y_n} = e^{-x_n}g_n \in H_{\xi_n}$, so $(e^{y_n}, \xi_n) \in \H$ and hence
\[
(e^{(k_nc_n)(c_n^{-1}y_n)}, \xi_n) = (e^{k_ny_n}, \xi_n)=((e^{y_n})^{k_n}, \xi_n) \in \H
\]
for all $n$.
Since $\H$ is closed in $G \times S$, taking the limit as $n \longrightarrow \infty$ shows that $(e^{ty_0}, \xi_0) \in \H$, i.e.\ $e^{ty_0} \in H_{\xi_0}$.
But $t$ is arbitrary, so we have $y_0 \in E_{\xi_0}$.
It follows that $y_0 \in E_{\xi_0} \cap F_{\xi_0} = 0$, contradicting the fact that $\|y_0\| = 1$.
This proves the claim, implying that $\H$ is a submanifold of $G \times S$ in a neighbourhood of $(1, \xi_0)$.

Now let $(h_0, \xi_0) \in \H$ be arbitrary.
Let $U$ be a neighbourhood of $(0, \xi_0)$ in $\g \times S$ such that $\Phi$ restricts to a diffeomorphism (or biholomorphism) $U \longrightarrow \Phi(U)$ and $\Phi(U \cap E) = \Phi(U) \cap \H$.
Since $H_{\xi_0}$ is connected, we can write $h_0 = e^{x_1} \cdots e^{x_m}$ for some $x_i \in E_{\xi_0}$.
By further shrinking $S$ if necessary, we may assume that there are sections $s_i : S \longrightarrow E$ for $i = 1, \ldots, m$ of the form $s_i(\xi) = (y_i(\xi), \xi)$, where $y_i(\xi_0) = x_i$.
Then $h(\xi) \coloneqq e^{y_1(\xi)} \cdots e^{y_m(\xi)}$ is a map $S \longrightarrow G$ such that $(h(\xi), \xi) \in \H$ for all $\xi \in S$ and $h(\xi_0) = h_0$.
Define $\Psi : \g \times S \longrightarrow G \times S$ to be the composition of $\Phi$ and the diffeomorphism (or biholomorphism) $G \times S \longrightarrow G \times S$ given by $(g, \xi) \mtoo (h(\xi)g, \xi)$, i.e.\
\[
\Psi(x + y, \xi) = (h(\xi) e^x e^y, \xi)\quad \text{for all $\xi \in S$, $x \in E_\xi$, and $y \in F_\xi$}.
\]
Note that $\Psi$ is a diffeomorphism (or biholomorphism) onto its image, $\Psi(U \cap E) = \Psi(U) \cap \H$, and $\Psi(0, \xi_0) = (h_0, \xi_0)$.
We conclude that $\H$ is a submanifold in a neighbourhood of $(h_0, \xi_0)$.
\end{proof}

\begin{rem}\label{Remark: Useful}
The assumption in Theorem \ref{s72lkap4} that the groups $H_\xi$ are connected can be replaced by the following more general condition: for each $\xi_0 \in S$ and $h_0 \in H_{\xi_0}$, there is a map $s$ (smooth or holomorphic) from a neighbourhood of $\xi_0$ in $S$ to $G$ such that $s(\xi_0) = h_0$ and $s(\xi) \in H_\xi$ for all $\xi$.
In this case, the proof shows that $\H$ is a submanifold of $G \times S$ if it is closed.
\end{rem}

\subsection{Integration of $G$-invariant submanifolds}
\label{emhsjkzj}

Let $G$ be a Lie group with Lie algebra $\g$ and let $S \s \g^*$ be a $G$-invariant submanifold.
Lemma \ref{5r34bsl9} tells us that $S$ is a Poisson submanifold of $\g^*$, and we let $\sigma_S:T^*S\longrightarrow TS$ denote the Poisson bivector field that it inherits as such. One may then attempt to find a symplectic groupoid that integrates $(S,\sigma_S)$, i.e.\ a symplectic groupoid with Lie algebroid $(T^*S,\sigma_S:T^*S\longrightarrow TS)$.
In what follows, we relate symplectic reduction along $S$ to the process of integrating $(S, \sigma_S)$.  

Since $S$ is a Poisson submanifold, its stabilizer subalgebroid is given by $L_S = TS^\circ$.
We also know Poisson submanifolds to be stable (Lemma \ref{3uj9s7s9}), so that
\[
\h_\xi \coloneqq T_\xi S^\circ
\]
is a Lie subalgebra of $\g$ for all $\xi \in S$.
Let $H_\xi \s G$ be the connected Lie subgroup of $G$ with Lie algebra $\h_\xi$ and suppose that
\begin{equation}\label{qbmbd5d8}
\H \coloneqq \{(g, \xi) \in G \times S : g \in H_\xi\}
\end{equation}
is a submanifold of $T^*G$. 
It follows from Theorem \ref{amc0wisd}(iii) that $\H$ is a source-connected stabilizer subgroupoid of $S$ in $T^*G \tto \g^*$. Suppose also that $(S, \H)$ is a clean reduction datum for the action of $G$ on $T^*G$ by right translations, so that the universal reduced space $$\mathfrak{M}_{G, S}=T^*G \sll{S, \H} G=(G\times S)/\mathcal{H}$$ described in \S\ref{fvffdxir} is a Hamiltonian $G$-space.
Note that if $\H$ is closed in $G \times S$, then $\mathcal{H}$ is a submanifold of $T^*G$ (Theorem \ref{amc0wisd}(iv)) and $(S,\mathcal{H})$ is clean (Lemma \ref{4z1djzwm}).

\begin{rem}
The stabilizer subgroupoid \eqref{qbmbd5d8} of a $G$-invariant submanifold $S \s \g^*$ has been studied as a so-called \textit{character Lagrangian}.
We refer the reader to \cite{weinstein-symplectic-category}, \cite[Eq.\ (4.5)]{guillemin-sternberg-homogeneous}, \cite{guillemin-sternberg-symplectic-analogies}, and \cite[\S12.5]{guillemin-sternberg-semi-classical} for further details.
\end{rem}

\begin{thm}[Theorem \ref{Theorem: Reduction by a Lie group along a submanifold}(iv)]\label{Theorem: invariant}
Let $S\subseteq\g^*$ be a $G$-invariant submanifold. Assume that \eqref{qbmbd5d8} is a submanifold of $T^*G$ and that $(S,\mathcal{H})$ is a clean reduction datum for the action of $G$ on $T^*G$ by right translations. The symplectic manifold $\mathfrak{M}_{G, S}$ then has the structure of a symplectic groupoid integrating $(S, \sigma_S)$. The underlying source and target maps are given by
\begin{equation}\label{2vbq8gis}
\sss:\mathfrak{M}_{G,S}\longrightarrow S,\quad [(g,\xi)]\mtoo\mathrm{Ad}_g^*\xi\quad\text{and}\quad \ttt:\mathfrak{M}_{G,S}\longrightarrow S,\quad [(g,\xi)]\mtoo\xi.
\end{equation}
\end{thm}

\begin{proof}
Note that $\H$ is a Lie subgroupoid of $G \times S = (T^*G)\big\vert_S$.
Hence, to give
\[
\mathfrak{M} \coloneqq \mathfrak{M}_{G, S} = (G \times S) / \H
\]
the structure of a Lie groupoid over $S$, it suffices to show that $\H$ is normal in $G \times S$.
This amounts to checking that $(ghg^{-1}, \Ad_g^*\xi) \in \H$ for all $g \in G$ and $(h, \xi) \in \H$. 
In other words, we want to show that
\begin{equation}\label{51ynu8qy}
gH_\xi g^{-1} \s H_{\Ad_g^*\xi} \quad \text{for all $g \in G$ and $\xi \in S$.}
\end{equation}
Note that $G$-invariance of $S$ yields $\Ad_g^*(T_\xi S) = T_{\Ad_g^*\xi} S$, so $$\Ad_g((T_\xi S)^\circ) = (T_{\Ad_g^*\xi}S)^\circ.$$
We conclude that $\Ad_g \h_\xi = \h_{\Ad_g^*\xi}$, and \eqref{51ynu8qy} now results from connectedness of $H_\xi$. It follows that $\mathfrak{M}$ is a Lie groupoid with source and target maps given by \eqref{2vbq8gis}.

Let $\Gamma_{\mathfrak{M}} \s \mathfrak{M} \times \mathfrak{M} \times \mathfrak{M}$ be the graph of multiplication in $\mathfrak{M}$.
Similarly, let $\G \coloneqq T^*G$ and let $\Gamma_\G$ be its graph of multiplication.
Recall that $\Gamma_\G$ is Lagrangian in $\G \times \G \times \G^-$ by assumption.
We have $$\Gamma_{\mathfrak{M}} = \pi(\Gamma_\G \cap (G \times S)^3),$$ where $\pi : G \times S \longrightarrow {\mathfrak{M}}$ is the quotient map.
On the other hand, let $\bar{\omega}$ and $\omega$ be the symplectic forms on ${\mathfrak{M}} \times {\mathfrak{M}} \times {\mathfrak{M}}^-$ and $\G \times \G \times \G^-$, respectively.
Note that for all $u_1, u_2 \in T\Gamma_{\mathfrak{M}}$, we have $u_i = d\pi(v_i)$ for some $v_i \in T\Gamma_\G$, and hence $$\bar{\omega}(u_1, u_2) = \omega(v_1, v_2) = 0.$$
It follows that $\Gamma_{\mathfrak{M}}$ is Lagrangian in $\mathfrak{M} \times \mathfrak{M} \times \mathfrak{M}^-$, so that $\mathfrak{M}$ is a symplectic groupoid.

Note also that $\H$ acts trivially on $S$, as $H_\xi$ is contained in the $G$-stabilizer of $\xi$ for all $\xi \in S$.
Theorem \ref{omd5dx3w} or Theorem \ref{6hr3szsv} then implies that the source map $\sss : T^*G \longrightarrow \g^*$ descends to a Poisson map ${\mathfrak{M}} \longrightarrow S$. We conclude that ${\mathfrak{M}}$ integrates $(S, \sigma_S)$.
\end{proof}

\begin{rem}
The assumption that the groups $H_\xi$ are connected can be omitted if \eqref{51ynu8qy} holds.
\end{rem}

\begin{ex}
Let $G$ be a compact Lie group with Lie algebra $\g$ and choose a $G$-invariant inner product $\langle\cdot,\cdot\rangle:\g\otimes_{\mathbb{R}}\g\longrightarrow\mathbb{R}$. This induces an inner product on $\g^*$, in which context one has the \dfn{Lie-Poisson sphere} \cite{marcut-2014}
\[
\mathbb{S}(\g^*) \coloneqq \{\xi \in \g^* : \|\xi\| = 1\}.
\]
Note that $(T_\xi \mathbb{S}(\g^*))^\circ = \R \xi_*$ for all $\xi \in \mathbb{S}(\g^*)$, where $\xi_*$ is the unique element of $\g$ satisfying $\langle\xi_*, \cdot\rangle = \xi$.
It follows that the stabilizer subgroupoid \eqref{qbmbd5d8} of $\mathbb{S}(\g^*)$ is of the form
\[
\H = \{(e^{t\xi_*}, \xi) \in G \times \g^* : \|\xi\| = 1, t \in \R\}.
\]
Note that $\H$ is closed in $T^*G|_{\mathbb{S}(\g^*)}$ if and only if $\g = \mathfrak{su}(2)$.
Indeed, in rank higher than one, almost all fibres of $\H$ are dense lines on tori.
This is consistent with the fact that $\mathbb{S}(\g^*)$ is integrable to a symplectic groupoid if and only if $\g = \mathfrak{su}(2)$ \cite{marcut-2014}.
If $G = \mathrm{SU}(2)$, then $\mathbb{S}(\g^*) = S^2$ is a symplectic leaf and $$T^*G \sll{\mathbb{S}(\g^*)} G = S^2 \times (S^2)^-$$ is its canonical integration. This is a straightforward consequence of Lemma \ref{0c3nmzej} in the next subsection.
\end{ex}

\subsection{Symplectic reduction along a coadjoint orbit}\label{Subsection: Symplectic reduction along a coadjoint orbit}

Let $(M, \omega)$ be symplectic manifold and $G$ a Lie group acting on $M$ in a Hamiltonian way with moment map $\mu : M \longrightarrow \g^*$.
We now describe the symplectic reduction of $M$ by $G$ along a coadjoint orbit $\O \s \g^*$.

\begin{lem}\label{zc31vfge}
The coadjoint orbit $\O\subseteq\g^*$ is pre-Poisson and
\begin{equation}\label{v6vzwjx9}
\H \coloneqq \{ (g, \xi) \in G \times \O : \Ad_g^*\xi = \xi\}
\end{equation}
is a stabilizer subgroupoid of $\O$ in $T^*G$.
\end{lem}

\begin{proof}
Lemma \ref{5r34bsl9} implies that $\mathcal{O}$ is a stable pre-Poisson submanifold of $\g^*$.
Note that for all $\xi \in \O$ we have $(T_\xi \O)^\circ = \g_\xi$, so the Lie subalgebra $\h_\xi \coloneqq (T_\xi \O)^\circ \cap \g_\xi$ of $\g$ considered in \eqref{4mlgs031} is $\g_\xi$.
By Theorem \ref{amc0wisd}(ii), it suffices to check that $\H$ is isotropic in $T^*G$.
To this end, we first compute the tangent space of $\H$ at a general point.
Consider the map
\[
\varphi : G \times \O \too \g^*, \quad (g, \xi) \mtoo \Ad_g^*\xi - \xi,
\]
noting that $T_{(g, \xi)}\H = \ker d\varphi_{(g, \xi)}$.
Note also that any vector in $T_{(g, \xi)}(G \times \O)$ is of the form $(u, \ad_v^*\xi)$ for some $u, v \in \g$.
We then have
\[
d\varphi_{(g, \xi)}(u, \ad_v^* \xi) = \frac{d}{dt}\Big\vert_{t = 0} \varphi(g e^{tu}, \Ad_{e^{tv}}^*\xi) = \Ad_g^*(\ad_u^*\xi + \ad_v^*\xi) - \ad_v^*\xi,
\]
so that
\[
T_{(g, \xi)}\H = \{(u, \ad_v^*\xi) : u, v \in \g \text{ and }\ad_u^*\xi = (\Ad_{g^{-1}}^* - 1)\ad_v^*\xi\}.
\]
Let $(u_i, \ad_{v_i}^*\xi)$ for $i \in \{1, 2\}$ be two such vectors.
By \eqref{lcnmyorx}, we have
\begin{equation}\label{z9dej986}
\Omega_{(g, \xi)}((u_1, \ad_{v_1}^*\xi), (u_2, \ad_{v_2}^*\xi)) = -(\ad_{v_2}^*\xi)(u_1) + (\ad_{v_1}^*\xi)(u_2) - \xi([u_1, u_2]).
\end{equation}
The first term of \eqref{z9dej986} is
\[
-(\ad_{v_2}^*\xi)(u_1) = (\ad_{u_1}^*\xi)(v_2) = (\Ad_{g^{-1}}^* - 1)(\ad_{v_1}^*\xi)(v_2) = \xi([\Ad_gv_2, v_1] + [v_1, v_2]),
\]
and a similar computation shows that the second term is
\[
(\ad_{v_1}^*\xi)(u_2) = \xi([v_2, \Ad_gv_1] + [v_1, v_2]).
\]
The third term of \eqref{z9dej986} is 
\begin{align*}
-\xi([u_1, u_2]) &= (\ad_{u_1}^*\xi)(u_2) \\
&= (\Ad_{g^{-1}}^* - 1)(\ad_{v_1}^*\xi)(u_2) \\
&= \xi([\Ad_g u_2, v_1]) + \xi([v_1, u_2]) \\
&= \xi([u_2, \Ad_{g^{-1}}v_1]) - \xi([u_2, v_1]) \\
&= (\ad_{u_2}^*\xi)(-\Ad_{g^{-1}}v_1 + v_1) \\
&= (\Ad_{g^{-1}}^* - 1)(\ad_{v_2}^*\xi)(-\Ad_{g^{-1}}v_1 + v_1) \\
&= -\xi([v_1, v_2] + [v_2, \Ad_g v_1] + [v_2, \Ad_{g^{-1}}v_1] + [v_1, v_2]) \\
&= -\xi(2[v_1, v_2] + [v_2, \Ad_g v_1] + [\Ad_g v_2, v_1]).
\end{align*}
By combining the last three expressions with \eqref{z9dej986}, we get
\[
\Omega_{(g, \xi)}((u_1, \ad_{v_1}^*\xi), (u_2, \ad_{v_2}^*\xi)) = 0.\qedhere
\]
\end{proof}

Note that $\H$ is closed in $T^*G\big\vert_\O$.
Proposition \ref{4ywv9zyt} therefore implies that
\[
\mathfrak{M}_{G, \O, \H} \coloneqq T^*G \sll{\O, \H} G
\]
is a Hamiltonian $G$-space.

\begin{lem}\label{0c3nmzej}
There is a canonical symplectomorphism
\[
\mathfrak{M}_{G, \O, \H} \overset{\cong}\too \O \times \O^-,
\]
and it is equivariant for the Hamiltonian $G$-action on $\mathfrak{M}_{G, \O, \H}$ and the $G$-action on the first factor of $\O \times \O^-$.
\end{lem}

\begin{proof}
We have $\mathfrak{M}_{G, \O, \H} = (G \times \O) / {\sim}$, where $(g, \xi) \sim (gh^{-1}, \xi)$ for all $(g, \xi) \in G \times \O$ and $h \in G_\xi$.
This implies that the map
\[
\psi : G \times \O \too \O \times \O, \quad (g, \xi) \mtoo (\Ad_g^*\xi, \xi),
\]
descends to a diffeomorphism $\mathfrak{M}_{G, \O, \H} \cong \O \times \O^-$.
To show that it is a symplectomorphism, we need to show that $\psi^*(\beta, -\beta) = i^*\Omega,$ where $\beta \in \Omega_{\O}^2$ is the Kirillov--Kostant--Souriau symplectic form on $\O$, $i$ is the inclusion $G \times \O \hookrightarrow T^*G = G \times \g^*$, and $\Omega$ is the canonical symplectic form on $T^*G$.

Observe that we have a short exact sequence $$0 \longrightarrow \g_\xi \longrightarrow \g \longrightarrow T_\xi\O \longrightarrow 0,$$ where the map $\g \longrightarrow T_\xi\O$ is given by $x \mtoo \ad_x^*\xi$. Observe also that $\beta$ is characterized by the condition $$\beta_\xi(\ad_x^*\xi, \ad_y^*\xi) = -\xi([x, y])$$ for all $x,y\in\g$.
We then find that
\[
d\psi_{(g, \xi)}(x, \ad_y^*\xi) = (\ad_{\Ad_g(x+y)}^* \Ad_g^*\xi, \ad_y^*\xi),
\]
so
\begin{align*}
\psi^*(\beta, -\beta)_{(g, \xi)} &((x, \ad_y^*\xi), (u, \ad_v^*\xi)) \\
&= -(\Ad_g^*\xi)([\Ad_g(x+y), \Ad_g(u+v)]) + \xi([y, u]) \\
&= -\xi([x+y,u+v]) + \xi([y,v]) \\
&= -\xi([x, u] + [x, v] + [y, u]).
\end{align*}
On the other hand,
\begin{align*}
\Omega_{(g, \xi)}((x, \ad_y^*\xi), (u, \ad_v^*\xi)) &= -\ad_v^*\xi(x) + \ad_y^*\xi(u) - \xi([x, u]) \\
&= -\xi([x, u] + [x, v] + [y, u]).
\end{align*}
It follows that $\psi^*(\beta, -\beta) = i^*\Omega$.
\end{proof}

The symplectic reduction $M \sll{\O, \H} G$ of $M$ by $G$ along $\O$ with respect to \eqref{v6vzwjx9} is not necessarily $\mu^{-1}(\O) / G$ \cite{kazhdan-kostant-sternberg}, but the two are closely related: 

\begin{prop}[Theorem \ref{Theorem: General examples}(i)]\label{l73ftj1i}
Let $\H$ be defined by \eqref{v6vzwjx9} and suppose that $(\O, \H)$ is a clean reduction datum for the action of $G$ on $M$.
Then there is a canonical symplectomorphism
\[
M \sll{\O, \H} G \cong \mu^{-1}(\O) / G \times \O^-.
\]
\end{prop}

\begin{proof}
By Theorem \ref{universality}, this is a consequence of Lemma \ref{0c3nmzej} and the ``shifting trick'' $(M \times \O) \sll{0} G \cong \mu^{-1}(\O) / G.$
\end{proof}

\subsection{Symplectic implosion}\label{nxuzyzm0}
The following subsection only applies to the smooth category.

Let $K$ be a compact connected Lie group with Lie algebra $\k$.
Let us also fix a maximal torus $T\subseteq K$ with Lie algebra $\mathfrak{t}\subseteq\mathfrak{k}$.
Write $\Phi\subseteq (i\mathfrak{t})^*$ for the associated set of roots and let $\alpha^{\vee}\in i\mathfrak{t}$ denote the coroot determined by $\alpha\in\Phi$.
Choose a closed fundamental Weyl chamber $\t_+^* \subseteq \t^*$ and let $\Delta\subseteq\Phi$ be the induced set of simple roots, i.e.
\[
\t_+^*=\{\xi\in\mathfrak{t}^*:\xi(\alpha^{\vee})\geq 0\text{ for all }\alpha\in\Delta\}.
\]
The faces of $\t_+^*$ are in one-to-one correspondence with the subsets $\sigma$ of $\Delta$ via
\[
\sigma \mtoo S_\sigma \coloneqq \{\xi\in\t^*:\xi(\alpha^{\vee})=0\text{ for all }\alpha \in \sigma \text{ and }\xi(\alpha^{\vee})>0\text{ for all }\alpha\in\Delta\setminus \sigma\}.
\]
If $S \s \t_+^*$ is a face, then the $K$-stabilizers $K_{\xi}$ and $K_{\eta}$ coincide for all $\xi,\eta\in S$; we write $K_S$ for this common stabilizer.

\begin{lem}\label{96kseenm}
Let $S \subseteq \t_+^*$ be a face.
Then $S$ is a pre-Poisson submanifold of $\mathfrak{k}^*$ and
\begin{equation}\label{u5v8dab4}
[K_S,K_S] \times S
\end{equation}
is a stabilizer subgroupoid of $S$ in $T^*K = K \times \k^*$.
\end{lem}

\begin{proof}
Note that $[K_S, K_S]$ is the connected Lie subgroup of $G$ with Lie algebra $[\k_S, \k_S]$.
This combines with Lemma \ref{m0mdb7qp}(iii) and reduces us to showing the following: for all $\xi \in S$, the Lie algebroid fibre
\[
	(L_S)_\xi \coloneqq \{x \in \k : x \in (T_\xi S)^\circ \text{ and } \ad_x^*\xi \in T_\xi S\}
\]
is equal to $[\k_\xi, \k_\xi]$.

Identify $\k$ with $\k^*$ via a $K$-invariant inner product $\langle\cdot, \cdot\rangle:\mathfrak{k}\otimes_{\mathbb{R}}\mathfrak{k}\longrightarrow\mathbb{R}$, and let $y \in \k$ be the element corresponding to $\xi \in S$.
Then $T_\xi S$ corresponds to the intersection of all root hyperplanes containing $y$, i.e.\ to
\[
V_y \coloneqq \bigcap_{\substack{\alpha \in \Phi \\ \alpha(y) = 0}} \ker \alpha.
\]
It therefore suffices to show that
\[
L_y \coloneqq V_y^\perp \cap \ad_y^{-1}(V_y)
\]
is equal to $[\k_y, \k_y]$.

Note that if $x \in L_y$, then $$[x, y] \in V_y \cap [\k, \t] \s \t \cap [\k, \t] = 0,$$ so in fact $L_y = V_y^\perp \cap \k_y$.
On the other hand, consider the complexifications $\g \coloneqq \k_\C$ and $\h \coloneqq \t_\C$ of $\k$ and $\t$, respectively.
Then $(V_y)_\C$ is the intersection of $\ker \alpha$ for all roots $\alpha : \h \longrightarrow \C$ such that $\alpha(y) = 0$.
In light of this, it suffices to show that $$(V_y)_\C^\perp \cap \g_y = [\g_y, \g_y].$$

Let
\[
\Psi \coloneqq \{\alpha \in \Phi : \alpha(y) = 0\}.
\]
One sees that $\Psi$ is a closed root subsystem of $\Phi$.
It therefore determines a semisimple subalgebra \cite{dynkin-1952}
\[
\g_\Psi \coloneqq \h_\Psi \oplus \bigoplus_{\alpha \in \Psi} \g_\alpha,
\]
where $\h_\Psi$ is the span of the coroots $\alpha^{\vee}$ for all $\alpha \in \Psi$ (see e.g.\ \cite[Chapter 6, \S1]{onishchik-vinberg}).
Note that $$\g_y = \h \oplus \bigoplus_{\alpha \in \Psi} \g_\alpha,$$ so $[\g_y, \g_y] = \g_\Psi$.
We also have $$(V_y)_\C^\perp = \h_\Psi \oplus \bigoplus_{\alpha \in \Phi} \g_\alpha,$$ so $(V_y)_\C^\perp \cap \g_y = \g_\Psi = [\g_y, \g_y].$ Our proof is therefore complete.
\end{proof}

Now let $K$ act on a symplectic manifold $(M, \omega)$ in a Hamiltonian way with moment map $\mu : M \longrightarrow \k^*$.
Recall that the \dfn{imploded cross-section} of $M$ \cite{guillemin-jeffrey-sjamaar} is the quotient topological space
\[
M_{\mathrm{impl}} \coloneqq \mu^{-1}(\t_+^*) / {\sim},
\]
where $p \sim q$ if $p = k \cdot q$ for some $k \in [K_{\mu(p)}, K_{\mu(p)}]$.
This space can also be written as the disjoint union
\[
M_{\mathrm{impl}} = \bigcup_{\sigma \subseteq \Delta}\mu^{-1}(S_\sigma)/[K_{S_\sigma}, K_{S_\sigma}],
\]
where each piece $\mu^{-1}(S_\sigma) / [K_{S_\sigma}, K_{S_\sigma}]$ is a symplectic manifold.
By Lemma \ref{96kseenm} and the fact that the groups $[K_S, K_S]$ in \eqref{u5v8dab4} are connected, we obtain the following result.

\begin{prop}[Theorem \ref{Theorem: General examples}(ii)]\label{0dxzp0g8}
The symplectic reduction of $M$ by $K$ along $\t_+^*$ is the imploded cross-section of $M$ by $K$, i.e.\
\[
M \sll{\t_+^*} {\mkern-5mu} K = M_{\mathrm{impl}}.\tag*{\qed}
\]
\end{prop}

\begin{rem}
While the subset $\t_+^* \s \k^*$ is not a submanifold, we define the quotient as in Remark \ref{zlyzkwmo}.
\end{rem}

\begin{rem}
Proposition \ref{0dxzp0g8} shows that the universal reduced space $\mathfrak{M}_{K, \mathfrak{t}_+^*}$ is the universal imploded cross-section $(T^*K)_{\mathrm{impl}}$ \cite[\S4]{guillemin-jeffrey-sjamaar}.
The universal property of $(T^*K)_{\mathrm{impl}}$ \cite[Theorem 4.9]{guillemin-jeffrey-sjamaar} is then a special case of Theorem \ref{universality}.
\end{rem}

\subsection{Symplectic cutting}\label{jyt6wx64}
The following subsection only applies to the smooth category.

The symplectic cut construction of Lerman \cite{lerman} and its generalization to torus actions \cite{lerman-meinrenken-tolman-woodward} can be viewed as symplectic reduction along a polyhedral set, as we now explain.

Let $(M, \omega)$ be a symplectic manifold with an effective action of a compact torus $T$ with moment map $\mu : M \longrightarrow \t^*$.
Following \cite[\S2]{lerman-meinrenken-tolman-woodward}, let $P \s \t^*$ be a convex rational polyhedral set of the form
\[
P = \{\xi \in \t^* : \xi(v_i) \ge b_i \text{ for all } 1 \le i \le N\}
\]
for some $N \in \mathbb{N}$, $v_1, \ldots, v_N \in \t$ in the integral lattice of $T$, and $b=(b_1,\ldots,b_N) \in \R^N$.
The vectors $v_i$ define a Hamiltonian action of $(S^1)^N$ on $M \times \C^N$ with moment map $\nu : M \times \C^N \longrightarrow \mathbb{R}^N$, where the $i^{\text{th}}$ component of $\nu$ is $\nu_i(p, z) = \mu(p)(v_i) - |z_i|^2$.
The \dfn{symplectic cut of $M$ with respect to $P$} is the Marsden--Weinstein--Meyer reduction of $M \times \C^N$ by $(S^1)^N$ at level $b$. We denote this reduced space by $M_P$.

While $M_P$ is singular in general, it is stratified into symplectic manifolds.
A more precise statement is that
\[
M_P = \bigcup_F \mu^{-1}(F) / T_F,
\]
where the union ranges over the open faces $F$ of $P$ and $T_F \s T$ is the torus whose Lie algebra is $(T_\xi F)^\circ \s \t$ for any $\xi \in F$.
Each manifold $\mu^{-1}(F) / T_F$ is given a symplectic structure $\omega_F$ characterized by $$\pi_F^*\omega_F = i_F^*\omega,$$ where $\pi_F : \mu^{-1}(F) \longrightarrow \mu^{-1}(F) / T_F$ is the quotient map and $i_F : \mu^{-1}(F) \longrightarrow M$ the inclusion map.

Since the Poisson structure on $\t^*$ is trivial, every face $F$ is Poisson and hence stable (Lemma \ref{3uj9s7s9}).
Theorem \ref{amc0wisd}(ii) then shows that $T_F \times F$ is a source-connected stabilizer subgroupoid of $F$ in $T^*T$.
The symplectic manifold $\mu^{-1}(F) / T_F$ is then the symplectic reduction $M \sll{F} T$ of $M$ by $T$ along $F$.
These considerations allow us to deduce the following result.

\begin{prop}[Theorem \ref{Theorem: General examples}(iii)]\label{Proposition: Cut}
The symplectic reduction of $M$ by $T$ along a polyhedral set $P$ is the symplectic cut of $M$ with respect to $P$, i.e.\
\[
M \sll{P} T = M_P.\tag*{\qed}
\]
\end{prop}

\subsection{Symplectic reduction along a decomposition class}\label{Subsection: Decomposition class}
Let $G$ be a connected complex semisimple linear algebraic group with Lie algebra $\g$, and let $\g_x$ denote the $\g$-centralizer of $x\in\g$. Write $x=x_s+x_n$ for the Jordan decomposition of any $x\in\mathfrak{g}$ into a semisimple element $x_s\in\mathfrak{g}$ and a nilpotent element $x_n\in\mathfrak{g}$. This gives rise to an equivalence relation $\sim$ on $\mathfrak{g}$, in which $x\sim y$ if and only if $\mathfrak{g}_{x_s}=\mathrm{Ad}_g(\mathfrak{g}_{y_s})$ and $x_n=\mathrm{Ad}_g y_n$ for some $g\in G$. The resulting equivalence classes are called \textit{decomposition classes} (a.k.a.\ \textit{Jordan classes}) and play an important role in Lie-theoretic geometry \cite{borho-kraft,popov,broer,broer2,spaltenstein,moreau,imhof}. Among other things, each decomposition class $\mathcal{D}\subseteq\mathfrak{g}$ is a smooth \cite[Corollary 3.8.1(i)]{broer2}, locally closed \cite[Corollary 39.1.7(ii)]{tauvel}, $G$-invariant subvariety of $\mathfrak{g}$. It follows from Lemma \ref{5r34bsl9} that $\mathcal{D}$ is a stable subvariety of $\mathfrak{g}$, allowing us to concretely study stabilizer subgroupoids of $\mathcal{D}$ in the cotangent groupoid $T^*G\tto\mathfrak{g}$. On the other hand, suppose that $x\in\mathcal{D}$ and consider the $[G_{x_s},G_{x_s}]$-stabilizer $[G_{x_s},G_{x_s}]_{x_n}$ and $[\mathfrak{g}_{x_s},\mathfrak{g}_{x_s}]$-centralizer $[\mathfrak{g}_{x_s},\mathfrak{g}_{x_s}]_{x_n}$ of the nilpotent part $x_n\in[\mathfrak{g}_{x_s},\mathfrak{g}_{x_s}]$. Let us also consider the identity component $[G_{x_s},G_{x_s}]_{x_n}^{\circ}$ of $[G_{x_s},G_{x_s}]_{x_n}$ for each $x\in\mathcal{D}$. The group scheme
\begin{equation}\label{Equation: Stabilizer subgroupoid}\mathcal{H}\coloneqq\{(g,x)\in G\times\mathcal{D}:g\in [G_{x_s},G_{x_s}]_{x_n}^{\circ}\}\overset{\pi}\longrightarrow\mathcal{D}\end{equation} is a subgroupoid of $T^*G\tto\mathfrak{g}$ if we use the Killing form and left trivialization to identity $T^*G$ with $G\times\mathfrak{g}$. 

\begin{rem} 
The conjugacy classes of $[G_{x_s},G_{x_s}]_{x_n}$ in $G$ and $[\mathfrak{g}_{x_s},\mathfrak{g}_{x_s}]_{x_n}$ in $\mathfrak{g}$ are easily seen to be independent of $x\in\mathcal{D}$. It follows that any two fibres of \eqref{Equation: Stabilizer subgroupoid} are isomorphic as algebraic groups.
\end{rem}

\begin{rem}
If $x\in\g$ is semisimple, then $[G_{x_s},G_{x_s}]_{x_n}^{\circ}=[G_x,G_x]$. The stabilizer subgroupoid \eqref{Equation: Stabilizer subgroupoid} thereby takes a simpler form if $\mathcal{D}$ consists of semisimple elements.
\end{rem}

\begin{prop}\label{Proposition: Subgroupoid for decomposition class}
	If $\mathcal{D}\subseteq\mathfrak{g}$ is a decomposition class, then $\pi:\mathcal{H}\longrightarrow\mathcal{D}$ is a stabilizer subgroupoid of $\mathcal{D}$ in $T^*G\tto\mathfrak{g}$. Furthermore, $(\mathcal{D},\mathcal{H})$ is a clean reduction datum for the action of $G$ on $T^*G$ by right translations.
\end{prop}

\begin{proof}
We begin by verifying that $\mathcal{H}$ is closed in $G\times\mathcal{D}$. To this end, fix any $y\in\mathcal{D}$, set $Q\coloneqq[G_{y_s},G_{y_s}]_{y_n}^{\circ}$, and let $R$ be the normalizer of $Q$ in $G$. We may $G$-equivariantly identify $G/R$ with the set of closed subgroups in $G$ conjugate to $Q$, i.e.\ via the unique $G$-equivariant bijection sending the identity coset $[1]\in G/R$ to $Q$. Associating the subgroup $[G_{x_s},G_{x_s}]_{x_n}^{\circ}$ to each $x\in\mathcal{D}$ then defines a continuous map $\phi:\mathcal{D}\longrightarrow G/R$. On the other hand, we may choose an open neighbourhood $U\subseteq G/R$ of $[1]$ and a holomorphic section $s:U\longrightarrow G$ of $G\longrightarrow G/R$ satisfying $s([1])=1$. Now set $V\coloneqq\phi^{-1}(U)\subseteq\mathcal{D}$ and note that 
$$t:V\longrightarrow G,\quad x\mtoo s(\phi(x))$$ is a holomorphic map satisfying $$t(x)Qt(x)^{-1}=[G_{x_s},G_{x_s}]_{x_n}^{\circ}$$ for all $x\in V$. The map 
$$V\times Q\overset{\cong}\longrightarrow\pi^{-1}(V),\quad (x,q)\mapsto (t(x)qt(x)^{-1},x)$$  
then defines a trivialization of $\pi:\mathcal{H}\longrightarrow\mathcal{D}$ over $V$. Since $Q$ is closed in $G$, we conclude that $\mathcal{H}$ is closed in $G\times\mathcal{D}$.

We now note that
$$\mathcal{D}=G(\mathfrak{z}(\mathfrak{g}_{x_s})^{\text{min}}+x_n),$$ where $$\mathfrak{z}(\mathfrak{g}_{x_s})^{\text{min}}\coloneqq\{y\in\mathfrak{z}(\mathfrak{g}_{x_s}):\dim(\mathfrak{g}_y)\leq\dim(\mathfrak{g}_z)\text{ for all }z\in\mathfrak{z}(\mathfrak{g}_{x_s})\}$$
is the locus of elements in $\mathfrak{z}(\mathfrak{g}_{x_s})$ having $\mathfrak{g}$-centralizers of minimal dimension \cite[Corollary 39.1.7]{tauvel}. This description of $\mathcal{D}$ implies that
$$T_x\mathcal{D}=\mathfrak{z}(\mathfrak{g}_{x_s})+[\mathfrak{g},x]$$ for all $x\in\mathcal{D}$. Writing $V^{\perp}\subseteq\g$ for the annihilator of any subspace $V\subseteq\g$ under the Killing form, we conclude that
\begin{align*}
(T_x\mathcal{D})^{\perp} & = \mathfrak{z}(\mathfrak{g}_{x_s})^{\perp}\cap[\mathfrak{g},x]^{\perp}\\
& = \mathfrak{z}(\mathfrak{g}_{x_s})^{\perp} \cap\mathfrak{g}_x\\
& = (\mathfrak{z}(\mathfrak{g}_{x_s})^{\perp}\cap\mathfrak{g}_{x_s})\cap\mathfrak{g}_{x_n}\\
& = [\mathfrak{g}_{x_s},\mathfrak{g}_{x_s}]\cap\mathfrak{g}_{x_n}\\
& = [\mathfrak{g}_{x_s},\mathfrak{g}_{x_s}]_{x_n}
\end{align*}
for all $x\in\mathcal{D}$. This combines with the source-connectedness of $\mathcal{H}\longrightarrow\mathcal{D}$, the closedness of $\mathcal{H}$ in $G\times\mathcal{D}$, and Theorem \ref{amc0wisd} to imply that $\mathcal{H}$ is a stabilizer subgroupoid of $\mathcal{D}$ in $T^*G\tto\g^*$.

It remains only to prove that $(\mathcal{D},\mathcal{H})$ is a clean reduction datum. But this is an immediate consequence of Proposition \ref{4ywv9zyt}(ii).
\end{proof} 

The clean reduction datum $(\mathcal{D},\mathcal{H})$ in Proposition \ref{Proposition: Subgroupoid for decomposition class} determines a universal reduced space
$\mathfrak{M}_{G,\mathcal{D},\mathcal{H}}=\mathfrak{M}_{G,\mathcal{D}}$. An examination of \eqref{Equation: Stabilizer subgroupoid} reveals that
$$\mathfrak{M}_{G,\mathcal{D}}=(G\times\mathcal{D})/\mathcal{H}=\bigsqcup_{x\in\mathcal{D}}G/[G_{x_s},G_{x_s}]_{x_n}^{\circ},$$ where $G$ acts on the third space via left multiplication on each factor. Given any $x\in\mathcal{D}$, this observation and the proof of Proposition \ref{Proposition: Subgroupoid for decomposition class} allow us to describe the tangent space of $([1],x)\in\mathfrak{M}_{G,\mathcal{D}}$ as  $$T_{([1],x)}\mathfrak{M}_{G,\mathcal{D}}= T_{[1]}(G/[G_{x_s},G_{x_s}]_{x_n}^{\circ})\oplus T_x\mathcal{D}=\g/[\g_{x_s},\g_{x_s}]_{x_n}\oplus [\g_{x_s},\g_{x_s}]_{x_n}^{\perp}.$$ One immediate consequence is that
\begin{equation}\label{Equation: Dimension}\dim\mathfrak{M}_{G,\mathcal{D}}=2\dim\mathcal{D}=2\dim G-2\dim([G_{x_s},G_{x_s}]_{x_n})\end{equation} for all $x\in\mathcal{D}$. Another consequence is that the symplectic form $\omega$ on $\mathfrak{M}_{G,\mathcal{D}}$ is characterized by being $G$-invariant and defined as follows on $T_{(x,[1])}\mathfrak{M}_{G,\mathcal{D}}=\g/[\g_{x_s},\g_{x_s}]_{x_n}\oplus [\g_{x_s},\g_{x_s}]_{x_n}^{\perp}$ for all $x\in\mathcal{D}$:
$$\omega_{([1],x)}(([u_1],\zeta_1),([u_2],\zeta_2))=-\langle u_1,\zeta_2\rangle+\langle u_2,\zeta_1\rangle-\langle x,[u_1,u_2]\rangle$$
for all $[u_1],[u_2]\in \g/[\g_{x_s},\g_{x_s}]_{x_n}$ and $\zeta_1,\zeta_2\in[\g_{x_s},\g_{x_s}]_{x_n}^{\perp}$, where $\langle\cdot,\cdot\rangle$ denotes the Killing form on $\g$; this formula is a consequence of \eqref{lcnmyorx} and Theorem \ref{ytazg95b}(iv). We also know that $\mathfrak{M}_{G,\mathcal{D}}$ has the structure of a holomorphic symplectic groupoid integrating $\mathcal{D}$, as follows from Theorem \ref{Theorem: invariant} and Proposition \ref{Proposition: Subgroupoid for decomposition class}. 

\begin{ex}[Theorem \ref{Theorem: Specific examples}(ii)]\label{Example: g_irr}
We now discuss a curious application of this discussion. Consider the subsets $$\g_{\text{irr}}\coloneqq\g\setminus\g_{\text{reg}},\quad\g_{\text{s}}\coloneqq\{x\in\g:x\text{ is semisimple}\},\quad\text{and}\quad\g_{\text{subreg}}\coloneqq\{x\in\g:\dim\g_x=\ell+2\}$$ of irregular, semisimple, and subregular elements in $\g$, respectively.
The intersection $$\g_{\text{irr}}^{\circ}\coloneqq\g_{\text{s}}\cap\g_{\text{subreg}}$$ is an open dense subset of $\g_{\text{irr}}$ equal to the disjoint union of finitely many decomposition classes \cite[Remark 3.7(a)]{popov}. Each of these decomposition classes consists of semisimple elements and has codimension three in $\g$ \cite[Lemma 3.6]{popov}. These considerations and Proposition \ref{Proposition: Subgroupoid for decomposition class} imply that $\g_{\text{irr}}^{\circ}\subseteq\g$ is a smooth, $G$-invariant, stable subvariety of $\g$ admitting $$\mathcal{H}\coloneqq\{(g,x)\in G\times\g_{\text{irr}}^{\circ}:g\in [G_x,G_x]\}\overset{\pi}\longrightarrow\mathfrak{g}_{\text{irr}}^{\circ}$$ as a source-connected stabilizer subgroupoid. It follows that
$$\mathfrak{M}_{G,\g_{\text{irr}}^{\circ}}=\bigsqcup_{x\in \g_{\text{irr}}^{\circ}}G/[G_x,G_x],$$ where $G$ acts on the right-hand side by left multiplication on each factor. Each of the subgroups $[G_x,G_x]$ appearing above is three-dimenional and has Lie algebra $[\g_x,\g_x]\cong\mathfrak{sl}_2$. This combines with \eqref{Equation: Dimension} to yield
$$\dim\mathfrak{M}_{G,\g_{\text{irr}}^{\circ}}=2\dim G-6.$$
The symplectic form and symplectic groupoid structure on $\mathfrak{M}_{G,\g_{\text{irr}}^{\circ}}$ are described analogously to those on $\mathfrak{M}_{G,\mathcal{D}}$ in the previous paragraph. 
\end{ex}

\section{Main construction: algebraic version}
\label{algebraic-case}  

We now develop, prove, and explore the consequences of Theorem \ref{Theorem: Algebraic category}, a counterpart of Theorems \ref{Theorem: Smooth category} and \ref{Theorem: Complex analytic category} in complex algebraic geometry. This undertaking begins with some algebro-geometric preliminaries in \S\ref{Subsection: Algebro-geometric preliminaries}. Theorem \ref{Theorem: Algebraic category} is proved in \S\ref{Subsection: Complex algebraic symplectic reduction along a subvariety} and followed by examples in \S\ref{Subsubsection: Poisson and hyperkahler slices} and \S\ref{Subsection: The Moore-Tachikawa TQFT}. The last of these sections uses Theorem \ref{Theorem: Algebraic category} to derive the Ginzburg--Kazhdan construction of Moore--Tachikawa varieties.    

\subsection{Algebro-geometric preliminaries}\label{Subsection: Algebro-geometric preliminaries}
In the interest of clarity, we now briefly outline the algebro-geometric counterparts of some definitions in \S\ref{phj103sm}. We work exclusively with algebraic varieties over $\mathbb{C}$.
\subsubsection{Quotients}
Suppose that an algebraic groupoid $\mathcal{H} \tto S$ acts algebraically on a variety $N$. An $\mathcal{H}$-invariant variety morphism $\pi:N \longrightarrow Q$ then determines a sheaf $(\pi_*\mathcal{O}_N)^{\mathcal{H}}$ on $Q$, i.e.\ $(\pi_*\mathcal{O}_N)^{\mathcal{H}}(U)$ is the algebra of $\mathcal{H}$-invariant elements of $\mathcal{O}_N(\pi^{-1}(U))$ for each open subset $U \subseteq Q$. Observe that pulling back along $\pi$ defines a sheaf morphism
\begin{equation}\label{Equation: Sheaf morphism}\mathcal{O}_Q\longrightarrow(\pi_*\mathcal{O}_N)^{\mathcal{H}}.\end{equation}

\begin{defn}\label{Definition: algebraic quotient}
Suppose that an algebraic groupoid $\mathcal{H}\tto S$ acts algebraically on an algebraic variety $N$. We define an \dfn{algebraic quotient} of $N$ by $\mathcal{H}$ to be an $\mathcal{H}$-invariant morphism $\pi:N\longrightarrow Q$ for which \eqref{Equation: Sheaf morphism} is an isomorphism.
\end{defn}

\begin{ex}
If $N$ is an affine variety and the algebra $\C[N]^{\mathbb{\H}}$ of $\mathcal{H}$-invariant regular functions is finitely generated, then the morphism $N \too \operatorname{Spec}\C[N]^{\mathcal{\H}}$ is an algebraic quotient.
\end{ex}

\subsubsection{Stabilizer subgroupoids}
We shall use the term \dfn{symplectic variety} in reference to a smooth algebraic variety endowed with an algebraic symplectic form.  The term \dfn{Poisson variety} shall be used for a potentially singular algebraic variety $X$ whose structure sheaf $\mathcal{O}_X$ is a sheaf of Poisson algebras. Now suppose that $X$ is a smooth Poisson variety. The Poisson bracket can then be encoded by a Poisson bivector field $\sigma:T^*X\longrightarrow TX$. A smooth locally-closed subvariety $S\subseteq X$ shall be called \dfn{pre-Poisson} if $\sigma^{-1}(TS) \cap TS^\circ$ has constant rank over $S$.

Now let $\mathcal{G}\tto X$ be an algebraic groupoid for which $\mathcal{G}$ and $X$ are smooth varieties. If $\mathcal{G}$ comes equipped with an algebraic symplectic form $\Omega$ such that the graph of multiplication is Lagrangian in $\G \times \G \times \G^-$, then $\mathcal{G}\tto X$ shall be called an \dfn{algebraic symplectic groupoid}. Each pre-Poisson subvariety $S\subseteq X$ then determines a Lie subalgebroid
\[
L_S\coloneqq \sigma^{-1}(TS) \cap TS^\circ
\]
of $T^*X\cong\mathrm{Lie}(\mathcal{G})$, to be called the \dfn{stabilizer subalgebroid} of $S$. We define a \dfn{stabilizer subgroupoid of} $S$ \dfn{in} $\mathcal{G}$ to be any smooth algebraic groupoid $\mathcal{H}\tto S$  integrating $L_S$, together with an algebraic Lie groupoid homomorphism $j : \H \longrightarrow \G$ such that $j^*\Omega = 0$.

\subsubsection{Algebraic Hamiltonian systems}
Definition \ref{Definition: Hamiltonian system} has an obvious algebro-geometric analogue, i.e.\ a notion of an \dfn{algebraic Hamiltonian action} of an algebraic symplectic groupoid on a symplectic variety. In turn, this gives rise to the definition of an \dfn{algebraic Hamiltonian system}. 

\subsection{Complex algebraic symplectic reduction along a subvariety}\label{Subsection: Complex algebraic symplectic reduction along a subvariety} 
Given any algebraic variety $(Z,\mathcal{O}_Z)$, let $\mathcal{O}^{\text{an}}_Z$ denote the structure sheaf that realizes $Z$ as a complex analytic space. 
For a point $z \in Z$, we may identify the stalk $\O_{Z,z}$ with a subalgebra of $\O_{Z, z}^{\text{an}}$.

\begin{thm}[Theorem \ref{Theorem: Algebraic category}]\label{Theorem: Algebraic theorem}
Let $((M, \omega), \G \tto X, \mu)$ be an algebraic Hamiltonian system and $S\subseteq X$ a pre-Poisson subvariety for which $N\coloneqq\mu^{-1}(S)$ is reduced. Suppose that $\pi : N \longrightarrow Q$ is an algebraic quotient of $N$ by a stabilizer subgroupoid $\mathcal{H}\tto S$ of $S$.
\begin{enumerate}
\item[\textup{(i)}]
For all $p \in N$ and $f \in \O_{Q, \pi(p)}$, there exists $F \in \O_{M, p}^{\mathrm{an}}$ such that $\pi^*f = F\big\vert_N \in \O_{N, p}^{\mathrm{an}}$ and $dF(TN^\omega) = 0$.
\item[\textup{(ii)}]
There is a unique algebraic Poisson structure $\{\cdot, \cdot\}_Q$ on $Q$ such that 
\[
\pi^*\{f, g\}_Q = \{F, G\}\big\vert_N
\]
for all $f, g \in \O_{Q, \pi(p)}$ and $F, G \in \O_{M, p}^{\mathrm{an}}$ related to $f, g$ as in \textup{Part (i)}.
\item[\textup{(iii)}]
Assume that there exists $p \in N$ such that $d\pi_p$ is surjective, $\pi^{-1}(\pi(p))$ is an $\mathcal{H}$-orbit, and $\pi(p)$ is a smooth point of $Q$.
The Poisson structure in \textup{Part (ii)} is then non-degenerate on a Zariski-open subset of the smooth locus of $Q$ containing $\pi(p)$.
\end{enumerate}
\end{thm}

\begin{proof}
To prove Part (i), let $D \coloneqq TN \cap TN^\omega$.
By Proposition \ref{kwi1nwnj}, we have $D_q = T_q(\H \cdot q)$ for all $q \in N$.
It follows that $$\pi^*f \in \O_{N, p} \s \O_{N, p}^{\text{an}}$$ and $d(\pi^*f)(D) = 0$, so we can find such an $F$ by Proposition \ref{8fzyp3dy}.

To prove Part (ii), let $U \s Q$ be Zariski-open.
Theorem \ref{zogb0t7j} provides a Poisson bracket $\{\cdot, \cdot\}'$ on $\O_N^{\mathrm{an}}(\pi^{-1}(U))^\H$.
It suffices to show that the subalgebra $\O_N(\pi^{-1}(U))^\H \cong \O_Q(U)$ of algebraic functions is closed under this bracket.

Let $f, g \in \O_Q(U)$, let $p \in \pi^{-1}(U)$, and let $F, G \in \O_{M, p}^{\text{an}}$ be related to the germs $f_{\pi(p)}, g_{\pi(p)} \in \O_{Q,\pi(p)}$ as in Part (i).
By definition, the germ of $\{\pi^*f, \pi^*g\}'$ at $p$ is $\{F, G\}\big\vert_N \in \O_{N, p}^{\mathrm{an}}$.
It therefore suffices to show that $\{F, G\}\big\vert_N$ is algebraic, i.e.\ that it lies in the image of the map $\O_{N, p} \longrightarrow \O_{N, p}^{\mathrm{an}}$.

Suppose that $X$ is affine for the moment.
Since $TS + \sigma(TS^\circ)$ has constant rank, one may choose an algebraic subbundle $R \subseteq TX\big\vert_{S}$ such that
\[
TX\big\vert_S = (TS + \sigma(TS^{\circ})) \oplus R.
\]
Remark \ref{Remark: Stein} and the proof of Proposition \ref{Proposition: Transversal germ} imply that there is a complex analytic Poisson transversal $Y \subseteq X$ such that
\[
TS \oplus R = TY\big\vert_S.
\]
By \cite[Lemma 7]{frejlich-marcut}, $Y$ and $\mu$ are transverse and $P \coloneqq \mu^{-1}(Y)$ is a complex symplectic submanifold of $M$.
It follows that
\begin{equation}\label{Equation: Definition of Q}
E \coloneqq d\mu^{-1}(TS \oplus R) = TP\big\vert_N
\end{equation}
is an algebraic subbundle of $TM\big\vert_N$ satisfying
\begin{equation}\label{Equation: Decomposition} 
TM\big\vert_N = E \oplus E^{\omega}.
\end{equation}
Since $E$ is locally trivial, we may extend the projection $ E \oplus E^\omega \longrightarrow E$ to an algebraic section $\theta$ of $\mathrm{End}(TM\big\vert_W)$ for some affine-open neighbourhood $W$ of $p$.

Noting that $F\big\vert_N = \pi^*f$ and $G\big\vert_N = \pi^*g$ are algebraic, we may choose $\tilde{F}, \tilde{G} \in \O_{M, p}$ satisfying $\tilde{F}\big\vert_N = F\big\vert_N$ and $\tilde{G}\big\vert_N = G\big\vert_N$.
We claim that
\begin{equation}\label{Equation: Preliminary identity}
\{F, G\}\big\vert_{N} = \omega(\theta(X_{\tilde{F}}),\theta(X_{\tilde{G}}))\big\vert_N
\end{equation}
as germs in $\O_{N, p}^{\mathrm{an}}$.
Since the right-hand side is in $\O_{N, p}$, this would prove Part (ii) in the case of an affine $X$.

To verify \eqref{Equation: Preliminary identity}, note that $N$ is coisotropic in $P$. 
We therefore have $TN^\omega \cap E \s TN$, which in turn forces $TN^\omega = D \oplus E^\omega$ to hold. It follows that
\[
\omega(\theta(X_{\tilde{F}}), TN^\omega) = \omega(\theta(X_{\tilde{F}}), D) = \omega(X_{\tilde{F}}, D) = d\tilde{F}(D) = d(\pi^*f)(D)=0,
\]
i.e.\ $\theta(X_{\tilde{F}})$ is tangent to $N$.
Since $dF(TN^\omega) = 0$, we get that $\theta(X_{\tilde{F}}) - X_F$ is also tangent to $N$.
On the other hand, the inclusion $TN \s E$ tells us that
\[
\omega(\theta(X_{\tilde{F}}) - X_F, TN) = \omega(X_{\tilde{F}} - X_F, TN) = d(\tilde{F} - F)(TN) = 0.
\]
We conclude that $\theta(X_{\tilde{F}}) - X_F$ takes values in $TN \cap TN^\omega = D$, and hence
\[
d\tilde{G}(\theta(X_{\tilde{F}}) - X_F)\big\vert_N = d(\pi^*g)(\theta(X_{\tilde{F}}) - X_F)\big\vert_N = 0.
\]
One then deduces that
\begin{align*}
\{F, G\}\big\vert_N &= \omega(X_{F}, X_{G})\big\vert_N = -dG(X_F)\big\vert_N = -d\tilde{G}(X_F)\big\vert_N = -d\tilde{G}(\theta(X_{\tilde{F}}))\big\vert_N = \omega(\theta(X_{\tilde{F}}), X_{\tilde{G}})\big\vert_N \\
&= \omega(\theta(X_{\tilde{F}}),\theta(X_{\tilde{G}}))\big\vert_N.
\end{align*}
This proves Part (ii) if $X$ is affine.
For the general case, one simply chooses an affine-open neighbourhood $V$ of $\mu(p)$ in $X$ and repeats the argument with $S$ and $X$ replaced by $S \cap V$ and $V$, respectively.

The proof of Part (iii) is entirely analogous to that of Theorem \ref{ytazg95b}(iii).
\end{proof}

\begin{defn}
The Poisson variety $Q$ in Theorem \ref{Theorem: Algebraic theorem} is called the \dfn{symplectic reduction of} $M$ \dfn{by} $\mathcal{G}$ \dfn{along} $S$ \dfn{with respect to} $\mathcal{H}$ \dfn{and} $\pi$. In this case, we denote $Q$ by $M\sll{S,\mathcal{H},\pi}\mathcal{G}$.
\end{defn}

\begin{rem}\label{hfs7e921}
Although the Poisson structure on $Q$ is algebraic, it may not always be possible to write $\pi^*\{f, g\}_Q = \{F, G\}\big\vert_N$ for \emph{algebraic} functions $F, G$ satisfying (i). To be more explicit, let $((M, \omega), \G \tto X, \mu)$, $(S, \H)$, and $\pi : N \too Q$ be as in Theorem \ref{Theorem: Algebraic theorem}.
Part (i) tells us that for all $p \in N$ and $f \in \O_{Q, \pi(p)}$, there exists an \emph{analytic} germ $F \in \O_{M, p}^{\mathrm{an}}$ such that $\pi^*f = F\big\vert_N$ and $dF(TN^\omega) = 0$.
It is natural to ask if $F$ can be taken to be algebraic.
This would certainly afford a more direct proof of the algebraicity of the Poisson bracket in Part (ii) (cf.\ Remark \ref{hfs7e921}).
While this approach works if $N$ is coisotropic, it fails in the general case.
Indeed, finding $F$ such that $dF(TN^\omega) = 0$ is essentially solving a first-order partial differential equation; its solutions have no reason to be algebraic.
To see this, consider the following example.

Take $M=\C^4$ with coordinates $(x, y, u, v)$ and symplectic form $\omega = dx \wedge dy + du \wedge dv$.
Consider the pair groupoid $\G \coloneqq M \times M^-$ acting on $M$ with moment map the identity $M \longrightarrow M$ (cf.\ Example \ref{ty6jqjxq}). 
Then $((M, \omega), \G \tto M, \mathrm{Id})$ is an algebraic Hamiltonian system.
We claim that the subvariety
\[
S \coloneqq \{(x, y, u, v) \in \C^4 : x^2 = u \ne 0, y = 0\}
\]
is pre-Poisson.

Observe that $TS$ is spanned by $\partial_x + 2x \partial_u$ and $\partial_v$.
We then find that $TS^\circ$ is spanned by $du - 2x dx$ and $dy$, and $\omega(TS)$ is spanned by $dx + 2xdv$ and $du$.
It follows that the stabilizer subalgebroid $L_S \coloneqq \omega(TS) \cap TS^\circ$ of $S$ is trivial, so that $S$ is pre-Poisson.
The trivial groupoid $\H$ over $S$ is then a stabilizer subgroupoid of $S$ in $\G$, and the identity map $S \too S$ is an algebraic quotient of $S$ by $\H$.
We are therefore in the context of Theorem \ref{Theorem: Algebraic theorem}.

On the other hand, we claim that there exist $p \in S$ and $f \in \O_{S, p}$ such that there is no algebraic germ $F \in \O_{M, p}$ with $F\big\vert_S = f$ and  $dF(TS^\omega) = 0$. To this end, let $f : S \too \C$ be given by $f(x, y, u, v) = x$.
Note that $TS^\omega$ is spanned by $2x \partial_y - \partial_v$ and $\partial_x$.
It follows that a germ $F \in \O_{M, p}$ for $p \in S$ satisfies $dF(TS^\omega) = 0$ if and only if
\[
2x \frac{\partial F}{\partial y} - \frac{\partial F}{\partial v} = 0 \quad\text{and}\quad \frac{\partial F}{\partial x} = 0
\]
on $S$.
Solving this partial differential equation with the constraint $F\big\vert_S = f$, e.g. by the method of characteristics, we find that $$F = \sqrt{u} + g$$ for any $g \in \I_S^2$, where $\I_S$ is the analytic ideal of $S$ and the square root is defined by a branch of the logarithm such that $\sqrt{x^2} = x$ near $p$.
Note that $u \ne 0$ on $S$, so $F$ is well-defined as a germ at $p$.
But there is no $g \in \I_S^2$ such that $\sqrt{u} + g$ is algebraic.
\end{rem}

\subsection{Poisson and hyperk\"ahler slices}\label{Subsubsection: Poisson and hyperkahler slices}
Suppose that $G$ is a connected complex semisimple linear algebraic group with Lie algebra $\mathfrak{g}$. Use the Killing form to freely identify $\mathfrak{g}$ and $\mathfrak{g}^*$ as $G$-modules. The Lie algebra $\mathfrak{g}$ thereby inherits a Poisson variety structure from $\mathfrak{g}^*$. At the same time, recall that $\tau=(e,h,f)\in\mathfrak{g}^{3}$ is called an $\mathfrak{sl}_2$-triple if the identities $$[e,f]=h,\quad [h,e]=2e, \quad\text{and}\quad [h,f]=-2f$$ hold in $\mathfrak{g}$. One then has an associated Slodowy slice
$$\mathcal{S}_\tau\coloneqq e+\mathfrak{g}_f\subseteq\mathfrak{g},$$ which is known to be a Poisson transversal in $\mathfrak{g}$ \cite[Section 3.1]{gan-ginzburg}.

Fix an $\mathfrak{sl}_2$-triple $\tau\in\mathfrak{g}^{3}$ and suppose that $\mu:M\longrightarrow\mathfrak{g}$ is a Hamiltonian $G$-variety, i.e.\ an algebraic Hamiltonian system $((M, \omega), \G\tto X, \mu)$ with $(\mathcal{G}\tto X)=(T^*G\tto\mathfrak{g})$. The previous paragraph and Example \ref{6sap623o} tell us that
$$M_{\tau}\coloneqq\mu^{-1}(\mathcal{S}_{\tau})$$ is a symplectic subvariety of $M$, and that the trivial groupoid is a stabilizer subgroupoid of $\mathcal{S}_{\tau}$. It follows that $M_{\tau}$ is the symplectic reduction of $M$ by $T^*G\tto\g^*$ with respect to $S_{\tau}$ and the trivial stabilizer subgroupoid $\mathcal{H}\tto\mathcal{S}_{\tau}$.  

The symplectic variety $M\sll{\mathcal{S}_{\tau}}\mathcal{G}=M_{\tau}$ has been studied in each of the following contexts.
\begin{itemize}
	\item[(i)] One obtains $M_{\tau}$ as a special case of the \textit{Poisson slice} construction in \cite{crooks-roeser}.
	\item[(ii)] If $M$ is hyperk\"ahler and satisfies a mild condition, then $M_{\tau}$ is hyperk\"ahler \cite[Theorem 1]{bielawski97} and called a \textit{hyperk\"ahler slice} \cite{bielawski}. The study of hyperk\"ahler slices originates in Bielawski's works \cite{bielawski97,bielawski} and occurs in the subsequent papers \cite{crooks-rayan,crooks-vanpruijssen}.
	\item[(iii)] Consider the symplectic manifold $T^*G$ equipped with the action of $G$ by right translations. One then finds that $(T^*G)_\tau = \mathfrak{M}_{G, \mathcal{S}_\tau} = G \times \mathcal{S}_\tau$. This symplectic subvariety of $T^*G$ is sometimes called a \textit{Nahm pole}, and it receives considerable attention in \cite{bielawski97,bielawski,crooks-rayan,moore-tachikawa}. It enjoys a universal property, i.e.\ $M_\tau \cong (M \times (G \times \mathcal{S}_\tau) \sll{0} G$ under certain hypotheses (see \cite[Theorem 1(ii)]{bielawski} and \cite[Proposition 4.2]{crooks-vanpruijssen}). This can now be viewed as a special case of Theorem \ref{universality}.
\end{itemize}

\subsection{The Moore--Tachikawa TQFT}\label{Subsection: The Moore-Tachikawa TQFT}
Retain the Lie-theoretic objects and conventions established in the first paragraph of \S\ref{Subsubsection: Poisson and hyperkahler slices}. Consider the locus of regular elements
$$\mathfrak{g}_{\text{reg}}\coloneqq\{x\in\mathfrak{g}:\dim(\mathfrak{g}_x)=\ell\}\subseteq\mathfrak{g},$$ where $\ell$ is the rank of $\mathfrak{g}$ and $\mathfrak{g}_x\subseteq\mathfrak{g}$ denotes the $\mathfrak{g}$-centralizer of $x\in\mathfrak{g}$. Let us also fix a principal $\sl_2$-triple, i.e.\ an $\mathfrak{sl}_2$-triple $\tau=(e, h, f)\in\mathfrak{g}^{3}$ for which $e,h,f\in\mathfrak{g}_{\text{reg}}$. The Slodowy slice  $$\mathcal{S}\coloneqq\mathcal{S}_{\tau}$$ is a fundamental domain for the action of $G$ on $\mathfrak{g}_{\text{reg}}$ \cite[Theorem 8]{kostant}, and it is a section of the adjoint quotient
$$\chi:\mathfrak{g}\longrightarrow\mathrm{Spec}(\mathbb{C}[\mathfrak{g}]^G)\eqqcolon\mathfrak{c}$$
\hspace{-0.00001pt}\cite[Theorem 7]{kostant}.
This gives context for a conjecture of Moore--Tachikawa \cite{moore-tachikawa} and its appearances in the literature \cite{arakawa,bielawskipreprint,braverman-finkelberg-nakajima, calaque-2015}. 

Moore and Tachikawa \cite{moore-tachikawa} have conjectured the existence of certain topological quantum field theories (TQFTs) valued in $\mathrm{HS}$, the category of affine symplectic varieties with Hamiltonian actions. This category has complex semisimple linear algebraic groups as its objects. A morphism $H_1\longrightarrow H_2$ is an equivalence class of an affine symplectic $(H_1\times H_2\times\mathbb{C}^{\times})$-variety\footnote{In this one case, ``symplectic variety" refers to a potentially singular Poisson variety that is symplectic on an open dense subset of its smooth locus. Symplectic varieties are otherwise understood to be smooth and everywhere symplectic in our paper.} for which the symplectic form has $\mathbb{C}^{\times}$-weight $-2$ and the action of $H_1\times H_2$ is Hamiltonian; two such varieties are declared to be equivalent if they are $(H_1\times H_2\times\mathbb{C}^{\times})$-equivariantly symplectomorphic. In this context, one defines morphism composition as follows:
$$X_2\circ X_1\coloneqq(X_1\times X_2)\sll{} H_2\in\mathrm{Hom}_{\mathrm{HS}}(H_1,H_3)$$ for all $X_1\in\mathrm{Hom}_{\mathrm{HS}}(H_1,H_2)$ and $X_2\in\mathrm{Hom}_{\mathrm{HS}}(H_2,H_3)$, where $(X_1\times X_2)\sll{} H_2$ is the affine GIT quotient by $H_2$ of the zero-level set of the $H_2$-moment map on $X_1\times X_2$. The identity object in $\mathrm{Hom}_{\mathrm{HS}}(H,H)$ is then $T^*H$, equipped with its usual Hamiltonian ($H\times H$)-action and its fibrewise dilation action of $\mathbb{C}^{\times}$. 
    
It turns out that $\mathrm{HS}$ is a symmetric monoidal category with duality; the symmetric monoidal structure is with respect to the usual products of linear algebraic groups and the usual products of affine symplectic varieties. The dual of $H\in\mathrm{Ob}(\mathrm{HS})$ is defined to be $H$ itself, while the corresponding evaluation and coevaluation morphism are $$T^*H\in\mathrm{Hom}_{\mathrm{HS}}(H\times H,\{e\})\quad\text{and}\quad T^*H\in\mathrm{Hom}_{\mathrm{HS}}(\{e\},H\times H),$$ respectively. On the other hand, one knows that two-dimensional cobordisms form a symmetric monoidal category $\mathrm{BO}_2$ with duality. These considerations lead to the following definition of Moore--Tachikawa \cite{moore-tachikawa}.

\begin{defn}\label{Definition: TQFT}
A \textit{two-dimensional $\mathrm{HS}$-valued TQFT} is a symmetric monoidal functor $\eta:\mathrm{BO}_2\longrightarrow\mathrm{HS}$ that respects duality.  
\end{defn}

Given a positive integer $n$, let $$C_n\in\mathrm{Hom}_{\mathrm{BO}_2}(\underbrace{S^1\sqcup\cdots\sqcup S^1}_{n\text{ times}},\emptyset)$$ denote basic two-dimensional cobordism with $n$ incoming circles. The Moore--Tachikawa conjecture \cite{moore-tachikawa} is then formulated as follows.

\begin{conj}\label{Conjecture: Moore--Tachikawa}
If $G$ is a connected complex simple linear algebraic group, then there exists a two-dimensional $\mathrm{HS}$-valued TQFT $\eta_G:\mathrm{BO}_2\longrightarrow\mathrm{HS}$ satisfying $\eta_G(S^1)=G$ and $\eta_G(C_1)=G\times\mathcal{S}$.
\end{conj}

\begin{rem}
The condition $\eta_G(C_1)=G\times\mathcal{S}$ is imposed for physical reasons, as is explained in \cite[Section 3.3]{moore-tachikawa}.
\end{rem}

Let $G$ be a connected complex simple linear algebraic group and $\eta_G:\mathrm{BO}_2\longrightarrow\mathrm{HS}$ a TQFT satisfying Conjecture \ref{Conjecture: Moore--Tachikawa}. The Hamiltonian $G^n$-varieties
$$\eta_G(C_n),\quad n\geq 1$$ are sometimes called \dfn{Moore--Tachikawa varieties}. Examples include $\eta_G(C_1)=G\times\mathcal{S}$ and $\eta_G(C_2)=T^*G$, the second assertion being a straightforward consequence of Definition \ref{Definition: TQFT} and Conjecture \ref{Conjecture: Moore--Tachikawa}.
One also knows that the third Moore--Tachikawa variety $\eta_G(C_3)$ determines $\eta_G$ \cite[Section 3.2]{moore-tachikawa}. Proving Conjecture \ref{Conjecture: Moore--Tachikawa} thereby amounts to finding a suitable candidate for $\eta_G(C_3)$.

\begin{rem}
The conjecture is known to be true for $G = \mathrm{SL}(2, \C)$ and $G = \mathrm{SL}(3, \C)$ \cite{moore-tachikawa}. One has $\eta_{\mathrm{SL}(2, \C)}(C_3) = \C^2 \otimes \C^2 \otimes \C^2$, while $\eta_{\mathrm{SL}(3, \C)}(C_3)$ is the closure of the minimal nilpotent orbit in the exceptional Lie algebra $E_6$.
\end{rem}

Ginzburg and Kazhdan \cite{ginzburg-kazhdan} construct candidates for the Moore--Tachikawa varieties as follows.   
Recall our discussion in \S\ref{Subsubsection: Poisson and hyperkahler slices}(iii) of the symplectic subvariety $$\mathcal{N}\coloneqq G\times\mathcal{S}\subseteq G\times\g=T^*G$$ and adjoint quotient morphism $\chi:\mathfrak{g}\longrightarrow\c$. One then defines the \dfn{universal centralizer} of $G$ to be the closed subvariety
\[
\mathcal{J}\coloneqq\{(g,x)\in\mathcal{N}:g\in G_x\}\subseteq\mathcal{N},
\]
where $G_x\subseteq G$ denotes the $G$-stabilizer of $x\in\mathfrak{g}$. We have the morphism
\[
\mathcal{N}\longrightarrow\c,\quad (g,x)\mtoo\chi(x)
\]
and its restriction
\[
\mathcal{J}\longrightarrow\c,\quad (g,x)\mtoo\chi(x)
\]
to $\mathcal{J}$. The latter is an abelian group scheme over $\c$, and it acts on $\mathcal{N}$ via the map $$\mathcal{J} \times_\c \mathcal{N} \longrightarrow \mathcal{N},\quad ((g,x),(h,x))\mtoo (hg^{-1},x),\quad x\in\mathcal{S},\text{ }g\in G_x,\text{ }h\in G.$$ The kernel
\[
\mathcal{J}_n \coloneqq \ker(\underbrace{\mathcal{J} \times_\c \cdots \times_\c \mathcal{J}}_{n\text{ times}} \longrightarrow \mathcal{J})
\]
of multiplication is then a group scheme over $\c$ acting on
\[
\mathcal{N}_n\coloneqq\underbrace{\mathcal{N} \times_\c \cdots \times_\c \mathcal{N}}_{n\text{ times}}\subseteq\mathcal{N}^n.
\]
Ginzburg--Kazhdan \cite{ginzburg-kazhdan} prove that $\mathcal{N} \times_\c \cdots \times_\c \mathcal{N}$ is coisotropic in $\mathcal{N}^n$, and that the corresponding null-foliation is formed by the $\mathcal{J}_n$-orbits. They also show the geometric quotient
\[
\mathcal{Z}_n^\circ \coloneqq \mathcal{N}_n / \mathcal{J}_n
\]
to be a smooth algebraic variety, so that $\mathcal{Z}_n^\circ$ is necessarily symplectic.  The left translation action of $G^n$ on $T^*G^n$ then descends to a Hamiltonian $G^n$-action on $\mathcal{Z}_n^\circ$. 

Ginzburg and Kazhdan subsequently define $\mathcal{Z}_n$ to be the affinization of $\mathcal{Z}_n^\circ$, i.e.\ the $G^n$-scheme
\[
\mathcal{Z}_n \coloneqq (\mathcal{Z}_n^\circ)^{\mathrm{aff}} = \operatorname{Spec}(\C[\mathcal{N}_n]^{\mathcal{J}_n}).
\]
If $\C[\mathcal{N}_n]^{\mathcal{J}_n}$ is finitely generated for all $n\geq 1$, then the $\mathcal{Z}_n$ are Moore--Tachikawa varieties that verify Conjecture \ref{Conjecture: Moore--Tachikawa} \cite{ginzburg-kazhdan}.

Fix an integer $n\geq 1$ and let
$$\Delta_n\mathcal{S}\subseteq\g^{n}$$ denote the diagonally embedded copy of $\mathcal{S}$ in $\g^{n}$. We may identify $\c$ with $\mathcal{S}$ via the restricted adjoint quotient $$\chi\big\vert_{\mathcal{S}}:\mathcal{S}\overset{\cong}\longrightarrow\c,$$ and $\mathcal{S}$ with $\Delta_n\mathcal{S}$ in the obvious way. The three varieties $\c$, $\mathcal{S}$, and $\Delta_n\mathcal{S}$ are thereby treated interchangeably in what follows. One immediate consequence is our ability to view $\mathcal{J}_n$ as a group scheme over $\Delta_n\mathcal{S}$.
At the same time, our identification $T^*G=G\times\mathfrak{g}$ induces a further identification $T^*G^n=(G\times\g)^n$. The set-theoretic inclusion $\mathcal{J}_n\subseteq (G\times\mathfrak{g})^n$ then renders $\mathcal{J}_n\longrightarrow\Delta_n\mathcal{S}$ a subgroupoid of $T^*G^n\tto\mathfrak{g}^{n}$.

\begin{thm}[Theorem \ref{Theorem: Specific examples}(i)]\label{Theorem: Moore-Tachikawa}
The following statements hold for all integers $n\geq 1$.
\begin{itemize}
\item[\textup{(i)}] The diagonal $\Delta_n\mathcal{S}\subseteq\g^{n}$ is stable\footnote{The algebro-geometric definition of ``stable" is identical to the manifold-theoretic one in \S\ref{peg44up7}.} and admits $\mathcal{J}_n$ as a stabilizer subgroupoid in $T^*G^n\tto\g^{n}$.
\item[\textup{(ii)}] We have 
\[
\mathcal{Z}_n^{\circ} = T^*G^n \sll{(\Delta_n\mathcal{S},\mathcal{J}_n,\pi_n)} G^n,
\]
where $\pi_n:\mathcal{N}_n\longrightarrow \mathcal{Z}_n^{\circ}$ is the geometric quotient map and $G^n$ acts on $T^*G^n$ by right translations.
\item[\textup{(iii)}] Assume that $\C[\mathcal{N}_n]^{\mathcal{J}_n}$ is finitely generated. We have
\[
\mathcal{Z}_n = T^*G^n \sll{(\Delta_n\mathcal{S},\mathcal{J}_n,\bar{\pi}_n)} G^n,
\]
where $\bar{\pi}_n:\mathcal{N}_n\longrightarrow \mathcal{Z}_n$ is the affine GIT quotient map, $G^n$ acts on $T^*G^n$ by right translations.
\end{itemize}
\end{thm}

\begin{proof}
Our arguments make extensive use of the fact that $\mathcal{S}$ is a transverse slice to the adjoint action \cite[Section 3.1]{gan-ginzburg}. This amounts to the statement that
\begin{equation}\label{qrv7yz97}
\g = \g_f \oplus [\g, x] \quad \text{for all }x \in \mathcal{S},
\end{equation}
where we have identified $T_x\mathcal{S}$ with $\g_f$ and $T_x(G \cdot x)$ with $[\g, x]$.

Lemma \ref{subalgebroid-lemma} implies that
\[
(L_{\Delta_n\mathcal{S}})_{\underline{x}} = \{(y_1,\ldots,y_n)\in\mathfrak{g}^{n}: y_1 + \cdots + y_n \in \g_f^\perp \text{ and } [x, y_1] = \cdots = [x, y_n] \in \g_f\},
\]
for all $\underline{x} = (x, \ldots, x) \in \Delta_n \mathcal{S}$, where $\g_f^\perp$ is the annihilator of $\g_f$ under the Killing form.
On the other hand \eqref{qrv7yz97} tells us that $\g_f \cap [\g, x] = 0$ for all $x\in\mathcal{S}$. One consequence is as follows: if $[x, y_i] \in \g_f$, then $[x, y_i] = 0$, i.e.\ $y_i \in \g_x$.
Equation \eqref{qrv7yz97} also implies that $\g_f^\perp \cap \g_x = (\g_f + [\g, x])^\perp = 0$.
These considerations and the observation that
\begin{equation}\label{Equation: Explicit description}\mathcal{J}_n=\bigg\{\bigg((g_1,x),\ldots,(g_n,x)\bigg):x\in\mathcal{S},\text{ }g_1,\ldots,g_n\in G_x,\text{ and }g_1\cdots g_n=1\bigg\}\end{equation}
give
\begin{equation}\label{9hqn8a8y}
(L_{\Delta_n\mathcal{S}})_{\underline{x}} = \{(y_1,\ldots,y_n) \in (\g_x)^{n}: y_1 + \cdots + y_n = 0\} = \Lie((\mathcal{J}_n)_{\underline{x}})
\end{equation}
for all $x \in \mathcal{S}$, where $(\mathcal{J}_n)_{\underline{x}}$ is the fibre of $\mathcal{J}_n$ over $\underline{x} \in \Delta_n\mathcal{S}$.
We also know that $\dim \g_x = \dim \g_y$ for all $x, y \in \mathcal{S}$, as $\mathcal{S} \s \g_{\mathrm{reg}}$.
It follows that $\dim(L_{\Delta_n\mathcal{S}})_{\underline{x}}$ is independent of $x \in \mathcal{S}$, i.e.\ $\Delta_n\mathcal{S}$ is pre-Poisson.
We also see that $(L_{\Delta_n\mathcal{S}})_{\underline{x}} \s (\g_x)^n$ for all $x \in \mathcal{S}$, so that $\Delta_n\mathcal{S}$ is stable.

In light of \eqref{9hqn8a8y} and Theorem \ref{amc0wisd}(i)--(iii), showing $(\mathcal{J}_n)_{\underline{x}}$ to be connected for all $x$ in an open dense subset of $\mathcal{S}$ will force $\mathcal{J}_n$ to be a stabilizer subgroupoid of $\Delta_n\mathcal{S}$. To this end, \eqref{Equation: Explicit description} tells us that 
$(\mathcal{J}_n)_{\underline{x}} \cong (G_x)^{n - 1}$ for all $x\in\mathcal{S}$. We also know $G_x$ to be a maximal torus if $x\in\mathcal{S}$ is semisimple. It follows that $(\mathcal{J}_n)_{\underline{x}}$ is indeed connected for all $x$ in an open dense subset of $x$, completing the proof of Part (i).

To establish the remaining parts, consider Equation \eqref{1h6um459}. This implies that the moment map for the action of $G^n$ on $T^*G^n$ by right translations is
\[
\mu:(G\times\mathfrak{g})^n\longrightarrow\mathfrak{g}^{n},\quad ((g_1,x_1),\ldots,(g_n,x_n))\mtoo (x_1,\ldots,x_n).
\]
A further observation is that
\[
\mu^{-1}(\Delta_n\mathcal{S}) = \mathcal{N}\times_{\mathcal{S}}\cdots\times_{\mathcal{S}}\mathcal{N} = \mathcal{N} \times_\c \cdots \times_\c \mathcal{N} = \mathcal{N}_n.
\]
Parts (ii) and (iii) then follow immediately from the definitions of $\mathcal{Z}_n^{\circ}$ and $\mathcal{Z}_n$, respectively.
\end{proof}

\begin{rem}
If $G$ is of adjoint type, then the $G$-stabilizers of regular elements are connected.
It follows that $\mathcal{J}_n$ is a source-connected stabilizer subgroupoid of $\Delta_n\mathcal{S}$ in $T^*G^n \tto \g^n$.
We thereby obtain the simplified expressions
\[
\mathcal{Z}^\circ_n = \mathfrak{M}_{G^n, \Delta_n\mathcal{S}} \quad\text{and}\quad \mathcal{Z}_n = (\mathfrak{M}_{G^n, \Delta_n\mathcal{S}})^{\mathrm{aff}},
\]
where $(\mathfrak{M}_{G^n, \Delta_n\mathcal{S}})^{\mathrm{aff}}$ is the affinization of $\mathfrak{M}_{G^n, \Delta_n\mathcal{S}}$.
\end{rem}

\section{Shifted symplectic interpretation}\label{95hfl28k}

We now interpret the notion of symplectic reduction along a submanifold in the context of shifted symplectic geometry \cite{ptvv}.
More precisely, we view our construction as a derived intersection of two Lagrangians in a 1-shifted symplectic stack and obtain a $0$-shifted symplectic structure from \cite[Theorem 2.9]{ptvv}.
There are similar statements in the literature for Marsden--Weinstein--Meyer reduction \cite[\S2.1.2]{calaque} and quasi-Hamiltonian reduction \cite{safronov-2016}.
We refer the reader to \cite{calaque, getzler, ptvv, pym-safronov} for the relevant background on shifted symplectic geometry.

Given a symplectic groupoid $\G \tto X$, recall that the quotient stack $[X/\G]$ inherits a $1$-shifted symplectic structure (see e.g.\ \cite[\S1.2.3]{calaque}, \cite{getzler} and \cite[Proposition 3.31]{safronov-2021}).
Recall also that a Hamiltonian action of $\G \tto X$ on a symplectic manifold $(M, \omega)$ with moment map $\mu : M \longrightarrow X$ gives rise to a Lagrangian structure on the map $[\mu] : [M/\G] \longrightarrow [X/\G]$ \cite[Example 1.31]{calaque}.

Consider a Lie subgroupoid $\H \tto S$ of $\G \tto X$ with morphism $f : \H \longrightarrow \G$.
An isotropic structure on $[f] : [S / \H] \longrightarrow [X/\G]$ is a closed 2-form $\beta$ on $S$ such that $f^*\Omega = \sss^*\beta - \ttt^*\beta,$ where $\Omega$ is the symplectic form on $\G$ (see \cite[Definition 2.7]{ptvv} or \cite[\S1.3.2]{calaque}).
The isotropic structure $\beta$ is called Lagrangian if the induced morphism $\mathbb{T}_f \longrightarrow \mathbb{L}_{[S/\H]}$ is a quasi-isomorphism \cite[Definition 2.8]{ptvv}, where $\mathbb{T}$ and $\mathbb{L}$ denote the tangent and cotangent complexes, respectively, and $$\mathbb{T}_f \coloneqq f^*\mathbb{T}_{[X/\G]}[-1] \oplus \mathbb{T}_{[S/\H]}$$ is the relative tangent complex.

\begin{prop}[Theorem \ref{Theorem: Shifted}(i)]\label{yrb59dpc}
Let $\H \tto S$ be a Lie subgroupoid of a symplectic groupoid $\G \tto X$.
The zero 2-form on $S$ is a Lagrangian structure on $[S / \H] \longrightarrow [X / \G]$ if and only if $\H$ is the stabilizer subgroupoid of a pre-Poisson submanifold.
\end{prop}

\begin{proof}
Note that $\beta = 0$ is an isotropic structure if and only if $\H$ is isotropic in $\G$.
We may therefore assume that $\H$ is isotropic, and aim to show that the trivial isotropic structure on $[f] : [S/\H] \longrightarrow [X/\G]$ is Lagrangian if and only if the Lie algebroid $L$ of $\H$, viewed as a subalgebroid of $T^*X = \operatorname{Lie}(\G)$, is equal to $\sigma^{-1}(TS) \cap TS^\circ$, where $\sigma : T^*X \longrightarrow TX$ is the Poisson structure.

Recall that the tangent complex of the quotient stack of a groupoid is the two-term complex given by the anchor map in degrees $-1$ and $0$ (see e.g.\ \cite[\S1.2.3]{calaque}).
It follows that the map of complexes $\mathbb{T}_f \longrightarrow \mathbb{L}_{[S/\H]}$ is given by
\[
\begin{tikzcd}
L \arrow{r}{\alpha} \arrow{d}
& T^*X\big\vert_S \oplus TS \arrow{r}{\beta} \arrow{d}{\gamma}
& TX\big\vert_S \arrow{d}{\delta} \\
0 \arrow{r} & T^*S \arrow{r}{\varepsilon} & L^*,
\end{tikzcd}
\]
where the maps are defined as follows:
\begin{itemize}
\item
$\alpha$ is the sum of the inclusion $L \longrightarrow T^*X\big\vert_S$ and the anchor map $\sigma_L : L \longrightarrow TS$;
\item
$\beta$ is the sum of $\sigma$ and the inclusion $TS \longrightarrow TX\big\vert_S$;

\item
$\gamma$ is the sum of the restriction map $T^*X\big\vert_S \longrightarrow T^*S$ and the map $TS \longrightarrow T^*S$ given by the isotropic structure (which is zero in this case);
\item
$\delta$ is the restriction map;
\item
$\varepsilon$ is the dual of the anchor map $\sigma_L$ of $L$.
\end{itemize}
Since $\H$ is isotropic in $\G$, we have $L \s TS^\circ$ and this diagram commutes.

Passing to the map on cohomologies, we get
\[
\begin{tikzcd}
0 \arrow{d} & \sigma^{-1}(TS)/L \arrow{d}{\varphi} & TX\big\vert_S/(TS + \sigma(T^*X\big\vert_S)) \arrow{d}{\psi}\\
0 & \mathrm{Ann}_{T^*S}(\sigma(L)) & L^*/\sigma_L^*(T^*S).
\end{tikzcd}
\]
The trivial isotropic structure is Lagrangian if and only if $\varphi$ and $\psi$ are isomorphisms.
Note that
\begin{equation}\label{lr86m34x}
\ker \varphi = (\sigma^{-1}(TS) \cap TS^\circ) / L.
\end{equation}
Hence, if the isotropic structure is Lagrangian, then $L = \sigma^{-1}(TS) \cap TS$. 
Conversely, suppose that $L = \sigma^{-1}(TS) \cap TS^\circ$.
We then have $\ker \varphi = 0$ by \eqref{lr86m34x}.
By identifying $\mathrm{Ann}_{T^*S}(\sigma(L))$ with $\sigma(L)^\circ / TS^\circ$, the image of $\varphi$ is $\sigma^{-1}(TS) / TS^\circ$.
But $$\sigma(L)^\circ = (TS \cap \sigma(TS^\circ))^\circ = TS^\circ + \sigma^{-1}(TS),$$ so $\sigma(L)^\circ / TS^\circ = \sigma^{-1}(TS) / TS^\circ$ and hence $\varphi$ is surjective.
On the other hand, note that $$L^* / \sigma_L^*(T^*S) = TX\big\vert_S / (\sigma(T^*X\big\vert_S) + L^\circ).$$ It follows that $\psi$ is an isomorphism if and only if $$\sigma(T^*X\big\vert_S) + L^\circ = \sigma(T^*X\big\vert_S) + TS.$$
The latter equality holds since $L^\circ = (\sigma^{-1}(TS) \cap TS^\circ)^\circ = \sigma(TS^\circ) + TS.$
\end{proof}

By combining Proposition \ref{yrb59dpc} and \cite[Theorem 2.9]{ptvv}, we obtain the following result.

\begin{thm}[Theorem \ref{Theorem: Shifted}(ii)]\label{pw00lbch}
Let $((M, \omega), \G \tto X, \mu)$ be an algebraic Hamiltonian system, $S \s X$ a pre-Poisson subvariety, and $\H \tto S$ a stabilizer subgroupoid of $S$ in $\G$.
Then the derived fibre product
\[
[M \sll{S} \G] \coloneqq [M / \G] \times_{[X  / \G]}^h [S / \H]
\]
has a canonical $0$-shifted symplectic structure.\qed
\end{thm}

\begin{ex}
If $\G \tto X$ in Theorem \ref{pw00lbch} is a cotangent groupoid $T^*G \tto \g^*$ and $S = \{0\}$, we recover the derived version of Marsden--Weinstein--Meyer reduction at level zero \cite[\S2.1.2]{calaque}.
\end{ex}

\section*{Notation}
{\renewcommand{\arraystretch}{1.0}
\begin{flushleft}
\begin{longtable}[l]{cp{0.85\textwidth}}
$\mathcal{O}_X$ & structure sheaf of $X$\\[5pt]
$\mathcal{O}_X^E$ & subsheaf of all $f\in\mathcal{O}_X$ satisfying $df(E)=0$\\[5pt] 
$\mathcal{F}_x$ & stalk of a sheaf $\mathcal{F}$ at a point $x$\\[5pt]
$f_x$ & class of $f\in\mathcal{F}$ in the stalk $\mathcal{F}_x$ \\[5pt]
$M^{\text{an}}$ & complex analytic space associated to a complex algebraic variety $M$\\[5pt]
$M^{-}$ & manifold $M$ endowed with the negated symplectic form\\[5pt]
$E_x$ & fibre of a vector bundle $E$ over a point $x$\\[5pt]
$E\big\vert_S$ & pullback of a vector bundle $E\longrightarrow X$ to a submanifold $S\subseteq X$\\[5pt]
$V^{\circ}$ & annihilator of a vector space $V$ in $W^*$, where $W$ is some ambient vector space containing $V$ as a subspace\\[5pt]
$E^{\circ}$ & annihilator of a vector bundle $E\longrightarrow S$ in $F^*\big\vert_S$, where $F\longrightarrow X$ is some ambient vector bundle, $S\subseteq X$ is a submanifold, and $E$ is a subbundle of $F\big\vert_S$\\[5pt]
$E^{\omega}$ & annihilator of a subbundle $E\longrightarrow N$ of $TM\big\vert_N$ with respect to $\omega$, where $(M,\omega)$ is some ambient symplectic manifold and $N\subseteq M$ is a submanifold\\[5pt]
$\mathcal{G}\tto X$ & Lie groupoid\\[5pt]
$\sss$ & source $\sss:\mathcal{G}\longrightarrow X$\\[5pt]
$\ttt$ & target $\ttt:\mathcal{G}\longrightarrow X$\\[5pt]
$\mathcal{G}\big\vert_S$ & restriction $\sss^{-1}(S)\cap\ttt^{-1}(S)$ of $\mathcal{G}$ to a submanifold $S\subseteq X$\\[5pt]
$\mathcal{G}_x$ & fibre $\ttt^{-1}(x)$ over $x\in X$\\[5pt]
$\mathcal{G}\cdot p$ & orbit of a point $p$ in a space equipped with an action of a Lie groupoid $\mathcal{G}\tto X$\\[5pt]
$\mathrm{Lie}(\mathcal{G})$ & Lie algebroid of a Lie groupoid $\mathcal{G}\tto X$\\[5pt]
$[X/\G]$ & quotient stack of a Lie groupoid $\G \tto X$\\[5pt]
$\mathbb{T}_{[X/\G]}$ & tangent complex of a quotient stack $[X / \G]$\\[5pt]
$\mathbb{L}_{[X/\G]}$ & cotangent complex of a quotient stack $[X/\G]$\\[5pt]
$\mathbb{T}_f$ & relative tangent complex of a morphism of stacks $f$\\[5pt]
$\times^h$ & derived fibre product\\[5pt]
$(X,\sigma)$ & Poisson manifold $X$ with Poisson bivector field $\sigma:T^*X\longrightarrow TX$\\[5pt]
$X_f$ & Hamiltonian vector field of a function $f$\\[5pt]
$L_S$ & the subset $\sigma^{-1}(TS)\cap TS^{\circ}\subseteq T^*X$ for $(X,\sigma)$ a Poisson manifold and $S\subseteq X$ a submanifold\\[5pt]
$\mathcal{G}_S\tto S$ & unique source-connected, source-simply-connected stabilizer subgroupoid of a pre-Poisson submanifold $S\subseteq X$ in a symplectic groupoid $\mathcal{G}\tto X$\\[5pt]
$M \sll{\xi} G$ & Marsden--Weinstein--Meyer reduction of $M$ by $G$ at level $\xi$\\[5pt]
$M\sll{S,\mathcal{H}}\mathcal{G}$ & symplectic reduction of $M$ by $\mathcal{G}$ along $S$ with respect to $\mathcal{H}$\\[5pt]
$M\sll{S}\mathcal{G}$ & symplectic reduction of $M$ by $\mathcal{G}$ along $S$ with respect to any source-connected $\mathcal{H}$\\[5pt]
$\mathfrak{M}_{G,S,\mathcal{H}}$ & symplectic reduction of $M$ by $T^*G$ along $S$ with respect to $\mathcal{H}$\\[5pt]
$\mathfrak{M}_{G,S}$ & symplectic reduction of $M$ by $T^*G$ along $S$ with respect to any source-connected $\mathcal{H}$\\[5pt]
$(S,\mathcal{H})$ & reduction datum of a pre-Poisson submanifold $S$ and a stabilizer subgroupoid $\mathcal{H}\tto S$\\[5pt]
$M\sll{S,\mathcal{H},\pi}\mathcal{G}$ & symplectic reduction of $M$ by $\mathcal{G}$ along $S$ with respect to $\mathcal{H}$ and $\pi$\\[5pt]
$\mathrm{Ad}$ & adjoint representation of a Lie group\\[5pt]
$\mathrm{Ad}^*$ & coadjoint representation of a Lie group\\[5pt]
$\mathrm{ad}$ & adjoint representation of a Lie algebra\\[5pt]
$\mathrm{ad}^*$ & coadjoint representation of a Lie algebra\\[5pt]
$G_x$ & stabilizer of a point $x$ in a space equipped with an action of a Lie group $G$\\[5pt]
$\g_x$ & centralizer of a point $x\in\g$, where $\g$ is a Lie algebra\\[5pt]
$\g_\xi$ & centralizer of a point $\xi\in\g^*$, where $\g$ is a Lie algebra\\[5pt]
$\g_{\text{reg}}$ & set of regular elements in a complex semisimple Lie algebra $\g$\\[5pt]
$\g_{\text{irr}}$ & complement of $\g_{\text{reg}}$ in $\g$\\[5pt]
$\g_{\text{s}}$ & set of semisimple elements in $\g$\\[5pt]
$\g_{\text{subreg}}$ & set of subregular elements in $\g$\\[5pt]
$\g_{\text{irr}}^{\circ}$ & the intersection $\g_\text{s}\cap\g_{\text{subreg}}$\\[5pt]
$\c$ & the affine space $\mathrm{Spec}(\mathbb{C}[\g]^G)$, where $G$ is a connected complex semisimple linear algebraic group with Lie algebra $\g$\\[5pt]
$\chi:\g\longrightarrow\c$ & adjoint quotient of a complex semisimple Lie algebra $\g$\\[5pt]
$\mathcal{N}$ & Nahm pole $G\times\mathcal{S}$, where $G$ is a connected complex semisimple linear algebraic group with Lie algebra $\g$ and $\mathcal{S}\subseteq\g$ is a principal Slodowy slice\\[5pt]
$\mathcal{N}_n$ & $n$-fold fibred product $\mathcal{N} \times_\mathcal{S} \cdots \times_\mathcal{S} \mathcal{N}$\\[5pt]
$\mathcal{J}$ & universal centralizer $\mathcal{J}\longrightarrow\mathcal{S}$, where $\mathcal{S}$ is a principal Slodowy slice\\[5pt]
$\mathcal{J}_n$ & kernel of multiplication $\mathcal{J} \times_\mathcal{S} \cdots \times_\mathcal{S} \mathcal{J}\longrightarrow \mathcal{J}$ in the $n$-fold fibred product of $\mathcal{J}$\\[5pt]
$\mathcal{Z}_n^{\circ}$ & geometric quotient of $\mathcal{N}_n$ of by $\mathcal{J}_n$\\[5pt]
$\mathcal{Z}_n$ & affinization $\mathrm{Spec}(\C[\mathcal{N}_n]^{\mathcal{J}_n})$ of $\mathcal{Z}_n^{\circ}$\\[5pt]
\end{longtable}
\end{flushleft}

\bibliographystyle{acm} 
\bibliography{symplectic-reduction-along-a-submanifold}

\end{document}